\documentclass[12pt]{article}
\usepackage[cp1251]{inputenc}
\usepackage[T1]{fontenc}
\usepackage[english]{babel}
\usepackage{amssymb,amsmath,amsthm}
\usepackage[nodisplayskipstretch]{setspace}
\usepackage{enumitem}
\usepackage{geometry}
\usepackage{titlesec}
\usepackage[square,numbers,sort&compress]{natbib}
\usepackage[colorlinks,linkcolor=blue,pagecolor=blue, citecolor=blue,urlcolor=black]{hyperref}
\usepackage{hypcap,breakurl}
\usepackage{verbatim}
\usepackage{euscript}

\titlelabel{\thetitle.\quad} 

\allowdisplaybreaks[4] 

\numberwithin{equation}{section}

\theoremstyle{definition}
 \newtheorem{definition}{Definition}[section]
 \newtheorem{remark}{Remark}[section]

\theoremstyle{plain}
 \newtheorem{theorem}{Theorem}[section]
 \newtheorem{lemma}{Lemma}[section]
 \newtheorem*{lemma*}{Lemma}
 \newtheorem{corollary}{Corollary}[section]
 \newtheorem{proposition}{Proposition}[section]

\begin{document}
\newcommand{\Es}{\EuScript}
\newcommand{\scrscr}{\scriptscriptstyle}
\newcommand{\diag}{\mathop{\mathrm{diag}}}
\newcommand{\mE}{{\mathcal{E}}}
\newcommand{\mV}{{\mathcal{V}}}

\newcommand{\R}{{\mathbb R}}
\newcommand{\N}{{\mathbb N}}
\newcommand{\mC}{{\mathbb C}}
\newcommand{\Z}{{\mathbb Z}}
\newcommand{\Rn}{{{\mathbb R}^n}}
\newcommand{\Rm}{{{\mathbb R}^m}}
\newcommand{\Cn}{{{\mathbb C}^n}}
\newcommand{\Cm}{{{\mathbb C}^m}}

\newcommand{\T}{{\mathop{\mathrm{T}}}}
\newcommand{\tr}{\mathop{\mathrm{tr}}\nolimits}
\newcommand{\sgn}{\mathop{\mathrm{sgn}}}
\newcommand{\Ker}{\mathop{\mathrm{Ker}}}
\newcommand{\rank}{\mathop{\mathrm{rank}}}
\newcommand{\res}{\mathop{\mathrm{Res}}}
\newcommand{\ee}{\mathrm{e}}
\newcommand{\ii}{\mathrm{i}}
\newcommand{\dd}{\textup{d}}
\newcommand{\la}{\lambda}
\renewcommand{\Im}{\operatorname{Im}}
\renewcommand{\Re}{\operatorname{Re}}
\newcommand{\ord}{\mathrm{O}}
\newcommand{\grad}{\mathop{\mathrm{grad}}}

\title{\textbf{Qualitative analysis of nonregular differential-algebraic equations and the dynamics of gas networks}}
\author{\large \textbf{Maria Filipkovska}  \\
 \it\small  Friedrich-Alexander-Universit\"{a}t Erlangen-N\"{u}rnberg,\; Cauerstrasse 11, 91058 Erlangen, Germany\\
 \it\small  B. Verkin Institute for Low Temperature Physics and Engineering \\
 \it\small  of the National Academy of Sciences of Ukraine,\; Nauky Ave. 47, 61103 Kharkiv, Ukraine\\
 \small maria.filipkovska@fau.de;\; filipkovskaya@ilt.kharkov.ua}

\date{ }
\maketitle

\begin{abstract}
  The conditions for the existence, uniqueness and boundedness of global solutions, as well as ultimate boundedness of solutions, and the conditions for the blow-up of solutions of nonregular semilinear differential-algebraic equations are obtained. An example demonstrating the application of the obtained results is considered. Isothermal models of gas networks are proposed as applications
\end{abstract}



 \section{Introduction}

This paper deals with systems of equations which can be represented as a differential-algebraic equation (DAE) of the form
\begin{equation}\label{IntroDAE}
\frac{d}{dt}[Ax]+Bx=f(t,x),
 \end{equation}
where $A$, $B$ are linear operators from $\Rn$ into $\Rm$ or $m\times n$-matrices. Various systems consisting of ordinary differential equations (ODEs) \,(or partial differential equations (PDEs), which after applying spatial discretization become ODEs) and of algebraic equations (not containing a derivative) can be written in this form. Note that this type of DAEs includes underdetermined and overdetermined systems of equations. DAEs of the form \eqref{IntroDAE} are commonly referred to as nonregular (or singular) semilinear  DAEs. In general, they belong to the class of ODEs unsolved for the higher derivative of the unknown function and are also called descriptor systems or degenerate differential equations.

In the present paper, the conditions for the existence, uniqueness and boundedness of global solutions, as well as ultimate boundedness of solutions, and the conditions for the blow-up of solutions of nonregular semilinear DAEs are obtained both in the general form (Sections \ref{SectGlobSolv}--\ref{SectWeakerCond}) and in certain particular cases (Section \ref{Sect_chi_V}) which are convenient for practical application.

DAEs arise from the modelling of various systems and processes in control problems, gas industry, mechanics, radio engineering, chemical kinetics, economics and other fields (see, e.g., \cite{GNM,Riaza,BGHHST,Kunkel_Mehrmann,Vlasenko1} and references therein). It is well known that DAEs arise in electrical circuit theory (the use of
DAEs in electrical circuit modelling is described in detail in \cite{Riaza}; see also \cite{Kunkel_Mehrmann,Vlasenko1,Fil.UMJ,Fil,Fil.MPhAG,Fil.sing,Rut,rut-sing,RF1}).
Besides electrical networks, DAEs are also used in modelling other objects whose structure is described by directed graphs, e.g., gas and neural networks. In the papers \cite{Fil.sing,Fil.UMJ}, nonlinear electrical circuits described by singular (nonregular) semilinear DAEs were considered.  The present paper is focused on the DAEs describing the dynamics of gas networks in the isothermal case (Section \ref{Sect_GasNetDAE}). The theorems obtained in this paper allow one to carry out the qualitative analysis of the dynamics of gas networks described by DAEs of the form \eqref{IntroDAE}. The description of gas network models, including the construction of models in the form of DAEs, is presented in \cite{GNM,A-PerJ,BandaHerty,BGHHST,GU2018,Huc18,KSSTW22,MartinMM05}. Generally, the dynamics of a gas flow in a pipeline (for a single pipe) is modelled by PDEs, namely, by the isothermal Euler equations in the case considered in Section \ref{Sect_GasDAEsingle}, and by the equation of state for a gas, which is an algebraic equation. We apply the spatial discretization (described in \cite{Huc18,BGHHST}) for the isothermal Euler equations, which leads to a semilinear DAE. A similar discretization is used to obtain a DAE which describes the dynamics of flows in gas networks (Section \ref{Sect_GasDAEnetwork}) and arises from a system of differential and algebraic equations, which is presented in \cite{KSSTW22}.  The work \cite{KSSTW22} provides and studies the least-squares collocation method for the numerical solution of DAEs arising from gas network modelling.

Most of the works on DAEs are related to the study of regular DAEs: their structure, index, local solvability, the Lyapunov stability of their equilibrium positions and the development of numerical methods for solving them. Due to the large number of these works, it is not possible to list all of them, therefore, as an example, we refer to monographs \cite{Brenan-C-P,Kunkel_Mehrmann,Riaza,Vlasenko1} where the list of references can be found. Much fewer works deal with nonregular DAEs in general and with the global solvability of DAEs in particular. Nonregular DAEs have been studied by using the concept of the ``strangeness index'' in \cite{Kunkel_Mehrmann}. This concept generalizes the ``differentiation index'' (see \cite{Brenan-C-P}) to underdetermined and overdetermined systems of equations. We use the concept of an index only for a regular block of the characteristic pencil of the DAE \eqref{IntroDAE} (see Section \ref{Preliminaries}). To solve a singular linear time-invariant DAE, one usually uses the Weierstrass-Kronecker canonical form of a singular matrix pencil associated with the linear DAE, which is described in \cite{GantmaherII}. The solvability of nonregular time-varying linear DAEs with the use of a generalized canonical form and the application of the least squares method for their numerical solution has been studied in \cite{ChistyakovCh11}.  In the paper \cite{Chuiko18} the conditions for the solvability of the Cauchy problem for a nonregular time-varying linear DAE with the use of a generalized Green operator and the structure of the generalized Green operator of the Cauchy problem have been found. The conditions for the Lagrange stability and instability of nonregular semilinear DAE, which are a particular case of the conditions obtained in this paper, have been presented in \cite{Fil.UMJ}.  The local solvability of singular (nonregular) semilinear DAEs in Banach spaces has been studied in \cite{rut-sing}. Also, in \cite{rut-sing} the decomposition of a singular pencil into regular and purely singular components, which was called the RS-splitting of the pencil, and the block representations of the singular component in two special cases were presented.

In this paper, we use the special block form of a singular operator pencil \cite{Fil.sing,Fil.KNU2019}, which consists of the singular and regular blocks where zero and invertible blocks are separated out (see Section \ref{BlockStruct}). This block form is used to reduce the DAE with the singular characteristic pencil to a system of ordinary differential and algebraic equations (see Section \ref{DAEsingReduce}). The results from \cite{GantmaherII} were used when constructing this block form in \cite{Fil.sing,Fil.KNU2019}. The method for constructing the direct decompositions of spaces, which reduce the pencil to the required block form, and the corresponding projectors, is briefly described in \cite{Fil.sing} and is described in detail in \cite{Fil.KNU2019} (where results from \cite{Fil.sing} were generalized). In addition to the mentioned block form, differential inequalities (see Section \ref{LaSDiffIneq}) proposed in \cite{LaSal-Lef} when studying the continuation of solutions and Lagrange stability of ODEs, the Lyapunov type functions, the spectral projector introduced in \cite{Rut} and the method of T. Yoshizawa \cite{Yoshizawa}, as well as some other methods, are used to prove theorems from Sections \ref{SectGlobSolv}--\ref{SectWeakerCond} and the results presented in Section \ref{Sect_chi_V}.  An example demonstrating the application of the obtained results is given in Section \ref{ExGener}.

In the present paper, we use the following notations: $I_X$ is an identity operator in the space $X$; $\EuScript A^{(-1)}$ is a semi-inverse operator of the operator $\EuScript A$ \,($A^{-1}$ is an inverse operator of $A$); $\Ker(A)$ is the kernel (the null-space) of an operator $A$;
$\mathcal R(A)$ is the range of an operator $A$; $D^c$ is the complement of a set $D$; $\overline{D}$ is the closure of a set $D$;
$\mathrm{L}(X,Y)$ is the space of continuous linear operators from $X$ to $Y$; $\mathrm{L}(X,X)=\mathrm{L}(X)$, and similarly, $C((a,b),(a,b))=C(a,b)$;
$\delta_{ij}$ is the Kronecker delta; $X'$ is the conjugate space of $X$ (it is also called an adjoint or dual space); 
$A^\T$ is a transposed operator (i.e., the adjoint operator acting in real linear spaces to which the transposed matrix correspond) or a transposed matrix;  $\|\cdot\|$ denotes some norm in a finite-dimensional space (it will be clear from the context in which one), unless it is explicitly stated which norm is considered.
Also, a partial derivative with respect to $x$ will sometimes be denoted by $\partial_x:=\partial/\partial x$.

Note that if $t\in [a,b]$, $a,b\in \R$, $a\ne b$, then by a neighborhood $U_\delta(a)$ (a neighborhood $U_\delta(b)$) we mean a semi-open interval $[a,a+\delta)$ (a semi-open interval $(b-\delta,b]$), $0<\delta<b-a$,  instead of an open interval (in general, a neighborhood of a point $x$ is an open set containing $x$). In what follows, we will call $U_\delta(x_0)$ a neighborhood of the point $x_0\in X$ and denote by $U_\delta(x_0)$ the open ball $\{x\in X\mid \|x-x_0\|<\delta\}$, $\delta>0$,  except for the case mentioned above and the case when it is explicitly indicated that $U_\delta(x_0)$ is a closed neighborhood, that is, $U_\delta(x_0)=\{x\in X\mid \|x-x_0\|\le\delta\}$. %$\delta>0$, $\delta\ge 0$.

 \section{Problem statement, definitions and preliminary constructions}\label{Preliminaries}

Consider an implicit differential equation
 \begin{equation}\label{DAE}
\frac{d}{dt}[Ax]+Bx=f(t,x),
 \end{equation} 
where $t\in [t_+,\infty)$, $t_+\ge 0$, $x\in\Rn$, $A,\, B\in \mathrm{L}(\Rn,\Rm)$ and $f\in C([t_+,\infty)\times \Rn,\Rm)$. 
In the case when $m\ne n$ or $m=n$ and the operator $A$ is noninvertible (degenerate), the equation \eqref{DAE} is called a \emph{differential-algebraic equation} (\emph{DAE}) or \emph{degenerate differential equation}. In the DAE terminology, equations of the form \eqref{DAE} are called \emph{semilinear}. For the considered equation, the initial condition (Cauchy condition) is given in the form
 \begin{equation}\label{ini}
x(t_0)=x_0\qquad  (t_0\ge t_+).
 \end{equation}

A DAE that contains the linear part $\frac{d}{dt}[Ax]+Bx$ such that the pencil $\lambda A+B$ is singular (see Definition~\ref{block-def.1}) is called  \emph{singular} or \emph{nonregular} (or \emph{irregular} \cite{Fil.UMJ}). The pencil $\lambda A+B$ corresponding to this linear part is called \emph{characteristic}.

If $\rank(\lambda A+B)=m<n$, then the DAE \eqref{DAE} corresponds to an underdetermined system of equations (that is, the number of equations is less than the number of unknowns).

If $\rank(\lambda A+B)=n<m$, then DAE \eqref{DAE} corresponds to an overdetermined system of equations (that is, the number of equations is greater than the number of unknowns).

  \smallskip
The function $x(t)$ is called a \emph{solution of the equation \eqref{DAE} on $[t_0,t_1)$}, $t_1\le \infty $, if $x(t)\in C([t_0,t_1 ),\Rn)$, $(Ax)(t) \in C^1([t_0,t_1),\Rm)$ and $x(t)$ satisfies \eqref{DAE} on $[t_0,t_1)$. If the function $x(t)$ additionally satisfies the initial condition \eqref{ini}, then it is called a \emph{solution of the initial value problem} (\emph{IVP}) or \emph{the Cauchy problem} \eqref{DAE}, \eqref{ini}.

A solution $x(t)$ (of an equation or inequality) is called \emph{global} if it exists on the entire interval $[t_0,\infty)$ (where $t_0$ is an initial value). %(time)).

A solution $x(t)$ is called \emph{Lagrange stable} if it is global and bounded, i.e., $x(t)$ exists on $[t_0,\infty)$ and $\sup\limits_{t\in [t_0,\infty)} \|x(t)\|<\infty$.

A solution $x(t)$ \emph{has a finite escape time} (or \emph{is blow-up in finite time}) and is called \emph{Lagrange unstable} if it exists on some finite interval $[t_0,\tau)$ and is unbounded, i.e., there exists $\tau>t_0$ ($\tau<\infty$) such that $\lim\limits_{t\to \tau-0} \|x(t)\|=\infty$.

 \subsection{Remarks on differential inequalities}\label{LaSDiffIneq}

Here we give brief information about the existence of positive solutions (of different types) for differential inequalities which will be used below.

Consider two differential inequalities:
\begin{equation}\label{L1v} 
\dfrac{dv}{dt}\le \chi(t,v),
\end{equation}
\begin{equation}\label{L2v} 
\dfrac{dv}{dt}\ge \chi(t,v),
\end{equation}
where $t\in [t_+,\infty)$ and $\chi\in C([t_+,\infty)\times (0,\infty),\R)$. 
A positive scalar function $v\in C^1([t_0,\infty),(0,\infty))$   satisfying the differential inequality (one of the mentioned inequalities) is called a \emph{positive solution} of this inequality on $[t_0,\infty)$ ($t_0\ge t_+$).

Let
 \begin{equation}\label{La-kU}
\chi(t,v)=k(t)\,U(v),
 \end{equation}
where $k\in C([t_+,\infty),\R)$ and $U\in C(0,\infty)$ (that is, $U\in C((0,\infty),\R)$ is a positive function), then the inequalities~\eqref{L1v} and \eqref{L2v} take the form
\begin{equation}\label{La-S.3}
\dfrac{dv}{dt}\le k(t)\,U(v),
\end{equation}
\begin{equation}\label{La-S.4}
\dfrac{dv}{dt}\ge k(t)\,U(v),
\end{equation}
respectively, and the following statements hold (see, e.g., \cite{LaSal-Lef}):
 \begin{itemize}
 \item  if $\int\limits_c^{\infty}U^{-1}(v)\, dv=\infty $  ($c>0$ is some constant), then the inequality~\eqref{La-S.3} does not have positive solutions with finite escape time;
 \item if $\int\limits_c^{\infty}U^{-1}(v)\, dv=\infty$ and $\int\limits_{t_0}^{\infty }k(t)dt<\infty $ ($t_0\ge t_+$ is some number), then the inequality~\eqref{La-S.3} does not have unbounded positive solutions for $t\in [t_+,\infty)$;
 \item  if $\int\limits_c^{\infty}U^{-1}(v)\, dv<\infty$ and $\int\limits_{t_0}^{\infty}k(t)dt=\infty$, then the inequality \eqref{La-S.4} does not have global (i.e., defined on $[t_+,\infty)$) positive solutions.
 \end{itemize}

 \subsection{Block form of a singular pencil, the corresponding direct decompositions of spaces and projectors}\label{BlockStruct}

Let $A$, $B$ be linear operators mapping $\Rn$ into $\Rm$ or $\Cn$ into $\Cm$; by $A$, $B$ we also denote $m\times n$-matrices corresponding to the operators $A$, $B$  (with respect to some bases in $\Rn$, $\Rm$ or $\Cn$, $\Cm$ respectively).  Consider the operator pencil $\lambda A+B$, where $\lambda$ is a complex parameter.

The \emph{rank of an operator pencil} $\lambda A+B$ is the dimension of its range.  The \emph{rank of a matrix pencil} $\lambda A+B$ is the largest among the orders of the pencil minors that do not vanish identically  \cite{GantmaherII}.  The rank of the matrix pencil equals the maximum number of its linearly independent columns (or rows), i.e.,  the maximum number of columns (or rows) of the pencil that are linearly independent set of vectors for some $\lambda=\lambda_0$. It is clear that the rank of the operator pencil and the rank of the corresponding matrix pencil coincide.

\begin{definition}\label{block-def.1}
A pencil of operators (or matrices) $\lambda A+B$ is called \emph{regular} if $n=m=\rank(\lambda A+B)$; otherwise, i.e.,  if  $n\ne m$ or $n=m$ and  $\rank(\lambda A+B)<n$, the pencil is called \emph{singular} or \emph{nonregular} (\emph{irregular}).
\end{definition}
For $m\times n$-matrices  $A$, $B$ corresponding to the operators $A$, $B$, this definition is equivalent to that given in \cite{GantmaherII}, namely, the pencil $\lambda A+B$ is called \emph{regular} if  $n=m$ and $\det(\lambda A+B)\not \equiv 0$, and \emph{singular} otherwise ($n\ne m$  or $n=m$ and $\det(\lambda A+B)\equiv 0$).

Note that Definition  \ref{block-def.1} is also equivalent to the following. An operator pencil $\lambda A+B\colon\Cn \to \Cm$ is called \emph{regular} if the set of its regular points  $\rho(A,B)= \{\lambda \in \mC \mid  (\lambda A+B)^{-1} \in \mathrm{L}(\Cm,\Cn)\}$ is not empty, and \emph{singular} if $\rho(A,B) = \emptyset$. 
For the real operators $A,\, B\colon\Rn\to \Rm$ this definition takes the following form. First, we introduce the complex extensions $\hat{A},\, \hat{B}$  of operators $A$, $B$, which map $\Cn$ into $\Cm$.
 %% (see, e.g., \cite{Halmos}).
Recall that the ranks of the pencil $\lambda A+B$ and its complex extension $\lambda \hat{A}+\hat{B}$ coincide, and that the matrices of the operators $A$, $B$ with respect to some bases in $\Rn$, $\Rm$ coincide with the matrices of their complex extensions $\hat{A}$, $\hat{B}$ with respect to the same bases in $\Cn$, $\Cm$.
A pencil $\lambda A+B$ of the operators $A,\, B\colon\Rn\to \Rm$  is called  \emph{regular} if the set of regular points $\rho(\hat{A},\hat{B})$ of its complex extension $\lambda\hat{A}+\hat{B}\in \mathrm{L}(\Cn,\Cm)$ is not empty, and \emph{singular} if $\rho(\hat{A},\hat{B})=\emptyset$.
The regular points $\lambda$ of the complex extension $\lambda\hat{A}+\hat{B}$ are called \emph{regular points} of the pencil $\lambda A+B$ (since for these points there exists the resolvent $(\lambda A+B)^{-1}$).

 \smallskip
The results from \cite{Fil.sing}, \cite{Fil.KNU2019} which will be used hereinafter are given below. The detailed description and proof of these results  can be found in \cite{Fil.KNU2019}.

In what follows, we will consider linear operators $A,\, B\colon \Rn\to\Rm$.

Note that instead of the real operators we can consider the complex operators  $A,\, B\colon \Cn \to \Cm$, for which Proposition~\ref{STABssr} (see below) remains true, but when constructing direct decompositions of the form $\eqref{ssr}$ for the complex spaces $\Cn$ and $\Cm$ and the corresponding projectors, it is necessary to replace transposition by Hermitian conjugation everywhere. 

Let $A\colon X\to Y$ be a linear operator and $X_0$, $Y_0$ be some subspaces in $X$, $Y$ respectively. The pair of subspaces $\{X_0,Y_0\}$ is said to be \emph{invariant} under the operator $A$ if $A\colon X_0\to Y_0$, i.e., $A X_0 \subseteqq Y_0$ (cf. \cite{rut-sing}; in the case when $X=Y$ and $X_0=Y_0$, this is the classical definition of invariance \cite{Halmos}).

Recall the following classic definition:  A linear space $L$ is decomposed into the \emph{direct sum} $L =L_1\dot + L_2$ of the subspaces $L_1\subseteq L$ and $L_2\subseteq L$ if $L_1\cap L_2 =\{0\}$ and $L_1+L_2=\{x_1+x_2 \mid x_1\in L_1, x_2\in L_2\}=L$, or, equivalently, if every $x\in L$ can be uniquely represented in the form $x=x_1+x_2$ where $x_i\in L_i$, $i=1,2$. The representation $L=L_1\dot + L_2$ is also called a \emph{direct decomposition} of the space $L$.

Since the direct (Cartesian) product $L_1\times L_2$ is the direct sum of the spaces $L_1\times \{0\}$ and $\{0\}\times L_2$, where $0$ from $L_2$ and $L_1$ respectively, then it can be identified with the direct sum $L_1\dot + L_2$ by identifying $L_1\times \{0\}$ with $L_1$ and $\{0\}\times L_2$ with $L_2$.

\emph{Thus, below, when indicating the block structures of operators, we identify direct sums and the corresponding direct products of subspaces for convenience of notation.}

 \begin{proposition}[see \cite{Fil.sing,Fil.KNU2019}]\label{STABssr}
For the operators $A,\, B\colon\Rn\to\Rm$, which form a singular pencil $\lambda A+B$, there exist the decompositions of the spaces $\Rn$, $\Rm$ into the direct sums of subspaces (which can always be construct)
 \begin{equation}\label{ssr}
\Rn=X_s\dot +X_r=X_{s_1}\dot+X_{s_2}\dot+X_r,\quad  \Rm=Y_s\dot+Y_r=Y_{s_1}\dot+Y_{s_2}\dot+Y_r,
 \end{equation} 
with respect to which  $A$, $B$ have the block structure  
 \begin{equation}\label{srAB}
A=\begin{pmatrix} %%= A_s\dot +A_r
   A_s & 0   \\
   0   & A_r \end{pmatrix},\;
B=\begin{pmatrix} %%= B_s\dot +B_r
   B_s & 0   \\
   0   & B_r \end{pmatrix}\colon X_s\dot +X_r\to Y_s\dot +Y_r\quad  (X_s\times X_r\to Y_s\times Y_r),
 \end{equation}
where $A_s=A\big|_{X_s},\, B_s=B\big|_{X_s}\colon X_s\to Y_s$ and $A_r=A\big|_{X_r},\, B_r=B\big|_{X_r}\colon X_r\to Y_r$, 
i.e., the pair of singular subspaces $\{X_s,Y_s\}$ and the pair of regular subspaces  $\{X_r,Y_r\}$  are invariant under the operators  $A$, $B$\,  ($A,\, B\colon X_s\to Y_s$ and $A,\, B\colon X_r\to Y_r$), and their singular blocks $A_s$, $B_s$ have the structure
 \begin{equation}\label{sab}
A_s = \begin{pmatrix} A_{gen} & 0 \\ 0 & 0
     \end{pmatrix},\;
B_s = \begin{pmatrix} B_{gen} & B_{und} \\ B_{ov} & 0
     \end{pmatrix}\colon X_{s_1}\dot + X_{s_2} \to Y_{s_1}\dot + Y_{s_2}\quad (X_{s_1}\times X_{s_2} \to Y_{s_1}\times Y_{s_2}),
 \end{equation}
where the operator ${A_{gen}\colon X_{s_1}\to Y_{s_1}}$ has the inverse ${A_{gen}^{-1}\in \mathrm{L}(Y_{s_1},X_{s_1})}$ (if ${X_{s_1}\ne \{0\}}$), ${B_{gen}\colon X_{s_1}\to Y_{s_1}}$, ${B_{und}\colon X_{s_2}\to Y_{s_1}}$, ${B_{ov}\colon X_{s_1}\to Y_{s_2}}$, at that, if ${\rank(\lambda A+B)=m<n}$, then the structure of the singular blocks takes the form %\eqref{sab1}
  \begin{equation}\label{sab1}
A_s = \begin{pmatrix} A_{gen} & 0 \end{pmatrix},\;
B_s = \begin{pmatrix} B_{gen} & B_{und} \end{pmatrix}\colon X_{s_1}\dot +  X_{s_2}\to Y_s\quad (X_{s_1}\times  X_{s_2}\to Y_s)
 \end{equation}
and $Y_{s_1}=Y_s$, $Y_{s_2}= \{0\}$ in the decompositions \eqref{ssr}, and if ${\rank(\lambda A+B)=n<m}$, then the structure of singular blocks takes the form 
 \begin{equation}\label{sab2}
A_s = \begin{pmatrix} A_{gen} \\ 0
  \end{pmatrix},\;
B_s = \begin{pmatrix} B_{gen} \\ B_{ov}
 \end{pmatrix}\colon X_s\to Y_{s_1}\dot+ Y_{s_2}\quad (X_s\to Y_{s_1}\times Y_{s_2})
\end{equation}
and $X_{s_1}=X_s$, $X_{s_2}=\{0\}$ in the decompositions  \eqref{ssr}. 
The direct decompositions of spaces \eqref{ssr} generate the pair $S$, $P$, the pair $F$, $Q$, the pair $S_1$, $S_2$ and the pair $F_1$, $F_2$  of the mutually complementary projectors (${S+P=I_{\Rn}}$, ${S^2=S}$, $P^2=P$, $SP=PS=0$;\,  $F+Q=I_{\Rm}$, $F^2=F$, $Q^2=Q$, $FQ=QF=0$;\,  $S_1+S_2=S$, $S_iS_j=\delta_{ij}S_i$;\,  $F_1+F_2=F$, $F_iF_j=\delta_{ij}F_i$)
\begin{equation}\label{ProjRS}
S \colon \Rn \to X_s,\; P \colon \Rn\to X_r,\qquad F \colon \Rm \to Y_s,\; Q \colon \Rm \to Y_r,
\end{equation}
\begin{equation}\label{ProjS}
S_i\colon \Rn \to X_{s_i},\qquad F_i\colon \Rm \to Y_{s_i},\quad i=1, 2,
\end{equation}
where $F_1=F$, $F_2=0$ if $\rank(\lambda A+B)= m<n$, and $S_1=S$, $S_2=0$ if $\rank(\lambda A+B)= n<m$, which have the properties
 \begin{equation}\label{ProjRSInvar}
FA=AS,\quad FB=BS,\qquad  QA =AP,\quad QB=BP,
 \end{equation}
 \begin{equation}\label{properProjS}
{A S_2 = 0},\quad {F_2 A = 0},\quad {F_2 B S_2=0}.
 \end{equation}
 \end{proposition}
The converse assertion (regarding the one given in Proposition \ref{STABssr}) that there exist the pairs of mutually complementary projectors \eqref{ProjRS}, \eqref{ProjS} (with the properties \eqref{ProjRSInvar}, \eqref{properProjS})  which generate the direct decompositions of spaces  \eqref{ssr}  is also true.

The method for constructing the subspaces from the decompositions \eqref{ssr} and the corresponding projectors \eqref{ProjRS}, \eqref{ProjS} is described in \cite{Fil.KNU2019,Fil.UMJ}.

 \smallskip
It follows from Proposition \ref{STABssr} that, with respect to the decompositions \eqref{ssr}, the singular pencil $\lambda A+B$ of the operators $A,\, B\colon \Rn \to \Rm$ takes the block form
 \begin{equation}\label{penc}
\lambda A+B =  \begin{pmatrix}
 \lambda A_s+B_s & 0 \\
 0 &  \lambda A_r+B_r \end{pmatrix},\quad
  A_s,\, B_s\colon X_s\to Y_s,\quad A_r,\, B_r\colon X_r\to Y_r,
 \end{equation}
where the regular block $\lambda A_r+ B_r$ is a regular pencil and the singular block $\lambda A_s+ B_s$ is a purely singular pencil, i.e., it is impossible to separate out a regular block in this pencil.

In \cite{Fil.KNU2019}, extensions of the operators from the block representations \eqref{srAB}, \eqref{sab}, \eqref{sab1}, \eqref{sab2} to $\Rn$ and corresponding semi-inverse operators were introduced.  These operators are described below and are used in Section~\ref{DAEsingReduce} and in subsequent sections.

Extensions of the operators $A_s$, $A_r$, $B_s$, $B_r$ from \eqref{srAB} to $\Rn$ are introduced as follows:
\begin{equation}\label{AsrBsrExtend}
 \EuScript A_s = F A,\quad \EuScript A_r= Q A,\quad \EuScript B_s= F B,\quad \EuScript B_r= Q B,
\end{equation}
$\EuScript A_s, \EuScript B_s, \EuScript A_r, \EuScript B_r\in \mathrm{L}(\Rn,\Rm)$. Then
\begin{equation}\label{AsrBsr}
\EuScript A_s\big|_{X_s}=A_s,\quad \EuScript A_r\big|_{X_r}=A_r,\quad \EuScript B_s\big|_{X_s}=B_s,\quad \EuScript B_r\big|_{X_r}=B_r,
\end{equation}
and the operators \eqref{AsrBsrExtend} act so that ${\EuScript A_s,\, \EuScript B_s\colon \Rn\to Y_s}$,  ${\EuScript A_r,\, \EuScript B_r\colon \Rn\to Y_r}$ \,($\EuScript A_s, \EuScript B_s\colon {X_s\to Y_s}$,  $\EuScript A_r, \EuScript B_r\colon X_r\to Y_r$) and $X_r\subset \Ker(\EuScript A_s)$, $X_r\subset \Ker(\EuScript B_s)$, $X_s\subset \Ker(\EuScript A_r)$, ${X_s\subset \Ker(\EuScript B_r)}$.

 \smallskip
In the general case, when $\rank(\lambda A+B) < n$ and $\rank(\lambda A+B) < m$, the spaces $\Rn$, $\Rm$ have the direct decompositions \eqref{ssr} and, accordingly, the singular subspaces are decomposed into the direct sums of subspaces
\begin{equation*}
X_s = X_{s_1}\dot +  X_{s_2},\quad Y_s = Y_{s_1}\dot +  Y_{s_2}
\end{equation*}
with respect to which the singular blocks $A_s$, $B_s$ have the structure \eqref{sab}, and extensions of the operators (blocks) from \eqref{sab} to $\Rn$ are introduced as follows:
 \begin{equation}\label{ABssExtend}
\EuScript A_{gen}= F_1 A,\quad \EuScript B_{gen}=F_1 B S_1,\quad
\EuScript B_{und}=F_1 B S_2,\quad \EuScript B_{ov}=F_2 B S_1,
 \end{equation} 
$\EuScript A_{gen},\, \EuScript B_{gen},\, \EuScript B_{und},\, \EuScript B_{ov}\in \mathrm{L}(\Rn,\Rm)$ \,(notice, that $F_1 A=AS_1=FA$).  Then
\begin{equation}\label{ABss}
 \EuScript A_{gen}\big|_{X_{s_1}}=A_{gen},\;\;
 \EuScript B_{gen}\big|_{X_{s_1}}=B_{gen},\;\;
 \EuScript B_{und}\big|_{X_{s_2}}=B_{und},\;\;
 \EuScript B_{ov}\big|_{X_{s_1}}=B_{ov},
\end{equation}
and the operators \eqref{ABssExtend} act so that ${\EuScript A_{gen}\Rn = \EuScript A_{gen}X_{s_1}= Y_{s_1}}$ \,($X_{s_2}\dot +X_r= \Ker(\EuScript A_{gen})$),\, ${\EuScript B_{gen}\colon \Rn\to Y_{s_1}}$, $X_{s_2}\dot +X_r\subset \Ker(\EuScript B_{gen})$,\, ${\EuScript B_{und}\colon \Rn\to Y_{s_1}}$, $X_{s_1}\dot + X_r\subset \Ker(\EuScript B_{und})$,\, and ${\EuScript B_{ov}\colon \Rn\to Y_{s_2}}$, $X_{s_2}\dot +X_r\subset \Ker(\EuScript B_{ov})$.

 \smallskip 
In the case when $\rank(\lambda A+B) = m < n$, the singular subspace $X_s$ is  decomposed into the direct sum of subspaces
\begin{equation*}%\label{sUnd}
X_s = X_{s_1}\dot +  X_{s_2}
\end{equation*}
with respect to which the singular blocks $A_s$, $B_s$ have the structure \eqref{sab1}, and extensions of the operators (blocks) from \eqref{sab1} to $\Rn$ are introduced as follows:
\begin{equation}\label{ABsab1Extend} % B S_i = F B S_i,  i=1,2
\EuScript A_{gen}=AS_1,\quad \EuScript B_{gen}=B S_1,\quad \EuScript B_{und}=B S_2,
\end{equation} 
$\EuScript A_{gen},\, \EuScript B_{gen},\, \EuScript B_{und}\in \mathrm{L}(\Rn,\Rm)$. Then
\begin{equation}\label{ABsab1}
\EuScript A_{gen}\big|_{X_{s_1}}=A_{gen},\quad \EuScript B_{gen}\big|_{X_{s_1}}=B_{gen},\quad \EuScript B_{und}\big|_{X_{s_2}}=B_{und},
\end{equation}
and the operators \eqref{ABsab1Extend} act so that ${\EuScript A_{gen}\Rn =\EuScript A_{gen}X_{s_1}=Y_s}$ \,(${X_{s_2}\dot + X_r= \Ker(\EuScript A_{gen})}$),\, $\EuScript B_{gen}\colon \Rn\to Y_s$, ${X_{s_2}\dot + X_r\subset \Ker(\EuScript B_{gen})}$, and ${\EuScript B_{und}\colon \Rn\to Y_s}$, ${X_{s_1}\dot +X_r\subset \Ker(\EuScript B_{und})}$.

  \smallskip
In the case when $\rank(\lambda A+B)=n<m$, the singular subspace $Y_s$ is  decomposed into the direct sum of subspaces
\begin{equation*}%\label{sOver}
Y_s= Y_{s_1}\dot +  Y_{s_2}
\end{equation*}
with respect to which the singular blocks $A_s$, $B_s$ have the structure \eqref{sab2}, and extensions of the operators (blocks) from \eqref{sab2} to $\Rn$ are introduced as follows:
\begin{equation}\label{ABsab2Extend}
\EuScript A_{gen}=F_1 A,\quad \EuScript B_{gen}=F_1 B,\quad \EuScript B_{ov}=F_2 B,
\end{equation}
$\EuScript A_{gen},\, \EuScript B_{gen},\, \EuScript B_{ov}\in \mathrm{L}(\Rn,\Rm)$. Then
\begin{equation}\label{ABsab2}
\EuScript A_{gen}\big|_{X_s}=A_{gen},\quad \EuScript B_{gen}\big|_{X_s}=B_{gen},\quad \EuScript B_{ov}\big|_{X_s}=B_{ov},
\end{equation}
and the operators \eqref{ABsab2Extend} act so that $\EuScript A_{gen}\Rn =\EuScript A_{gen}X_s= Y_{s_1}$ \,($X_r= \Ker(\EuScript A_{gen})$), $\EuScript B_{gen}\colon \Rn\to Y_{s_1}$, $X_r\subset \Ker(\EuScript B_{gen})$, and $\EuScript B_{ov}\colon \Rn\to Y_{s_2}$, $X_r\subset \Ker(\EuScript B_{ov})$.

\begin{remark}\label{Rem_InvAgen}
Note that the extension $\EuScript A_{gen}^{(-1)}\in \mathrm{L}(\Rm,\Rn)$ of the operator $A_{gen}^{-1}$ to $\Rm$ that satisfies the properties
\begin{equation}\label{InvAgen}
\EuScript A_{gen}^{(-1)} \EuScript A_{gen}=S_1,\quad  \EuScript A_{gen}\, \EuScript A_{gen}^{(-1)}=F_1,\quad \EuScript A_{gen}^{(-1)}=S_1 \EuScript A_{gen}^{(-1)},
\end{equation}
where $F_1 = F$ if $\rank(\lambda A+B)= m<n$ and $S_1=S$ if $\rank(\lambda A+B)= n<m$, is a semi-inverse operator of $\EuScript A_{gen}$, i.e., $\EuScript A_{gen}^{(-1)}\Rm =\EuScript A_{gen}^{(-1)}Y_{s_1}=X_{s_1}$ \,($Y_{s_2}\dot +Y_r= \Ker(\EuScript A_{gen}^{(-1)})$) and $A_{gen}^{-1}= \EuScript A_{gen}^{(-1)}\big|_{Y_{s_1}}$ \cite{Fil.KNU2019}  (the definition of a semi-inverse operator can be found in \cite{Faddeev}).
\end{remark}

If $X_r=\{0\}$, $Y_r=\{0\}$, then the regular block $\lambda A_r+ B_r$ is absent and $\lambda A+B=\lambda A_s+ B_s$ is a purely singular pencil.

 \smallskip
Consider a regular pencil $\lambda A_r+ B_r$ of operators $A_r,\, B_r\colon X_r\to Y_r$ acting in finite-dimensional spaces ($\dim X_r =\dim Y_r$).
We assume that either $\lambda=\infty$ is a removable singular point of the resolvent  $(\lambda A_r+B_r)^{-1}$, or the operator $A_r$ is invertible. These conditions are equivalent to the following:  the point $\mu=0$ is either a pole of order 1 of the resolvent $(A_r+ \mu B_r)^{-1}$ or a regular point of the pencil $A_r+ \mu B_r$.
Thus, we assume the following: there exist constants $C_1,\, C_2 >0$ such that
 \begin{equation}\label{index1}
\left\|(\lambda A_r+B_r)^{-1}\right\|\le C_1,\quad   |\lambda|\ge C_2.
 \end{equation}

If $A_r$ is noninvertible and \eqref{index1} holds ($\mu=0$ is a simple pole of the resolvent $(A_r+ \mu B_r)^{-1}$), then $\lambda A_r+ B_r$ is a regular pencil of \emph{index 1}. Note that if $A_r=0$ and there exists $B_r^{-1}$, then $\lambda A_r+ B_r\equiv B_r$ can be considered as a regular pencil of index 1.
If $A_r$ is invertible ($\mu=0$ is a regular point of $A_r+\mu B_r$), then $\lambda A_r+B_r$ is a regular pencil of \emph{index~0}. Thus, if $\lambda A_r+B_r$ is a regular pencil and \eqref{index1} holds, then $\lambda A_r+B_r$ is a regular pencil of \emph{index not higher than 1} (index~0 or~1).

 \begin{remark}
If the regular block $\lambda A_r+ B_r$ from \eqref{penc} is a regular pencil of index not higher than 1 (i.e., satisfies \eqref{index1}), then there exist the pair $\Tilde P_j\colon X_r \to X_j$, $j=1,2$, and the pair $\Tilde Q_j\colon Y_r \to Y_j$, $j=1,2$, of mutually complementary projectors  which generate the direct decompositions
\begin{equation}\label{rr}
  X_r = X_1 \dot +X_2,\quad Y_r = Y_1 \dot +Y_2
\end{equation}
such that $A_r,\, B_r \colon  X_j \to Y_j$, $j=1,2$ (the pairs of subspaces $X_1$, $Y_1$ and $X_2$, $Y_2$ are invariant under $A_r$, $B_r$), i.e.,
\begin{equation}\label{ProjRInv}
\Tilde Q_j A_r=A_r \Tilde P_j,\quad \Tilde Q_j B_r=B_r \Tilde P_j,
\end{equation}
and the restricted operators $A_j =A_r\big|_{X_j}\colon  X_j \to Y_j$,\, $B_j=B_r\big|_{X_j}\colon X_j \to Y_j$, $j=1,2$, are such that  ${A_2=0}$ ($\Tilde Q_2 A_r=0$) and there exist $A_1^{-1} \in \mathrm{L}(Y_1,X_1)$\, (if $X_1\not=\{0\}$) and $B_2^{-1} \in \mathrm{L}(Y_2,X_2)$\, (if $X_2\not=\{0\}$). For a regular pencil of operators, the pairs of projectors with the specified properties were introduced in \cite{Rut}. %%\cite[Следствие 3.4]{Rut}.

With respect to the direct decompositions \eqref{rr} the operators $A_r$, $B_r$ have the block structure
\begin{equation}\label{rab}
  A_r = \begin{pmatrix}  A_1 & 0 \\ 0 & 0
  \end{pmatrix},\;
  B_r = \begin{pmatrix}  B_1 & 0 \\ 0 & B_2
  \end{pmatrix}\colon X_1\dot +X_2 \to Y_1\dot +Y_2\quad (X_1\times X_2 \to Y_1\times Y_2),
\end{equation}
where $A_1$ and $B_2$ are invertible (if $X_1\not=\{0\}$ and $X_2\not=\{0\}$ respectively).
 \end{remark}

Thus, if the regular block $\lambda A_r+ B_r$ is a regular pencil of index not higher than 1, then there exist the direct decompositions of  the regular spaces \eqref{rr} with respect to which $A_r$, $B_r$ have the block structure \eqref{rab}.

Projectors $\Tilde P_j$ and $\Tilde Q_j$ can be calculated by using contour integration \cite[p.~2005]{Rut} or defined by the formulas \cite{Fil.KNU2019}:
 \begin{equation}\label{ProjRes}
\begin{aligned}
&\Tilde P_1 = \mathop{Res }\limits_{\mu =0}\left(\frac{(A_r+ \mu B_r)^{-1} A_r}{\mu} \right),\quad &&  \Tilde Q_1 =\mathop{Res }\limits_{\mu =0}\left(\frac{A_r(A_r+ \mu B_r)^{-1}}{\mu} \right), \\
&\Tilde P_2 =E_{X_r}-\Tilde P_1,\quad && \Tilde Q_2 =E_{Y_r} - \Tilde Q_1.
\end{aligned}
 \end{equation}
In addition, for a regular pencil of index 1 
one can obtain projectors onto the subspaces from the decompositions \eqref{rr} without using the formulas from \cite{Rut} or formulas \eqref{ProjRes} as described in \cite[Remark~3, p.~44--45]{Fil.KNU2019}.

Introduce the extensions $P_j$, $Q_j$ of the projectors $\Tilde P_j$, $\Tilde Q_j$ to $\Rn$, $\Rm$, respectively, so that $X_j=P_j \Rn$, $Y_j=Q_j \Rm$, $j=1,2$ (where $X_j$, $Y_j$ from \eqref{rr}) \cite{Fil.KNU2019}. The extended projectors
\begin{equation}\label{ProjR}
P_j\colon \Rn \to X_j,\quad Q_j\colon \Rm \to Y_j,\quad j=1,2,
\end{equation}
have the properties of the original ones, i.e., $P_i P_j =\delta _{ij} P_i$, $P_1+P_2=P$, $Q_i Q_j =\delta _{ij} Q_i$, $Q_1+Q_2=Q$ (the pairs $P_1$, $P_2$ and $Q_1$, $Q_2$ are two pairs of mutually complementary projectors) and
\begin{equation*}
Q_j A=A P_j,\quad Q_j B=B P_j,\quad Q_2 A=0\quad (Q_1 A=QA).
\end{equation*}
The properties of the operator $A_j=A\big|_{X_j}\colon X_j\to Y_j$ and $B_j = B\big|_{X_j}\colon X_j\to Y_j$, $j=1,2$, are also retained. 
Introduce their extensions to $\Rn$ as follows:
\begin{equation}\label{ABrrExtend}
 \EuScript A_j=Q_j A,\quad \EuScript B_j=Q_j B,\quad j=1,2.
\end{equation}
Then
 \begin{equation}\label{ABrr}
\EuScript A_j\big|_{X_j}=A_j,\quad \EuScript B_j\big|_{X_j}=B_j,\quad  j=1,2,
 \end{equation}
and the operators $\EuScript A_j,\, \EuScript B_j\in \mathrm{L}(\Rn,\Rm)$ \eqref{ABrrExtend}  act so that $\EuScript A_1\Rn =\EuScript A_1 X_1= Y_1$ \,($X_2\dot +X_s= \Ker(\EuScript A_1)$), $\EuScript A_2=0$,  $\EuScript B_1\colon \Rn\to Y_1$, ${X_2\dot +X_s\subset \Ker(\EuScript B_1)}$, and  $\EuScript B_2\Rn =\EuScript B_2 X_2= Y_2$ \,(${X_1\dot +X_s= \Ker(\EuScript B_2)}$).

\begin{remark}[{\cite{Fil.KNU2019}}]
The extension $\EuScript A_1^{(-1)}\in \mathrm{L}(\Rm,\Rn)$ of the operator $A_1^{-1}$ to $\Rm$ that satisfies the properties
\begin{equation}\label{InvA1}
\EuScript A_1^{(-1)} \EuScript A_1=P_1,\quad \EuScript A_1 \EuScript A_1^{(-1)}=Q_1,\quad \EuScript A_1^{(-1)}=P_1 \EuScript A_1^{(-1)},
\end{equation}
is a semi-inverse operator of $\EuScript A_1$, i.e., $\EuScript A_1^{(-1)}\Rm =\EuScript A_1^{(-1)} Y_1 =X_1$ \,($Y_2\dot + Y_s= \Ker(\EuScript A_1^{(-1)})$) and $A_1^{-1}=\EuScript A_1^{(-1)}\big|_{Y_1}$. Obviously, $P_1 \EuScript A_1^{(-1)}=\EuScript A_1^{(-1)}Q_1$. These properties enable one to find the form of $\EuScript A_1^{(-1)}$ (or $A_1^{-1}$), using the form of the projectors.

The semi-inverse operator $\EuScript B_2^{(-1)}\in \mathrm{L}(\Rm,\Rn)$ of $\EuScript B_2$, i.e., $\EuScript B_2^{(-1)}\Rm =\EuScript B_2^{(-1)}Y_2=X_2$ \,($Y_1\dot + Y_s= \Ker(\EuScript B_2^{(-1)})$) and $B_2^{-1}=\EuScript B_2^{(-1)}\big|_{Y_2}$,  can be calculated in a similar way:
\begin{equation*}
\EuScript B_2^{(-1)}\EuScript B_2=P_2,\quad \EuScript B_2\EuScript B_2^{(-1)}=Q_2,\quad \EuScript B_2^{(-1)}=P_2 \EuScript B_2^{(-1)}\quad (P_2 \EuScript B_2^{(-1)}=\EuScript B_2^{(-1)}Q_2).
\end{equation*}
  \end{remark}

The decompositions \eqref{ssr} and \eqref{rr} together give the decomposition of $\R^n$ into the direct sum of subspaces
 \begin{equation}\label{Xssrr}
\Rn=X_s\dot +X_r =X_{s_1}\dot +  X_{s_2}\dot +  X_1 \dot +X_2
 \end{equation} 
with respect to which any element $x\in \Rn$ can be uniquely represented as
 \begin{equation}\label{xsr}
x= x_s+x_r= x_{s_1}+x_{s_2}+x_{p_1}+x_{p_2}\qquad (x_s=x_{s_1}+x_{s_2},\quad x_r=x_{p_1}+x_{p_2}),
 \end{equation}
where $x_s=Sx\in X_s$, \,$x_r=Px\in X_r$, \,$x_{s_i}=S_i x\in X_{s_i}$, \,$x_{p_i}=P_i x\in X_i$, $i=1,2$.

In what follows, it is assumed that the specified correspondence between the subscript of an element from the subspace present in the decomposition \eqref{Xssrr}  (or a component from the representation \eqref{xsr})   and the subspace to which this element belongs is always fulfilled, i.e., the element $x_{s_i}$ ($i=1,2$) belongs to $X_{s_i}$ because it has the subscript ${s_i}$ ($i=1,2$), the element $x_{p_j}$ belongs to $X_j$ ($j=1,2$), and so on. Thus, we will not always explicitly indicate belonging to one of the subspaces introduced in \eqref{Xssrr}, when the element has one of the subscripts given in \eqref{xsr}, can be any element from the corresponding subspace, and it is clear from the context what exactly is meant.

Similarly, the decompositions \eqref{ssr} and \eqref{rr} together also give the decomposition of $\Rm$ into the direct sum of subspaces
\begin{equation}\label{Yssrr}
 \Rm=Y_s\dot +Y_r =Y_{s_1}\dot +  Y_{s_2}\dot + Y_1 \dot +Y_2,
\end{equation}
with respect to which any element $y\in \Rm$ can be uniquely represented as
 \begin{equation}\label{ysr}
y= y_s+y_r= y_{s_1}+y_{s_2}+y_{p_1}+y_{p_2}\qquad (y_s=y_{s_1}+y_{s_2},\quad y_r=y_{p_1}+y_{p_2}),
 \end{equation}
where $y_s=Fy\in Y_s$, $y_r=Qy\in Y_r$, $y_{s_i}=F_i\, y\in Y_{s_i}$ and $y_{p_i}=Q_i\, y\in Y_i$, $i=1,2$.

\subsection{Reduction of a DAE with the singular characteristic pencil to a system of ordinary differential and algebraic equations}\label{DAEsingReduce}

In Section \ref{BlockStruct}, the block form of a singular pencil of operators, which consists of the singular and regular blocks where zero and invertible blocks are separated out, was described.  Recall that the direct decompositions of spaces, which reduce the pencil to this block form, generate the projectors onto the subspaces from these decompositions, and the converse is also true.
The information given in Section~\ref{BlockStruct} is used below.

Consider the DAE \eqref{DAE} with the singular characteristic pencil $\lambda A+B$ that has the regular block $\lambda A_r+B_r$ (see \eqref{penc}) of index not higher than~1.

Applying the projectors $F_1$, $Q_1$, $Q_2$, $F_2$  from \eqref{ProjS}, \eqref{ProjR} to the equation \eqref{DAE} and using their properties, we obtain the equivalent system
 \begin{align}%\label{DAEsysProj}
\frac{d}{dt} (F_1 A S_1 x)+F_1 B S x &=F_1 f(t,x), \label{DAEsysProj1} \\
\frac{d}{dt} (Q_1 A P_1 x)+Q_1 BP_1 x &=Q_1 f(t,x), \label{DAEsysProj2} \\
Q_2 B P_2 x &= Q_2 f(t,x), \label{DAEsysProj3} \\
F_2 B S_1 x &= F_2 f(t,x). \label{DAEsysProj4}
 \end{align}

Using the representation \eqref{xsr}, i.e., ${x=x_{s_1}+x_{s_2}+x_{p_1}+x_{p_2}}$ where ${x_{s_i}=S_i x\in X_{s_i}}$ and ${x_{p_i}=P_i x\in X_i}$ ($i=1,2$), and the operators \eqref{ABssExtend} and \eqref{ABrrExtend}, we obtain the system equivalent~to~\eqref{DAEsysProj1}--\eqref{DAEsysProj4}:
 \begin{equation}\label{DAEsysExtOp}
\begin{split}
\frac{d}{dt} (\Es A_{gen} x_{s_1})+ \Es B_{gen} x_{s_1} + \Es B_{und} x_{s_2} = F_1 f(t,x), \\
\frac{d}{dt} (\Es A_1 x_{p_1})+ \Es B_1 x_{p_1} = Q_1 f(t,x),\\
\Es B_2 x_{p_2} =Q_2 f(t,x), \\
\Es B_{ov}x_{s_1}=F_2 f(t,x).
\end{split}
 \end{equation}
Multiplying the first three equations of the system \eqref{DAEsysExtOp}  by the semi-inverse operators $\EuScript A_{gen}^{(-1)}$, $\EuScript A_1^{(-1)}$ and $\EuScript B_2^{(-1)}$ (the method of their calculation is indicated in Section \ref{BlockStruct}) respectively,  we get the equivalent system (where $x_{s_i}=S_i x$, $x_{p_i}=P_i x$, $i=1,2$)
\begin{align}
\frac{d}{dt} x_{s_1} &=\Es A_{gen}^{(-1)}\big(F_1 f(t,x)-\Es B_{gen} x_{s_1}-\Es B_{und} x_{s_2}\big),   \label{DAEsysExtDE1}\\
\frac{d}{dt}x_{p_1} &=\Es A_1^{(-1)} \big(Q_1 f(t,x)-\Es B_1 x_{p_1}\big),    \label{DAEsysExtDE2}\\
0 &=\Es B_2^{(-1)} Q_2 f(t,x)- x_{p_2},   \label{DAEsysExtAE1}\\
0 &=F_2 f(t,x)-\Es B_{ov}x_{s_1}.  \label{DAEsysExtAE2}
\end{align}

Thus, the singular semilinear DAE \eqref{DAE} has been reduced to the equivalent system \eqref{DAEsysExtDE1}--\eqref{DAEsysExtAE2} of ordinary differential equations (ODEs) and algebraic equations (AEs). Instead of the obtained system \eqref{DAEsysExtDE1}--\eqref{DAEsysExtAE2}, equivalent to the DAE \eqref{DAE}, one can also obtain the equivalent system with the restricted operators, as in \cite{Fil.UMJ}.

 \smallskip
Recall that a function $W\in C(D,\R)$, where $D\subset\Rn$ is some set  containing the origin, is called \emph{positive definite} if $W(0)=0$ and $W(x)>0$ for all $x\in D\setminus \{0\}$; a function $V\in C([t_+,\infty)\times D,\R)$ ($D$ means the same set) is called \emph{positive definite} if $V(t,0)\equiv 0$ and there exists a positive definite function $W\in C(D,\R)$ such that $V(t,x)\ge W(x)$ for all $t\in [t_+,\infty)$, $x\in D\setminus \{0\}$.

Let $V\in C^1([t_+,\infty)\times D_{s_1}\times D_{p_1},\R)$ be some positive definite function and $D_{s_1}\times D_{p_1}\subset X_{s_1}\times X_1$. The derivative  $V'_{\eqref{DAEsysExtDE1},\eqref{DAEsysExtDE2}}(t,x_{s_1},x_{p_1})$ of the function $V$ along the trajectories of the system  (or the derivative of $V$ with respect to the system) \eqref{DAEsysExtDE1}, \eqref{DAEsysExtDE2} has the form
 \begin{multline}\label{dVDAEsing}
V'_{\eqref{DAEsysExtDE1},\eqref{DAEsysExtDE2}}(t,x_{s_1},x_{p_1})= \frac{\partial V}{\partial t}(t,x_{s_1},x_{p_1})+\frac{\partial V}{\partial (x_{s_1},x_{p_1})}(t,x_{s_1},x_{p_1})\cdot \Upsilon(t,x) = \\
= \frac{\partial V}{\partial t}(t,x_{s_1},x_{p_1})+ \frac{\partial V}{\partial x_{s_1}}(t,x_{s_1},x_{p_1})\cdot  \left[\Es A_{gen}^{(-1)}\big(F_1 f(t,x)-\Es B_{gen} x_{s_1}-\Es B_{und} x_{s_2}\big)\right] + \\
+\frac{\partial V}{\partial x_{p_1}}(t,x_{s_1},x_{p_1})\cdot   \left[\Es A_1^{(-1)} \big(Q_1 f(t,x)-\Es B_1 x_{p_1}\big)\right],
 \end{multline}
where
\begin{equation}\label{Upsilon}
\Upsilon(t,x)=\begin{pmatrix}
\Es A_{gen}^{(-1)}\big(F_1 f(t,x)-\Es B_{gen} x_{s_1}-\Es B_{und} x_{s_2}\big) \\
\Es A_1^{(-1)} \big(Q_1 f(t,x)-\Es B_1 x_{p_1}\big)  \end{pmatrix}
\end{equation}
is a vector consisting of the right-hand sides of the equations \eqref{DAEsysExtDE1}, \eqref{DAEsysExtDE2}. As usual,  $\frac{\partial}{\partial(x_{s_1},x_{p_1})}= \left(\frac{\partial}{\partial x_{s_1}},\frac{\partial}{\partial x_{p_1}}\right)$.

When proving theorems, we will use the representation of an element $x\in\R^n$ in the form \eqref{xsr} (with respect to the direct sum of subspaces \eqref{Xssrr}) and its corresponding representation in the form $x=(x_{s_1},x_{s_2},x_{p_1},x_{p_2})$ (with respect to the corresponding direct product of subspaces). The correspondence between these representations is established below and, in general, is obvious.

Taking into account that the sum of subspaces in the decomposition \eqref{Xssrr} is direct and, accordingly, any element $x\in \Rn$ can be uniquely represented in the form $x=x_{s_1}+x_{s_2}+x_{p_1}+x_{p_2}$ \eqref{xsr}, one can identify an ordered collection $(x_{s_1},x_{s_2},x_{p_1},x_{p_2})\in X_{s_1}\times X_{s_2}\times X_1 \times X_2$, which is assumed to be a column vector  (i.e., a vector $(x_{s_1}^\T,x_{s_2}^\T,x_{p_1}^\T,x_{p_2}^\T)^\T$) when it is used in matrix calculations, with the corresponding element $x=x_{s_1}+x_{s_2}+x_{p_1}+x_{p_2}\in \Rn=X_{s_1}\dot+ X_{s_2}\dot+X_1 \dot+X_2$. 
Obviously, $\dim(X_{s_1}\times X_{s_2}\times X_1 \times X_2)=\dim(\Rn)=n$ and the space $\Rn$ is isomorphic to the space $X_{s_1}\times X_{s_2}\times X_1 \times X_2$.   A norm in the space $X_{s_1}\times X_{s_2}\times X_1 \times X_2$ is defined so that the norms of any element of the form $x=x_{s_1}$ and the corresponding element (ordered collection) $x=(x_{s_1},0,0,0)$ coincide and, similarly, the norms of the elements $x=x_{s_2}$, $x=x_{p_1}$, $x=x_{p_2}$ and their notations in the form of the corresponding ordered collections  $x=(0,x_{s_2},0,0)$, $x=(0,0,x_{p_1},0)$ and $x=(0,0,0,x_{p_2})$ coincide.  In addition, norms in $X_{s_1}\dot+ X_{s_2}\dot+X_1 \dot+X_2$ and $X_{s_1}\times X_{s_2}\times X_1\times X_2$ are defined so that they coincide for any element $x$. 
Thus, the representations $x=x_{s_1}+x_{s_2}+x_{p_1}+x_{p_2}$ and $x=(x_{s_1},x_{s_2},x_{p_1},x_{p_2})$ define the same element $x$, where $x_{s_i}\in X_{s_i}$, $x_{p_i}\in X_i$, $i=1,2$. %$\|\cdot\|$

 \smallskip
In a similar way, an ordered collection (a column vector) $y=(y_{s_1},y_{s_2},y_{p_1},y_{p_2})\in Y_{s_1}\times Y_{s_2}\times Y_1 \times Y_2$ can be identified with the corresponding element  $y=y_{s_1}+y_{s_2}+y_{p_1}+y_{p_2}\in \Rm=Y_{s_1}\dot+ Y_{s_2}\dot+Y_1 \dot+Y_2$.

 \smallskip
Consider one more representation of a vector $x\in \Rn$, which allows one to reduce the DAE \eqref{DAE} to an equivalent system of ODEs and AEs with the operators restricted to the subspaces from \eqref{Xssrr} (the system of such type was used in \cite{Fil.UMJ}).  Denote the dimensions of the subspaces from the decomposition \eqref{Xssrr} as $\dim X_{s_1}=b$, $\dim X_{s_2}=l$, $\dim X_1=a$ and $\dim X_2=d$ ($b+l+a+d=n$, $\dim X_s=b+l$, $\dim X_r=a+d$). Further, we choose some bases $\{s_j\}_{j=1}^b$, $\{s_{b+j}\}_{j= 1}^l$, $\{p_j\}_{j=1}^a$ and $\{p_{a+j}\}_{j=1}^d$ of the subspaces $X_{s_1}$, $X_{s_2}$, $X_1$ and $X_2$, respectively. 
The union of these bases is a basis of the space $\Rn = \R^b\times \R^l\times \R^a\times \R^d$, and with respect to this basis each vector $x\in \Rn$ ($x=x_{s_1}+x_{s_2}+x_{p_1}+x_{p_2}$) can be written in the form of the column vector $x=(w^T,\, \xi^T,\, z^T,\, u^T)^T$, where $w \in \R^b$, $\xi \in \R^l$, $z \in \R^a$ and $u\in\R^d$ are column vectors consisting of the coordinates of the vector $x$ with respect to the chosen bases in the subspaces  $X_{s_1}$, $X_{s_2}$, $X_1$ and $X_2$ respectively ($x=\sum\limits_{j=1}^b w_j\, s_j + \sum\limits_{j=1}^l \xi_j\, s_{b+j} + \sum\limits_{j=1}^a z_j\, p_j +\sum\limits_{j=1}^d u_j\, p_{a+j}$ is a representation of the vector $x$   with respect to the chosen basis of $\Rn$).  The specified one-to-one correspondence between $X_{s_1}$, $X_{s_2}$, $X_1$, $X_2$ and  $\R^b$, $\R^l$, $\R^a$, $\R^d$ (between each $x_{s_1}$, $x_{s_2}$, $x_{p_1}$, $x_{p_2}$ and each $w$, $\xi$, $z$, $u$), respectively, defines the linear operators ${S_b\colon \R^b \to X_{s_1}}$, ${S_l\colon \R^l \to X_{s_2}}$, ${P_a\colon  \R^a \to X_1}$, ${P_d\colon \R^d \to X_2}$ establishing an isomorphism between the spaces,  which have the inverse ${S_b^{-1}\colon X_{s_1}\to \R^b}$, ${S_l^{-1}\colon X_{s_2}\to \R^l}$, ${P_a^{-1}\colon  X_1\to \R^a}$ and ${P_d^{-1}\colon X_2 \to \R^d}$. We restrict the operators in the equations of the system \eqref{DAEsysProj1}--\eqref{DAEsysProj4} to the subspaces $X_{s_1}$, $X_{s_2}$, $X_1$, $X_2$, use the representation \eqref{xsr} and make the change of variables
 $$
x_{s_1}=S_b\, w,\quad x_{s_2}=S_l\,\xi,\quad x_{p_1}=P_a\, z,\quad x_{p_2}=P_d\,u,
 $$
using the operators introduced above. Then, applying the operators $A_{gen}^{-1}$, $A_1^{-1}$, $B_2^{-1}$ (see Section~\ref{BlockStruct}) and $S_b^{-1}$, $P_a^{-1}$, $P_d^{-1}$, we obtain the system (similar to the one in \cite{Fil.UMJ}), equivalent to the DAE \eqref{DAE}:
\begin{align}
\frac{d}{dt}w &= S_b^{-1} A_{gen}^{-1}\left(F_1 \tilde{f}(t,w,\xi,z,u)- B_{gen}S_b\,w - B_{und}S_l\,\xi \right), \label{DAEsysRestrDE1}\\ 
\frac{d}{dt}z &= P_{a}^{-1} A_1^{-1} \left(Q_1\tilde{f}(t,w,\xi,z,u)- B_1P_a\,z \right),  \label{DAEsysRestrDE2}\\
0 &=P_{d}^{-1} B_2^{-1} Q_{2} \tilde{f}(t,w,\xi,z,u)-u,  \label{DAEsysRestrAE1}\\
0 &=F_2 \tilde{f}(t,w,\xi,z,u)-B_{ov}S_b\, w, \label{DAEsysRestrAE2} 
\end{align}
where $\tilde{f}(t,w,\xi,z,u)= f(t,S_b\,w+S_l\,\xi+P_a\,z+P_d\,u)$ and the projectors $F_i$ and $Q_i$ ($i=1,2$) on the subspaces $Y_{s_i}$ and $Y_i$ are considered as the operators acting from $\Rm$ to the spaces $Y_{s_i}$ and $Y_i$, respectively. 
In general, the projectors $F_i$, $Q_i$ by definition belong to $\mathrm{L}(\Rm)$, and $Y_{s_i}$, $Y_i$ are their ranges, respectively ($\Ker F_i=(\Rm\setminus Y_{s_i})\cup\{0\}$ and $\Ker Q_i=(\Rm\setminus Y_i)\cup\{0\}$), but since the spaces $Y_{s_1}$, $Y_1$ and $Y_2$ are the domains of definition of the restricted (induced) operators $A_{gen}^{-1}$, $A_1^{-1}$ and $B_2^{-1}$, respectively, and in addition ${B_{ov}\in \mathrm{L}(X_{s_1},Y_{s_2})}$, here we consider the projectors $F_i$ and $Q_i$ as operators from $\mathrm{L}(\Rm,Y_{s_i})$ and $\mathrm{L}(\Rm,Y_i)$, respectively, with the preservation of their projection properties, i.e., $F_i y=F_i y_{s_i}=y_{s_i}\in Y_{s_i}$  and $Q_i y=Q_i y_{p_i}=y_{p_i}\in Y_i$ ($i=1,2$) for any $y=y_{s_1}+y_{s_2}+y_{p_1}+y_{p_2}\in \Rm$ (see the representation \eqref{ysr}), where $y_{s_i}\in Y_{s_i}$ and $y_{p_i}\in Y_i$, $i=1,2$, and, for convenience, we keep the previous notation for these operators.
For clarity, note that if we choose some basis $\{e_j\}_{j=1}^{m-d}$ of $Y_s\dot +Y_1$ and some basis $\{q_j\}_{j=1}^d$ of $Y_2$ (notice that $\dim Y_2=\dim X_2=d$), and we take the basis of $\Rm$ as the union of these bases, i.e., in the form $\{e_1,...,e_{m-d},q_1,...,q_d\}$, then the matrix corresponding to the mentioned operator $Q_2\in \mathrm{L}(\Rm,Y_2)$ with respect to the chosen bases in $\Rm$ and $Y_2$ will have the form $Q_2=\begin{pmatrix} 0 & I_{Y_2} \end{pmatrix}$, where $0$ is the null $d\times m-d$  matrix and $I_{Y_2}$ is the identity $d\times d$ matrix corresponding to the identity operator  $I_{Y_2}$ with respect to the chosen basis of $Y_2$.

 \section{Global solvability of singular (nonregular) semilinear DAEs}\label{SectGlobSolv}

\begin{remark}\label{RemConsistIni}
We introduce the manifold
\begin{equation}\label{L_tSing}
L_{t_*}=\{(t,x)\in [t_*,\infty)\times\Rn \mid (F_2+Q_2)[Bx-f(t,x)]=0\},
\end{equation}
where the number $t_*\ge t_+$ is a parameter. For example, the manifold $L_{t_+}$ considered below has the form \eqref{L_tSing} where $t_*=t_+$.
The manifold \eqref{L_tSing} is defined by the equations \eqref{DAEsysProj3} (or $Q_2[Bx-f(t,x)]=0$) and \eqref{DAEsysProj4} (or $F_2[Bx-f(t,x)]=0$) and can be represented as
 $$
L_{t_*}=\{(t,x)\in [t_*,\infty)\times\Rn \mid \text{$(t,x)$ satisfies the equations \eqref{DAEsysProj3}, \eqref{DAEsysProj4} }\}.
 $$

The initial values $t_0$, $x_0$ satisfying the \emph{consistency condition $(t_0,x_0)\in L_{t_+}$} are called \emph{consistent initial values}, and, accordingly, the initial point $(t_0,x_0)\in L_{t_+}$ is called a \emph{consistent initial point}.

The consistency condition $(t_0,x_0)\in L_{t_+}$ is one of the necessary conditions for the existence of a solution of the IVP \eqref{DAE}, \eqref{ini}.
\end{remark}

 \begin{theorem}\label{Th_SingGlobSol}
Let $f\in C([t_+,\infty)\times \Rn,\Rm)$ and $\lambda A+B$ be a singular pencil of operators such that its regular block $\lambda A_r+B_r$ from \eqref{penc} has the index not higher than 1. Let the following conditions be fulfilled:
\begin{enumerate}
\item\label{SoglSing} For any fixed  $t\in [t_+,\infty)$, $x_{s_1}\in X_{s_1}$, $x_{s_2}\in D_{s_2}$, where $D_{s_2}\subset X_{s_2}$ is a some set, and $x_{p_1}\in X_1$, there exists a unique $x_{p_2}\in X_2$ such that ${(t,x_{s_1}+x_{s_2}+x_{p_1}+x_{p_2})\in L_{t_+}}$.

\item\label{InvSing} There exists the partial derivative ${\frac{\partial}{\partial x} f\in C([t_+,\infty)\times \Rn, \mathrm{L}(\Rn,\Rm))}$.\; For any fixed $t_*$, ${x_*=x_{s_1}^*+x_{s_2}^*+x_{p_1}^*+x_{p_2}^*}$  such that $(t_*,x_*)\in L_{t_+}$ and $x_{s_2}^*\in D_{s_2}$, the operator $\Phi_{t_*,x_*}$ defined~by
     \begin{equation}\label{funcPhiSing}
    \Phi_{t_*,x_*}=\left[\frac{\partial Q_2f}{\partial x}(t_*,x_*)- B\right] P_2\colon X_2\to Y_2
     \end{equation}
     is invertible.

\item\label{ExtensSing} There exists a number $R>0$, a positive definite function  $V\in C^1([t_+,\infty)\times D_{s_1}\times D_{p_1},\R)$, where a set $D_{s_1}\times D_{p_1}\subset X_{s_1}\times X_1$ is such that $D_{s_1}\times D_{p_1}\supset \{\|(x_{s_1},x_{p_1})\|\ge R\}$, and a function $\chi\in C([t_+,\infty)\times (0,\infty),\R)$ such that:
  \begin{enumerate}[label={\upshape(\alph*)},ref={\upshape(\alph*)},topsep=1pt]
  \item\label{ExtensSing1}   ${V(t,x_{s_1},x_{p_1})\to\infty}$ uniformly in $t$ on each finite interval $[a,b)\subset[t_+,\infty)$ as ${\|(x_{s_1},x_{p_1})\|\to\infty}$;

  \item\label{ExtensSing2} for all $(t,x_{s_1}+x_{s_2}+x_{p_1}+x_{p_2})\in L_{t_+}$, for which $x_{s_2}\in D_{s_2}$ and ${\|(x_{s_1},x_{p_1})\|\ge R}$, the inequality
  \begin{equation}\label{LagrDAEsing}
    V'_{\eqref{DAEsysExtDE1},\eqref{DAEsysExtDE2}}(t,x_{s_1},x_{p_1})\le \chi\big(t,V(t,x_{s_1},x_{p_1})\big),
  \end{equation}
  where $V'_{\eqref{DAEsysExtDE1},\eqref{DAEsysExtDE2}}(t,x_{s_1},x_{p_1})$ has the form \eqref{dVDAEsing}, is satisfied;

  \item\label{GlobSolv} the differential inequality \eqref{L1v}, i.e., $dv/dt\le \chi(t,v)$  \,($t\in [t_+,\infty)$), does not have positive solutions with finite escape time.
  \end{enumerate}
 \end{enumerate}
Then for each initial point ${(t_0,x_0)\in L_{t_+}}$, where $S_2x_0\in D_{s_2}$, the initial value problem \eqref{DAE}, \eqref{ini} has a unique global (i.e., on $[t_0,\infty)$) solution $x(t)$ for which the choice of the function $\phi_{s_2}\in C([t_0,\infty),D_{s_2})$ with the initial value  $\phi_{s_2}(t_0)=S_2 x_0$ uniquely defines the component $S_2x(t)=\phi_{s_2}(t)$ when ${\rank(\lambda A+B)<n}$; when ${\rank(\lambda A+B)=n}$, the component $S_2 x$ is absent.
 \end{theorem}

 \begin{remark}
In general, the operator defined (for fixed $t_*$, $x_*$) by the formula from \eqref{funcPhiSing}  belongs to $\mathrm{L}(\Rn,\Rm)$, i.e.,
 \begin{equation}\label{funcPhiSingExt}
\widehat{\Phi}_{t_*,x_*}:=\bigg[\frac{\partial Q_2f}{\partial x}(t_*,x_*)- B\bigg]P_2\in \mathrm{L}(\Rn,\Rm),
 \end{equation}
and acts so that $\widehat{\Phi}_{t_*,x_*}\colon\Rn\to Y_2$ and ${X_1\dot +X_s\subset \Ker(\widehat{\Phi}_{t_*,x_*})}$. Its restriction to $X_2$ is the operator $\Phi_{t_*,x_*}=\widehat{\Phi}_{t_*,x_*}\big|_{X_2}$ defined by \eqref{funcPhiSing}, and since the operator $\Phi_{t_*,x_*}$ has an inverse, 
then $\widehat{\Phi}_{t_*,x_*}\Rn=\widehat{\Phi}_{t_*,x_*}X_2=Y_2$ (${X_1\dot+X_s=\Ker(\widehat{\Phi}_{t_*,x_*})}$).   Note that the operator $\Phi_{t_*,x_*}$ and its inverse $\Phi_{t_*,x_*}^{-1}\in \mathrm{L}(Y_2,X_2)$ satisfy the equalities $\Phi_{t_*,x_*} \Phi_{t_*,x_*}^{-1}=Q_2\big|_{Y_2}$,  $\Phi_{t_*,x_*}^{-1} \Phi_{t_*,x_*}=P_2\big|_{X_2}$ since $P_2$, $Q_2$ are the identity operators in $X_2$, $Y_2$, respectively.
Thus, the extension $\widehat{\Phi}_{t_*,x_*}^{(-1)}\in \mathrm{L}(\Rm,\Rn)$ of the operator $\Phi_{t_*,x_*}^{-1}$ to $\Rm$ that satisfies the equalities
 $$
\widehat{\Phi}_{t_*,x_*}^{(-1)} \widehat{\Phi}_{t_*,x_*}=P_2,\quad \widehat{\Phi}_{t_*,x_*} \widehat{\Phi}_{t_*,x_*}^{(-1)}=Q_2,\quad    \widehat{\Phi}_{t_*,x_*}^{(-1)}=P_2 \widehat{\Phi}_{t_*,x_*}^{(-1)},
 $$
is a semi-inverse operator of $\widehat{\Phi}_{t_*,x_*}$, that is, $\widehat{\Phi}_{t_*,x_*}^{(-1)}\Rm=\widehat{\Phi}_{t_*,x_*}^{(-1)}Y_2=X_2$ and $\Phi_{t_*,x_*}^{-1}=\widehat{\Phi}_{t_*,x_*}^{(-1)}\big|_{Y_2}$.
 \end{remark}

 \begin{proof}
As shown above, the DAE \eqref{DAE} is equivalent to the system
\eqref{DAEsysExtDE1}--\eqref{DAEsysExtAE2}, i.e.,
\begin{align*}
\frac{d}{dt} x_{s_1} &=\Es A_{gen}^{(-1)}[F_1 f(t,x)-\Es B_{gen} x_{s_1}-\Es B_{und} x_{s_2}],  \\
\frac{d}{dt}x_{p_1} &=\Es A_1^{(-1)}[Q_1 f(t,x)-\Es B_1 x_{p_1}], \\
0 &=\Es B_2^{(-1)} Q_2 f(t,x)- x_{p_2},   \\
0 &=F_2 f(t,x)-\Es B_{ov}x_{s_1},
\end{align*}
where $x_{s_i}=S_i x\in X_{s_i}$ and $x_{p_i}=P_i x\in X_i$, $i=1,2$,  and the representation $x=x_{s_1}+x_{s_2}+x_{p_1}+x_{p_2}$ \eqref{xsr} is uniquely determined for each $x\in \Rn$.

Introduce the mapping
 \begin{equation}\label{HatPsiSing}
\widehat{\Psi}(t,x):=\Es B_2^{(-1)}Q_2 f(t,x)-x_{p_2}\colon [t_+,\infty)\times \Rn\to X_2,
 \end{equation}
where $x_{p_2}=P_2 x$. 
Then the equation \eqref{DAEsysExtAE1} is equivalent to the equation
\begin{equation}\label{DAEsysExtAE1equiv}
\widehat{\Psi}(t,x)=0.
\end{equation}
Since $\Es B_2^{(-1)}Q_2\Rm=\Es B_2^{(-1)}Y_2=X_2=B_2^{-1}Y_2= B_2^{-1}(Q_2\Rm)$ (recall that $B_2^{-1}=\EuScript B_2^{(-1)}\big|_{Y_2}$ and  $Q_2\Rm=Y_2$) and $Q_2f(t,x)\in Y_2$ for any $(t,x)$, then \eqref{DAEsysExtAE1equiv} is equivalent to the equation
 \begin{equation}\label{DAEsysExtAE1equiv2}
B_2^{-1}Q_2f(t,x)-x_{p_2}=0,
 \end{equation}
where the projector $Q_2$ on the subspace $Y_2$ is considered as the operator belonging to $\mathrm{L}(\Rm,Y_2)$, and at the same time, its projection properties are retained and, for convenience, its previous notation $Q_2$ does not change as well as for the system \eqref{DAEsysRestrDE1}--\eqref{DAEsysRestrAE2} presented in Section~\ref{DAEsingReduce}.

In what follows, we will assume that $Q_2\in \mathrm{L}(\Rm)$ is the projector on the subspace $Y_2$ defined in Section \ref{BlockStruct} if the product of operators $\Es B_2^{(-1)}Q_2$ (or $Q_2 \Es B_2$) is considered, and that $Q_2\in \mathrm{L}(\Rm,Y_2)$ is the operator having the same projection properties as the projector $Q_2$ defined in Section \ref{BlockStruct}, i.e., $Q_2 y=Q_2 y_{p_2}=y_{p_2}\in Y_2$ for any $y=y_{s_1}+y_{s_2}+y_{p_1}+y_{p_2}\in \Rm$ where $y_{s_i}\in Y_{s_i}$ and $y_{p_i}\in Y_i$ ($i=1,2$), if the product of operators $B_2^{-1}Q_2$  is considered. Since, in fact, the described differences are formal and become significant only in the transition from the operators to the corresponding matrices, then we keep the same notation $Q_2$ for both considered cases. Similar assumptions can be made concerning other projectors considered in Section \ref{BlockStruct}.

Taking into account the correspondence $x=x_{s_1}+x_{s_2}+x_{p_1}+x_{p_2}= (x_{s_1},x_{s_2},x_{p_1},x_{p_2})$, where $x_{s_i}\in X_{s_i}$ and $x_{p_i}\in X_i$, $i=1,2$, (the correspondence between $X_{s_1}\dot+ X_{s_2}\dot+X_1 \dot+X_2$ and $X_{s_1}\times X_{s_2}\times X_1\times X_2$~) which is established  in Section \ref{DAEsingReduce}, we denote
 $$
\tilde{f}(t,x_{s_1},x_{s_2},x_{p_1},x_{p_2})= f(t,x_{s_1}+x_{s_2}+x_{p_1}+x_{p_2})=f(t,x)
 $$
and consider the mapping
 \begin{equation}\label{funcPsiSing}
\Psi(t,x_{s_1},x_{s_2},x_{p_1},x_{p_2}):= B_2^{-1}Q_2 \tilde{f}(t,x_{s_1},x_{s_2},x_{p_1},x_{p_2})-x_{p_2},
 \end{equation}
where $\Psi\colon [t_+,\infty)\times X_{s_1}\times X_{s_2}\times X_1\times X_2\to X_2$. Then the equation \eqref{DAEsysExtAE1equiv2} can be written as
 \begin{equation}\label{DAEsysAE1equiv}
\Psi(t,x_{s_1},x_{s_2},x_{p_1},x_{p_2})=0,
 \end{equation}
and this equation is equivalent to the equation \eqref{DAEsysExtAE1equiv} and hence to the equation \eqref{DAEsysExtAE1}, as shown above.
Obviously, $\Psi\in C([t_+,\infty)\times X_{s_1}\times X_{s_2}\times X_1\times X_2,\,X_2)$ has continuous partial derivatives with respect to $x_{s_1}$, $x_{s_2}$, $x_{p_1}$, $x_{p_2}$, and its partial derivatives with respect to $(x_{s_1},x_{s_2},x_{p_1})$ and $x_{p_2}$ at the point $(t_*,x^*_{s_1},x^*_{s_2},x^*_{p_1},x^*_{p_2})$ have the form\;
$\frac{\partial\Psi (t_*,x^*_{s_1},x^*_{s_2},x^*_{p_1},x^*_{p_2})}{\partial (x_{s_1},x_{s_2},x_{p_1})}= B_2^{-1} \bigg(\frac{\partial Q_2\tilde{f} (t_*,x^*_{s_1},x^*_{s_2},x^*_{p_1},x^*_{p_2})}{\partial x_{s_1}}, \frac{\partial Q_2 \tilde{f} (t_*,x^*_{s_1},x^*_{s_2},x^*_{p_1},x^*_{p_2})}{\partial x_{s_2}}, \frac{\partial Q_2 \tilde{f}(t_*,x^*_{s_1},x^*_{s_2},x^*_{p_1},x^*_{p_2})}{\partial x_{p_1}}\bigg)\in \mathrm{L}(X_{s_1}\times X_{s_2}\times X_1,\, X_2)$,
 \begin{equation}\label{W_tx}
W_{t_*,x_*}:=\dfrac{\partial \Psi}{\partial x_{p_2}}(t_*,x^*_{s_1},x^*_{s_2},x^*_{p_1},x^*_{p_2}) = 
B_2^{-1}Q_2\left[\frac{\partial Q_2 f}{\partial x}(t_*,x_*)- B\right]\! P_2\big|_{X_2}=B_2^{-1}\Phi_{t_*,x_*}\in\mathrm{L}(X_2),
 \end{equation}
where $x_*=x^*_{s_1}+x^*_{s_2}+x^*_{p_1}+x^*_{p_2}$ and $\Phi_{t_*,x_*}\in \mathrm{L}(X_2,Y_2)$ is the operator defined by \eqref{funcPhiSing}.  Since for any fixed element $(t,x_{s_1}+x_{s_2}+x_{p_1}+x_{p_2})\in L_{t_+}$  such that $x_{s_2}\in D_{s_2}$ the operator $\Phi_{t,x}$ (where $x=x_{s_1}+x_{s_2}+x_{p_1}+x_{p_2}$) has the inverse $\Phi_{t,x}^{-1}\in \mathrm{L}(Y_2,X_2)$, then the operator $W_{t,x}$ also has the inverse $W_{t,x}^{-1}=\Phi_{t,x}^{-1}B_2\in \mathrm{L}(X_2)$ for the indicated $(t,x)$.

Note that a point $(t,x)\in [t_+,\infty)\times\Rn$ belongs to the manifold $L_{t_+}$ if and only if it satisfies the equations \eqref{DAEsysProj3}, \eqref{DAEsysProj4} or the equivalent equations, e.g.,  \eqref{DAEsysExtAE1}, \eqref{DAEsysExtAE2} or \eqref{DAEsysExtAE1}, \eqref{DAEsysAE1equiv}, where  $(x_{s_1},x_{s_2},x_{p_1},x_{p_2})=x_{s_1}+x_{s_2}+x_{p_1}+x_{p_2}=x$.

Take any fixed $t_*\in [t_+,\infty)$, $x_{s_1}^*\in X_{s_1}$, $x_{s_2}^*\in D_{s_2}$, $x_{p_1}^*\in X_1$. Then, by virtue of condition \ref{SoglSing}, there exists a unique $x_{p_2}^*\in X_2$ such that $(t_*,x_*)\in L_{t_+}$, where $x_*=x_{s_1}^*+x_{s_2}^*+x_{p_1}^*+x_{p_2}^*$.
As shown above, for this $(t_*,x_*)$ the operator \eqref{W_tx} has the inverse $W_{t_*,x_*}^{-1}\in \mathrm{L}(X_2)$. In addition, the function  $\Psi(t,x_{s_1},x_{s_2},x_{p_1},x_{p_2})$ has a continuous partial derivative with respect to $(x_{s_1},x_{s_2},x_{p_1},x_{p_2})$ at every point from $[t_+,\infty)\times X_{s_1}\times X_{s_2}\times X_1\times X_2$.
Using the implicit function theorems and fixed point theorems \cite{Schwartz1}, we obtain that there exist neighborhoods $U_\delta (t_*,x_{s_1}^*,x_{s_2}^*,x_{p_1}^*)=U_{\delta_1}(t_*)\times U_{\delta_2}(x_{s_1}^*)\times U_{\delta_3}(x_{s_2}^*)\times U_{\delta_4}(x_{p_1}^*)$  and $U_\varepsilon(x_{p_2}^*)$ (for $t_*=t_+$ there exists a semi-open interval $U_{\delta_1}(t_+)=[t_+,t_+ +\delta_1)$, and the neighborhood $U_{\delta_3}(x_{s_2}^*)$ can be closed)  and a unique function  $x_{p_2}=\mu(t,x_{s_1},x_{s_2},x_{p_1})\in C(U_\delta (t_*,x_{s_1}^*,x_{s_2}^*,x_{p_1}^*), U_\varepsilon (x_{p_2}^*))$, continuously differentiable in $(x_{s_1},x_{s_2},x_{p_1})$, such that   $\mu(t_*,x_{s_1}^*,x_{s_2}^*,x_{p_1}^*)=x_{p_2}^*$ and $\Psi(t,x_{s_1},x_{s_2},x_{p_1},\mu(t,x_{s_1},x_{s_2},x_{p_1}))=0$ for all $(t,x_{s_1},x_{s_2},x_{p_1})\in U_\delta (t_*,x_{s_1}^*,x_{s_2}^*,x_{p_1}^*)$, i.e., the function $\mu$ is a solution of the equation \eqref{DAEsysAE1equiv} with respect to $x_{p_2}$. 
Since the implicit function theorems \cite{Schwartz1} assume that the set of variables is open, then to prove the existence of an implicitly defined function with the above properties when $t_*=t_+$ (i.e., ${U_{\delta_1}(t_+)=[t_+,t_+ +\delta_1)}$) and when the set $D_{s_2}$ is not open  (accordingly, $U_{\delta_3}(x_{s_2}^*)$ can be closed), the fixed point theorems \cite[Theorems 46, $46_2$]{Schwartz1} as well as the proofs of the implicit function theorems \cite[Theorems~25,~28]{Schwartz1} (which in turn use the same fixed point theorems) are used.

Thus, it is proved that in some neighborhood of each (fixed) point $(t_*,x_{s_1}^*,x_{s_2}^*,x_{p_1}^*)\in [t_+,\infty)\times X_{s_1}\times D_{s_2}\times X_1$ there exists a unique solution $x_{p_2}=\mu_{t_*,x_{s_1}^*,x_{s_2}^*,x_{p_1}^*}(t,x_{s_1},x_{s_2},x_{p_1})$ of the equation \eqref{DAEsysAE1equiv} and, hence, the equivalent equation \eqref{DAEsysExtAE1}, and this solution is continuous in $(t,x_{s_1},x_{s_2},x_{p_1})$, continuously differentiable in $(x_{s_1},x_{s_2},x_{p_1})$ and satisfies the equality  $\mu_{t_*,x_{s_1}^*,x_{s_2}^*,x_{p_1}^*}(t_*,x_{s_1}^*,x_{s_2}^*,x_{p_1}^*)=x_{p_2}^* \in D_{p_2}$,  where the set $D_{p_2}\subset  X_2$ is such that $[t_+,\infty)\times (X_{s_1}\dot+ D_{s_2}\dot+ X_1\dot+ D_{p_2})\subset L_{t_+}$. Recall that $x_{p_2}^*$  is uniquely determined for each such $(t_*,x_{s_1}^*,x_{s_2}^*,x_{p_1}^*)$ by virtue of condition \ref{SoglSing}. We introduce the function
 $$
\eta\colon [t_+,\infty)\times X_{s_1}\times D_{s_2}\times X_1\to D_{p_2}
 $$
and define  $\eta(t,x_{s_1},x_{s_2},x_{p_1})= \mu_{t_*,x_{s_1}^*,x_{s_2}^*,x_{p_1}^*}(t,x_{s_1},x_{s_2},x_{p_1})$ at \,$(t,x_{s_1},x_{s_2},x_{p_1})=(t_*,x_{s_1}^*,x_{s_2}^*,x_{p_1}^*)$ for each $(t_*,x_{s_1}^*,x_{s_2}^*,x_{p_1}^*)\in [t_+,\infty)\times X_{s_1}\times D_{s_2}\times X_1$. Then the function $x_{p_2}=\eta(t,x_{s_1},x_{s_2},x_{p_1})$ is continuous in $(t,x_{s_1},x_{s_2},x_{p_1})$, continuously differentiable in $(x_{s_1},x_{s_2},x_{p_1})$ and satisfies the equation \eqref{DAEsysExtAE1} as well as the equation \eqref{DAEsysAE1equiv}, i.e., $\Psi(t,x_{s_1},x_{s_2},x_{p_1},\eta(t,x_{s_1},x_{s_2},x_{p_1}))=0$, for   $(t,x_{s_1},x_{s_2},x_{p_1})\in [t_+,\infty)\times X_{s_1}\times D_{s_2}\times X_1$. 
Let us prove the uniqueness of the function $\eta$.  Indeed, if there exists a function $x_{p_2}=\zeta(t,x_{s_1},x_{s_2},x_{p_1})$ that has the same properties at some point $(t_*,x_{s_1}^*,x_{s_2}^*,x_{p_1}^*)\in [t_+,\infty)\times X_{s_1}\times D_{s_2}\times X_1$ as the function $\eta(t,x_{s_1},x_{s_2},x_{p_1})$, then $\eta(t_*,x_{s_1}^*,x_{s_2}^*,x_{p_1}^*)= \zeta(t_*,x_{s_1}^*,x_{s_2}^*,x_{p_1}^*)=x_{p_2}^*$ due to condition \ref{SoglSing}. This holds for each point $(t_*,x_{s_1}^*,x_{s_2}^*,x_{p_1}^*)\in [t_+,\infty)\times X_{s_1}\times D_{s_2}\times X_1$ and hence there exists a unique function $x_{p_2}=\eta(t,x_{s_1},x_{s_2},x_{p_1})$ with the properties indicated above.

Choose any initial point $(t_0,x_0)\in L_{t_+}$, where $S_2x_0\in D_{s_2}$, and any function $\phi_{s_2}\in C([t_0,\infty),D_{s_2})$ satisfying the condition $\phi_{s_2}(t_0)=S_2 x_0$. Substitute the chosen function into $\eta$ and denote $q(t,x_{s_1},x_{p_1})=\eta(t,x_{s_1},\phi_{s_2}(t),x_{p_1})$.
Further, we substitute the functions $x_{p_2}=q(t,x_{s_1},x_{p_1})$ and $x_{s_2}=\phi_{s_2}(t)$ in \eqref{DAEsysExtDE1}, \eqref{DAEsysExtDE2} and obtain the system
\begin{align*}
\frac{d}{dt} x_{s_1} &=\Es A_{gen}^{(-1)} \big[F_1\tilde{f}(t,x_{s_1},\phi_{s_2}(t),x_{p_1},q(t,x_{s_1},x_{p_1}))- \Es B_{gen} x_{s_1}-\Es B_{und} \phi_{s_2}(t)\big], \\ %\eqref{DAEsysExtDE1}
\frac{d}{dt}x_{p_1} &=\Es A_1^{(-1)} \big[Q_1\tilde{f}(t,x_{s_1},\phi_{s_2}(t),x_{p_1},q(t,x_{s_1},x_{p_1}))- \Es B_1 x_{p_1}\big].  %\eqref{DAEsysExtDE2}
\end{align*}
We write this system in the form
\begin{equation}\label{DAEsysExtDEeta}
\frac{d}{dt}\omega=\widetilde{\Upsilon}(t,\omega),
\end{equation}
 \begin{align*}
& \text{where}\quad \omega=\begin{pmatrix} x_{s_1} \\ x_{p_1} \end{pmatrix},\quad \widetilde{\Upsilon}(t,\omega) =\begin{pmatrix} \Es A_{gen}^{(-1)} \big[F_1\tilde{f}(t,x_{s_1},\phi_{s_2}(t),x_{p_1},q(t,x_{s_1},x_{p_1}))- \Es B_{gen} x_{s_1}-\Es B_{und} \phi_{s_2}(t)\big] \\
\Es A_1^{(-1)} \big[Q_1\tilde{f}(t,x_{s_1},\phi_{s_2}(t),x_{p_1},q(t,x_{s_1},x_{p_1}))- \Es B_1 x_{p_1}\big] \end{pmatrix} = \\
& =\Upsilon (t,x_{s_1}+\phi_{s_2}(t)+x_{p_1}+q(t,x_{s_1},x_{p_1}))\quad \text{($\Upsilon(t,x)$ is defined in \eqref{Upsilon})}.
 \end{align*}
Due to the properties of $\tilde{f}(t,x_{s_1},x_{s_2},x_{p_1},x_{p_2})$, $q(t,x_{s_1},x_{p_1})$ and $\phi_{s_2}(t)$, the function $\widetilde{\Upsilon}(t,\omega)$ is continuous in $(t,\omega)$ and continuously differentiable in $\omega$ on $[t_0,\infty)\times X_{s_1} \times X_1$. Consequently, there exists a unique solution $\omega=\omega(t)$ of \eqref{DAEsysExtDEeta} on some interval $[t_0,\beta)$ which satisfies the initial condition
\begin{equation}\label{DEeta_ini}
\omega(t_0)=\omega_0,\quad \omega_0=(x_{s_1,0}^\T,x_{p_1,0}^\T)^\T,\quad x_{s_1,0}=S_1x_0,\; x_{p_1,0}=P_1x_0.
\end{equation}
Let us prove that the maximal interval of the existence of the solution is  $[t_0,\infty)$.

Let us introduce the function $V(t,\omega):=V(t,x_{s_1},x_{p_1})$, where $V(t,x_{s_1},x_{p_1})$ from the theorem condition \ref{ExtensSing}. It follows from condition \ref{ExtensSing} that the derivative of $V$  along the trajectories of the equation \eqref{DAEsysExtDEeta} satisfies the inequality
 \begin{equation}\label{L1dVdtSing}
V'_{\eqref{DAEsysExtDEeta}}(t,\omega)=\dfrac{\partial V}{\partial t}(t,\omega)+\dfrac{\partial V}{\partial \omega}(t,\omega)\cdot \widetilde{\Upsilon}(t,\omega)\le \chi\big(t,V(t,\omega)\big)
 \end{equation}
for all $t\ge t_0$, $\|\omega\|\ge R$. Due to condition \ref{GlobSolv},  the differential inequality \eqref{L1v}, $t\ge t_0$, does not have positive solutions with finite escape time. Hence, by \cite[Chapter~IV, Theorem~XIII]{LaSal-Lef} every solution of  \eqref{DAEsysExtDEeta} exists on $[t_0,\infty)$ (every solution is defined in the future \cite{LaSal-Lef}), and, consequently, the solution $\omega(t)=(\omega_{s_1}(t)^\T,\omega_{p_1}(t)^\T)^\T$ is global.

Thus, the functions (components of the global solution $\omega(t)$ of the equation \eqref{DAEsysExtDEeta}) $\omega_{s_1}\in C^1([t_0,\infty),X_{s_1})$, $\omega_{p_1}\in C^1([t_0,\infty),X_1)$ and the function  $q(t,\omega_{s_1}(t),\omega_{p_1}(t))$ are a solution of the system \eqref{DAEsysExtDE1}, \eqref{DAEsysExtDE2} and \eqref{DAEsysExtAE1} on $[t_0,\infty)$, and the equation \eqref{DAEsysExtAE2} is an identity since $\big(t,\omega_{s_1}(t),\phi_{s_2}(t),\omega_{p_1}(t), q(t,\omega_{s_1}(t),\omega_{p_1}(t))\big)\in L_{t_0}$ for all $t\in [t_0,\infty)$. Therefore, the function
\begin{equation*}
x(t)=\omega_{s_1}(t)+\phi_{s_2}(t)+\omega_{p_1}(t)+ q(t,\omega_{s_1}(t),\omega_{p_1}(t))
\end{equation*}
is a solution of the IVP \eqref{DAE}, \eqref{ini} on $[t_0,\infty)$. Since the initial point $(t_0,x_0)$ was chosen arbitrarily, the existence of a global solution is proved for each initial point $(t_0,x_0)\in L_{t_+}$, where $S_2x_0\in D_{s_2}$.

The chosen function $\phi_{s_2}\in C([t_0,\infty),D_{s_2})$ with the initial value $\phi_{s_2}(t_0)=S_2 x_0$, which can be regarded as a functional parameter, uniquely defines the component $S_2x(t)=\phi_{s_2}(t)$ of the solution $x(t)$.  If $\rank(\lambda A+B)=n\le m$, then $X_{s_2}=\{0\}$, $S_2=0$ and the component $S_2 x$  is absent.

Let us prove the uniqueness of the global solution $x(t)$ with the fixed component  $S_2x(t)=\phi_{s_2}(t)$. It follows from the above that the global solution $\omega(t)$ of the IVP \eqref{DAEsysExtDEeta}, \eqref{DEeta_ini} is unique on some interval $[t_0,\beta)$. Assume that the solution is not unique on $[t_0,\infty)$, then there exists a number $t_*\ge \beta$ and two different global solutions $\omega(t)$, $\hat{\omega}(t)$ with the common value $\omega_*=\omega(t_*)=\hat{\omega}(t_*)$.  Take $(t_*,\omega_*)$ as an initial point, then on some interval $[t_*,\beta_1)$ there must exist a unique solution of \eqref{DAEsysExtDEeta} with the initial value $\omega(t_*)=\omega_*$. Thus,   we have a contradiction. Hence, the solution $\omega(t)$ is unique, and hence the solution $x(t)$ is also unique.
 \end{proof}

A mapping $f(t,x)$ of a set $J\times D$, where $J$ is an interval in $\R$, $D\subset X$ and $X$ is a linear space, into a linear space $Y$ is said to \emph{satisfy locally a Lipschitz condition} (or to be \emph{locally Lipschitz continuous}) \emph{with respect to $x$ on} $J\times D$ if for each point $(t_*,x_*)\in J\times D$ there exists a neighborhood $U(t_*,x_*)$ and a constant $L\ge 0$ such that $\|f(t,x_1)-f(t,x_2)\|\le L\|x_1-x_2\|$ for any $(t,x_1),(t,x_2)\in U(t_*,x_*)$.

 \begin{theorem}\label{Th_SingGlobSol-Lipsch}
Let $f\in C([t_+,\infty)\times \Rn,\Rm)$ and $\lambda A+B$ be a singular pencil of operators such that its regular block $\lambda A_r+B_r$ from \eqref{penc} has the index not higher than 1. Assume that conditions \ref{SoglSing} and \ref{ExtensSing} of Theorem \ref{Th_SingGlobSol} hold and that condition \ref{InvSing} of Theorem \ref{Th_SingGlobSol} is replaced by the following:
\begin{enumerate}
\addtocounter{enumi}{1}
\item\label{InvSing-Lipsch} A function $f(t,x)$ satisfy locally a Lipschitz condition with respect to $x$ on $[t_+,\infty)\times \Rn$.\; For~any fixed $t_*$, ${x_*=x_{s_1}^*+x_{s_2}^*+x_{p_1}^*+x_{p_2}^*}$  such that ${(t_*,x_*)\in L_{t_+}}$ and $x_{s_2}^*\in D_{s_2}$, there exists a neighborhood $U_\delta(t_*,x_{s_1}^*,x_{s_2}^*,x_{p_1}^*)= U_{\delta_1}(t_*)\times U_{\delta_2}(x_{s_1}^*)\times U_{\delta_3}(x_{s_2}^*)\times U_{\delta_4}(x_{p_1}^*)$, where the neighborhood $U_{\delta_3}(x_{s_2}^*)\subseteq D_{s_2}$ can be closed, a neighborhood $U_\varepsilon(x_{p_2}^*)$ (numbers $\delta, \varepsilon>0$ depend on  the choice of $t_*$, $x_*$) 
    and an invertible operator $\Phi_{t_*,x_*}\in \mathrm{L}(X_2,Y_2)$ such that for each $(t,x_{s_1},x_{s_2},x_{p_1})\in U_\delta(t_*,x_{s_1}^*,x_{s_2}^*,x_{p_1}^*)$ and each $x_{p_2}^i\in U_\varepsilon(x_{p_2}^*)$, $i=1,2$, the mapping
     \begin{equation}\label{tildePsiSing}
    \widetilde{\Psi}(t,x_{s_1},x_{s_2},x_{p_1},x_{p_2})\!:=\! Q_2f(t,x_{s_1}+x_{s_2}+x_{p_1}+x_{p_2})-B\big|_{X_2}x_{p_2}\colon\! [t_+,\infty)\times X_{s_1}\times X_{s_2}\times X_1\times X_2\to Y_2
     \end{equation}
   satisfies the inequality
     \begin{equation}\label{ContractiveMapPsi}
   \|\widetilde{\Psi}(t,x_{s_1},x_{s_2},x_{p_1},x_{p_2}^1)- \widetilde{\Psi}(t,x_{s_1},x_{s_2},x_{p_1},x_{p_2}^2)-\Phi_{t_*,x_*} [x_{p_2}^1-x_{p_2}^2]\|\le c\|x_{p_2}^1-x_{p_2}^2\|,
     \end{equation}
    where $c=c(\delta,\varepsilon)$ is such that
   \begin{equation}\label{ContractiveConst}
    \lim\limits_{\delta,\,\varepsilon\to 0} c(\delta,\varepsilon)<\|\Phi_{t_*,x_*}^{-1}\|^{-1}.
   \end{equation}
 \end{enumerate}
Then for each initial point ${(t_0,x_0)\in L_{t_+}}$, where $S_2x_0\in D_{s_2}$, the initial value problem \eqref{DAE}, \eqref{ini} has a unique global (i.e., on $[t_0,\infty)$) solution $x(t)$ for which the choice of the function $\phi_{s_2}\in C([t_0,\infty),D_{s_2})$ with the initial value  $\phi_{s_2}(t_0)=S_2 x_0$ uniquely defines the component $S_2x(t)=\phi_{s_2}(t)$ when ${\rank(\lambda A+B)<n}$; when ${\rank(\lambda A+B)=n}$, the component $S_2 x$ is absent.
 \end{theorem}

 \begin{remark}
If a function $f(t,x)$ has the partial derivative $\frac{\partial}{\partial x} f\in C([t_+,\infty)\times \Rn, \mathrm{L}(\Rn,\Rm))$, then the function \eqref{tildePsiSing} has the continuous partial derivatives with respect to $x_{s_1}$, $x_{s_2}$, $x_{p_1}$, $x_{p_2}$ on $[t_+,\infty)\times X_{s_1}\times X_{s_2}\times X_1\times X_2$ and
 \begin{equation*}
\dfrac{\partial\widetilde{\Psi}}{\partial x_{p_2}}(t_*,x^*_{s_1},x^*_{s_2},x^*_{p_1},x^*_{p_2})=\Phi_{t_*,x_*},
 \end{equation*}
where the operator $\Phi_{t_*,x_*}$  is defined by \eqref{funcPhiSing}, $x_*=x_{s_1}^*+x_{s_2}^*+x_{p_1}^*+x_{p_2}^*$.
 \end{remark}
 \begin{corollary}\label{Propos-SingGlobSol}
If the conditions of Theorem~\ref{Th_SingGlobSol} are fulfilled, then the conditions of Theorem~\ref{Th_SingGlobSol-Lipsch} are also fulfilled.
 \end{corollary}
\begin{proof}[The proof of Corollary~\ref{Propos-SingGlobSol}]
Obviously, it follows from the existence of $\frac{\partial}{\partial x} f\in C([t_+,\infty)\times \Rn, \mathrm{L}(\Rn,\Rm))$ that $f(t,x)$ satisfy locally a Lipschitz condition with respect to $x$ on $[t_+,\infty)\times \Rn$. 
Take $\Phi_{t_*,x_*}$ defined by \eqref{funcPhiSing} as the operator $\Phi_{t_*,x_*}$ appearing in condition \ref{InvSing-Lipsch} of Theorem~\ref{Th_SingGlobSol-Lipsch}. Then $\Phi_{t_*,x_*}=\dfrac{\partial\widetilde{\Psi}}{\partial x_{p_2}}(t_*,x_{s_1}^*,x_{s_2}^*,x_{p_1}^*,x_{p_2}^*)$, where $x_*=x_{s_1}^*+x_{s_2}^*+x_{p_1}^*+x_{p_2}^*$, and there exists $\Phi_{t_*,x_*}^{-1}\in \mathrm{L}(Y_2,X_2)$  by virtue of condition \ref{InvSing} of Theorem~\ref{Th_SingGlobSol}. It is readily verified that  condition \ref{InvSing-Lipsch} of Theorem~\ref{Th_SingGlobSol-Lipsch}, where $\Phi_{t_*,x_*}$ is the operator \eqref{funcPhiSing}, is satisfied. The rest of the conditions of Theorems~\ref{Th_SingGlobSol} and Theorem~\ref{Th_SingGlobSol-Lipsch} coincide.
 \end{proof}

 \begin{proof}[The proof of Theorem~\ref{Th_SingGlobSol-Lipsch}]
As in the proof of Theorem~\ref{Th_SingGlobSol},  we consider the system \eqref{DAEsysExtDE1}--\eqref{DAEsysExtAE2} equivalent to the DAE \eqref{DAE}.

Note that from the representation $x=x_{s_1}+x_{s_2}+x_{p_1}+x_{p_2}$  \eqref{xsr} of any element $x\in \Rn$ it follows that for any norm $\|\cdot\|_\Rn$ in $\Rn=X_{s_1}\dot+ X_{s_2}\dot+X_1 \dot+X_2$ the inequality $\|x\|_\Rn\le \|x_{s_1}\|_\Rn+\|x_{s_2}\|_\Rn+ \|x_{p_1}\|_\Rn+\|x_{p_2}\|_\Rn$ holds. When proving the theorem, we use the ``maximal'' norm, i.e., we define the norm $\|\cdot\|$ in $X_{s_1}\dot+ X_{s_2}\dot+X_1 \dot+X_2$ as
$\|x\|=\|x_{s_1}\|+\|x_{s_2}\|+ \|x_{p_1}\|+\|x_{p_2}\|$, where we denote by $\|x_{s_1}\|=\|x_{s_1}\|_{X_{s_1}}$, $\|x_{s_2}\|=\|x_{s_2}\|_{X_{s_2}}$, $\|x_{p_1}\|=\|x_{p_1}\|_{X_1}$ and $\|x_{p_2}\|=\|x_{p_2}\|_{X_2}$ the norms of the components $x_{s_1}$, $x_{s_2}$, $x_{p_1}$ and $x_{p_2}$ in the subspaces $X_{s_1}$, $X_{s_2}$, $X_1$ and $X_2$, respectively. Taking into account the correspondence $x=x_{s_1}+x_{s_2}+x_{p_1}+x_{p_2}= (x_{s_1},x_{s_2},x_{p_1},x_{p_2})$ between $X_{s_1}\dot+ X_{s_2}\dot+X_1 \dot+X_2$ and $X_{s_1}\times X_{s_2}\times X_1\times X_2$ which is established in Section \ref{DAEsingReduce}, the norm $\|x\|$ of $x\in X_{s_1}\times X_{s_2}\times X_1\times X_2$ is defined in the same way and coincides with the above-defined norm of the corresponding element $x\in X_{s_1}\dot+ X_{s_2}\dot+X_1 \dot+X_2$.  Similarly, in $\R\times \Rn$ or $\R\times X_{s_1}\times X_{s_2}\times X_1\times X_2$ we use the norm $\|(t,x)\|=\|t\|+\|x_{s_1}\|+\|x_{s_2}\|+ \|x_{p_1}\|+\|x_{p_2}\|$.

Consider the equation \eqref{DAEsysAE1equiv}, that is, $\Psi(t,x_{s_1},x_{s_2},x_{p_1},x_{p_2})=0$ where $\Psi$ is defined by \eqref{funcPsiSing}. Recall that this equation is equivalent to the equation \eqref{DAEsysExtAE1equiv} and to \eqref{DAEsysExtAE1}, and that $B\big|_{X_2}=\EuScript B_2\big|_{X_2}=B_2$ (see Section \ref{BlockStruct}).  The mapping \eqref{tildePsiSing} can be represented as
$$
\widetilde{\Psi}(t,x_{s_1},x_{s_2},x_{p_1},x_{p_2})= Q_2\tilde{f}(t,x_{s_1},x_{s_2},x_{p_1},x_{p_2})-B_2x_{p_2}=B_2 \Psi(t,x_{s_1},x_{s_2},x_{p_1},x_{p_2}),
$$
and we can rewrite the equation \eqref{DAEsysAE1equiv} in the form
\begin{equation}\label{DAEsysAE1equiv2}
x_{p_2}=\mathcal{N}(t,x_{s_1},x_{s_2},x_{p_1},x_{p_2}), \text{ where }
\mathcal{N}(t,x_{s_1},x_{s_2},x_{p_1},x_{p_2})\!:= x_{p_2}-\Phi_{t_*,x_*}^{-1}\widetilde{\Psi}(t,x_{s_1},x_{s_2},x_{p_1},x_{p_2}).
\end{equation}
Recall that if $(t_*,x_*)\in L_{t_+}$, then $(t_*,x_*)$ satisfies \eqref{DAEsysAE1equiv}, i.e., $\Psi(t_*,x_{s_1}^*,x_{s_2}^*,x_{p_1}^*,x_{p_2}^*)=0$, where $x_{s_1}^*+x_{s_2}^*+x_{p_1}^*+x_{p_2}^*=x_*$.

We will prove the following lemma, which will be used in what follows.
 \begin{lemma}\label{Lemma-ImplicitFunc}
For any fixed elements $t_*\in [t_+,\infty)$, $x_{s_1}^*\in X_{s_1}$, $x_{s_2}^*\in D_{s_2}$, $x_{p_1}^*\in X_1$, $x_{p_2}^*\in X_2$  for which $(t_*,x_{s_1}^*+x_{s_2}^*+x_{p_1}^*+x_{p_2}^*)\in L_{t_+}$, there exist neighborhoods $U_r(t_*,x_{s_1}^*,x_{s_2}^*,x_{p_1}^*)$, $U_\rho(x_{p_2}^*)$ and a unique function  $x_{p_2}=\mu(t,x_{s_1},x_{s_2},x_{p_1})\in C(U_r(t_*,x_{s_1}^*,x_{s_2}^*,x_{p_1}^*),U_\rho(x_{p_2}^*))$ which satisfies the equality $\mu(t_*,x_{s_1}^*,x_{s_2}^*,x_{p_1}^*)=x_{p_2}^*$ and a Lipschitz condition with respect to $(x_{s_1},x_{s_2},x_{p_1})$ on $U_r(t_*,x_{s_1}^*,x_{s_2}^*,x_{p_1}^*)$ and is a solution of the equation \eqref{DAEsysAE1equiv} with respect to $x_{p_2}$, i.e., $\Psi(t,x_{s_1},x_{s_2},x_{p_1},\mu(t,x_{s_1},x_{s_2},x_{p_1}))=0$, for all $(t,x_{s_1},x_{s_2},x_{p_1})\in U_r(t_*,x_{s_1}^*,x_{s_2}^*,x_{p_1}^*$) \,(notice that numbers $r, \rho>0$ and the function $\mu$ depend on the choice of $t_*$, $x_{s_1}^*,x_{s_2}^*,x_{p_1}^*,x_{p_2}^*$).
 \end{lemma}
 \begin{proof}
It follows from condition \ref{InvSing-Lipsch} that for any fixed point $(t_*,x_{s_1}^*+x_{s_2}^*+x_{p_1}^*+x_{p_2}^*)\in L_{t_+}$ for which $x_{s_2}^*\in D_{s_2}$, there exist closed neighborhoods $\overline{U_{\tilde{\delta}}}(t_*,x_{s_1}^*,x_{s_2}^*,x_{p_1}^*)\subset U_\delta(t_*,x_{s_1}^*,x_{s_2}^*,x_{p_1}^*)$ and $\overline{U_{\tilde{\varepsilon}}}(x_{p_2}^*)\subset U_\varepsilon(x_{p_2}^*)$  such that $\mathcal{N}$ is a contractive mapping with respect to $x_{p_2}$ (uniformly in $(t,x_{s_1},x_{s_2},x_{p_1})$) on $\overline{U_{\tilde{\delta}}}(t_*,x_{s_1}^*,x_{s_2}^*,x_{p_1}^*)\times \overline{U_{\tilde{\varepsilon}}}(x_{p_2}^*)$, i.e.,
 \begin{equation}\label{ContractiveMapN}
\|\mathcal{N}(t,x_{s_1},x_{s_2},x_{p_1},x_{p_2}^1)- \mathcal{N}(t,x_{s_1},x_{s_2},x_{p_1},x_{p_2}^2)\|\le l\|x_{p_2}^1-x_{p_2}^2\|,\quad \text{where $l<1$ is a constant,}
 \end{equation}
for every $(t,x_{s_1},x_{s_2},x_{p_1})\in \overline{U_{\tilde{\delta}}}(t_*,x_{s_1}^*,x_{s_2}^*,x_{p_1}^*)$,  $x_{p_2}^i\in \overline{U_{\tilde{\varepsilon}}}(x_{p_2}^*)$, $i=1,2$.
Indeed, due to \eqref{ContractiveMapPsi} and \eqref{ContractiveConst}, there exist numbers $\tilde{\delta}\in (0,\delta)$ and $\tilde{\varepsilon}\in (0,\varepsilon)$ such that for every $\alpha=(t,x_{s_1},x_{s_2},x_{p_1})\in \overline{U_{\tilde{\delta}}}(t_*,x_{s_1}^*,x_{s_2}^*,x_{p_1}^*)$, $x_{p_2}^i\in \overline{U_{\tilde{\varepsilon}}}(x_{p_2}^*)$, $i=1,2$, the following holds:
$\|\mathcal{N}(\alpha,x_{p_2}^1)- \mathcal{N}(\alpha,x_{p_2}^2)\|= 
\|x_{p_2}^1-x_{p_2}^2-\Phi_{t_*,x_*}^{-1} [\widetilde{\Psi}(\alpha,x_{p_2}^1)-\widetilde{\Psi}(\alpha,x_{p_2}^2)\|\le \|\Phi_{t_*,x_*}^{-1}\|c(\delta_0,\varepsilon_0)\|x_{p_2}^1-x_{p_2}^2\|$, where $\|\Phi_{t_*,x_*}^{-1}\|c(\delta_0,\varepsilon_0)\le l<1$ for every $\delta_0\in (0,\tilde{\delta}]$, $\varepsilon_0\in (0,\tilde{\varepsilon}]$. This implies \eqref{ContractiveMapN}.

Choose a point $(t_*,x_*)=(t_*,x_{s_1}^*+x_{s_2}^*+x_{p_1}^*+x_{p_2}^*)$ such that $x_{s_2}^*\in D_{s_2}$ and $(t_*,x_*)\in L_{t_+}$, and fix it. As above, denote
$$
\alpha=(t,x_{s_1},x_{s_2},x_{p_1}),
$$
then $\mathcal{N}(t,x_{s_1},x_{s_2},x_{p_1},x_{p_2})=\mathcal{N}(\alpha,x_{p_2})$; also, we denote  $\alpha_*=(t_*,x_{s_1}^*,x_{s_2}^*,x_{p_1}^*)$. 
Since $\widetilde{\Psi}$ is continuous on $[t_+,\infty)\times X_{s_1}\times X_{s_2}\times X_1\times X_2$, then $\widetilde{\Psi}(\alpha,x_{p_2}^*)\to \widetilde{\Psi}(\alpha_*,x_{p_2}^*)=0$ as $\alpha\to \alpha_*$, and therefore there exists a number $\delta_*\in (0,\tilde{\delta}]$ such that $\|\Phi_{t_*,x_*}^{-1}\|\, \|\widetilde{\Psi}(\alpha,x_{p_2}^*)\|\le (1-l)\tilde{\varepsilon}$ \,(where $l$ is the constant from \eqref{ContractiveMapN}) for every $\alpha\in \overline{U_{\delta_*}}(\alpha_*)$. Hence, for each (fixed) $\alpha\in \overline{U_{\delta_*}}(\alpha_*)$ and every $x_{p_2}\in\overline{U_{\tilde{\varepsilon}}}(x_{p_2}^*)$ we have $\|\mathcal{N}(\alpha,x_{p_2})-x_{p_2}^*\|\le \|\mathcal{N}(\alpha,x_{p_2})-\mathcal{N}(\alpha,x_{p_2}^*)\|+ \|\Phi_{t_*,x_*}^{-1}\|\, \|\widetilde{\Psi}(\alpha,x_{p_2}^*)\| \le l \tilde{\varepsilon} +(1-l)\tilde{\varepsilon}=\tilde{\varepsilon}$. 
Thus, $\mathcal{N}(\alpha,x_{p_2})$ map $\overline{U_{\tilde{\varepsilon}}}(x_{p_2}^*)$ into itself for each $\alpha\in \overline{U_{\delta_*}}(\alpha_*)$.

From the foregoing it follows that, by the fixed point theorems (see, e.g., \cite[Theorems 46, $46_2$]{Schwartz1}), the mapping $\mathcal{N}(\alpha,x_{p_2})$ as a function of $x_{p_2}$, depending on the parameter $\alpha=(t,x_{s_1},x_{s_2},x_{p_1})$, has a unique fixed point $\mu_\alpha=\mu(\alpha)$ (i.e., $\mathcal{N}(\alpha,\mu(\alpha))=\mu(\alpha)$)  in $\overline{U_{\tilde{\varepsilon}}}(x_{p_2}^*)$ for each $\alpha\in \overline{U_{\delta_*}}(\alpha_*)= \overline{U_{\delta_*}}(t_*,x_{s_1}^*,x_{s_2}^*,x_{p_1}^*)$, which satisfies the equality $\mu(\alpha_*)=x_{p_2}^*$, and $\mu(\alpha)$ depends continuously on $\alpha$. The continuity of the function $\mu\colon \overline{U_{\delta_*}}(\alpha_*)\to \overline{U_{\tilde{\varepsilon}}}(x_{p_2}^*)$  is proved in the same way as in \cite[Theorem $46_2$]{Schwartz1}.

Let us prove that $\mu(\alpha)=\mu(t,x_{s_1},x_{s_2},x_{p_1})$ satisfies a Lipschitz condition with respect to $(x_{s_1},x_{s_2},x_{p_1})$ on $U_r(t_*,x_{s_1}^*,x_{s_2}^*,x_{p_1}^*)$, where $r$ is specified below. Recall that here we use the notation $\tilde{f}(t,x_{s_1},x_{s_2},x_{p_1},x_{p_2})=f(t,x)$ introduced in the proof of Theorem~\ref{Th_SingGlobSol}.
Since $f(t,x)$ satisfy locally a Lipschitz condition with respect to $x$ on $[t_+,\infty)\times \Rn$, then there exists a neighborhood $U(t_*,x_{s_1}^*,x_{s_2}^*,x_{p_1}^*,x_{p_2}^*)$ and a constant $L\ge 0$ such that
 \begin{multline}\label{LocLipsch_f}
\|\tilde{f}(t,x_{s_1}^1,x_{s_2}^1,x_{p_1}^1,x_{p_2}^1)- \tilde{f}(t,x_{s_1}^2,x_{s_2}^2,x_{p_1}^2,x_{p_2}^2)\|\le L\|(x_{s_1}^1,x_{s_2}^1,x_{p_1}^1,x_{p_2}^1)- (x_{s_1}^2,x_{s_2}^2,x_{p_1}^2,x_{p_2}^2)\|= \\
=L\big(\|x_{s_1}^1-x_{s_1}^2\|+ \|x_{s_2}^1-x_{s_2}^2\|+ \|x_{p_1}^1-x_{p_1}^2\|+\|x_{p_2}^1-x_{p_2}^2\|\big)
 \end{multline}
for any $(t,x_{s_1}^i,x_{s_2}^i,x_{p_1}^i,x_{p_2}^i)\in U(t_*,x_{s_1}^*,x_{s_2}^*,x_{p_1}^*,x_{p_2}^*)$, $i=1,2$. Choose numbers ${r\in (0,\delta_*]}$ and ${\rho\in (0,\tilde{\varepsilon}]}$ so that $U_r(t_*,x_{s_1}^*,x_{s_2}^*,x_{p_1}^*)\times U_\rho(x_{p_2}^*)\subseteq U(t_*,x_{s_1}^*,x_{s_2}^*,x_{p_1}^*,x_{p_2}^*)$ and $\mu\colon U_r(t_*,x_{s_1}^*,x_{s_2}^*,x_{p_1}^*)\to U_\rho(x_{p_2}^*)$.
Then, carrying out certain transformations and using \eqref{ContractiveMapN}, \eqref{LocLipsch_f}, we obtain that for any $(t,x_{s_1}^i,x_{s_2}^i,x_{p_1}^i)\in U_r(t_*,x_{s_1}^*,x_{s_2}^*,x_{p_1}^*)$, ${i=1,2}$, the following holds:
 \begin{multline*}
\|\mu(t,x_{s_1}^1,x_{s_2}^1,x_{p_1}^1)-\mu(t,x_{s_1}^2,x_{s_2}^2,x_{p_1}^2)\|\le
\widehat{L}\, \|(x_{s_1}^1,x_{s_2}^1,x_{p_1}^1)- (x_{s_1}^2,x_{s_2}^2,x_{p_1}^2)\|= \\
=\widehat{L} \big(\|x_{s_1}^1-x_{s_1}^2\|+ \|x_{s_2}^1-x_{s_2}^2\|+ \|x_{p_1}^1-x_{p_1}^2\|\big),\qquad \widehat{L}=L\|\Phi_{t_*,x_*}^{-1}\|\,\|Q_2\|/(1-l)\ge 0.
 %%\widehat{L}=\frac{L\|\Phi_{t_*,x_*}^{-1}\|\,\|Q_2\|}{(1-l)}\ge 0.
 \end{multline*}

Thus, since the equations \eqref{DAEsysAE1equiv} and \eqref{DAEsysAE1equiv2} are equivalent, the lemma is proved.
 \end{proof}

Due to condition \ref{SoglSing} of Theorem \ref{Th_SingGlobSol}, for any fixed $(t_*,x_{s_1}^*,x_{s_2}^*,x_{p_1}^*)\in [t_+,\infty)\times X_{s_1}\times D_{s_2}\times X_1$ there exists a unique $x_{p_2}^*\in X_2$ such that $(t_*,x_{s_1}^*+x_{s_2}^*+x_{p_1}^*+x_{p_2}^*)\in L_{t_+}$. Further, it follows from Lemma \ref{Lemma-ImplicitFunc} that in some neighborhood of each (fixed) $(t_*,x_{s_1}^*,x_{s_2}^*,x_{p_1}^*)\in [t_+,\infty)\times X_{s_1}\times D_{s_2}\times X_1$ there exists a unique solution $x_{p_2}=\mu_{t_*,x_{s_1}^*,x_{s_2}^*,x_{p_1}^*}(t,x_{s_1},x_{s_2},x_{p_1})$ of the equation \eqref{DAEsysAE1equiv}, and this solution is continuous in $(t,x_{s_1},x_{s_2},x_{p_1})$, satisfies a Lipschitz condition with respect to $(x_{s_1},x_{s_2},x_{p_1})$ and the equality  $\mu_{t_*,x_{s_1}^*,x_{s_2}^*,x_{p_1}^*}(t_*,x_{s_1}^*,x_{s_2}^*,x_{p_1}^*)=x_{p_2}^* \in D_{p_2}$,  where the set $D_{p_2}\subset  X_2$ is such that $[t_+,\infty)\times (X_{s_1}\dot+ D_{s_2}\dot+ X_1\dot+ D_{p_2})\subset L_{t_+}$. As in the proof of Theorem~\ref{Th_SingGlobSol}, we introduce the function
$$
\eta\colon [t_+,\infty)\times X_{s_1}\times D_{s_2}\times X_1\to D_{p_2}
$$
and define  $\eta(t,x_{s_1},x_{s_2},x_{p_1})= \mu_{t_*,x_{s_1}^*,x_{s_2}^*,x_{p_1}^*}(t,x_{s_1},x_{s_2},x_{p_1})$ at \,$(t,x_{s_1},x_{s_2},x_{p_1})=(t_*,x_{s_1}^*,x_{s_2}^*,x_{p_1}^*)$ for each $(t_*,x_{s_1}^*,x_{s_2}^*,x_{p_1}^*)\in [t_+,\infty)\times X_{s_1}\times D_{s_2}\times X_1$. Then the function $x_{p_2}=\eta(t,x_{s_1},x_{s_2},x_{p_1})$ is continuous in $(t,x_{s_1},x_{s_2},x_{p_1})$, satisfy locally a Lipschitz condition with respect to $(x_{s_1},x_{s_2},x_{p_1})$ on $[t_+,\infty)\times X_{s_1}\times D_{s_2}\times X_1$ and is a unique solution of the equation \eqref{DAEsysAE1equiv}, i.e., $\Psi(t,x_{s_1},x_{s_2},x_{p_1},\eta(t,x_{s_1},x_{s_2},x_{p_1}))=0$ for all $(t,x_{s_1},x_{s_2},x_{p_1})\in [t_+,\infty)\times X_{s_1}\times D_{s_2}\times X_1$.
The uniqueness of $\eta$ is proved in the same way as the uniqueness of the function $\eta$ in the proof of Theorem~\ref{Th_SingGlobSol}. Note that the introduced function $\eta$ is also a unique solution of the equation \eqref{DAEsysExtAE1} with respect to $x_{p_2}$ since it is equivalent to \eqref{DAEsysAE1equiv}.

Choose any initial point $(t_0,x_0)\in L_{t_+}$, where $S_2x_0\in D_{s_2}$, and any function $\phi_{s_2}\in C([t_0,\infty),D_{s_2})$ satisfying the condition $\phi_{s_2}(t_0)=S_2 x_0$. We substitute the function $x_{s_2}=\phi_{s_2}(t)$ into $\eta$, denote $q(t,x_{s_1},x_{p_1})=\eta(t,x_{s_1},\phi_{s_2}(t),x_{p_1})$, and then we substitute the functions $x_{s_2}=\phi_{s_2}(t)$ and $x_{p_2}=q(t,x_{s_1},x_{p_1})$ in \eqref{DAEsysExtDE1}, \eqref{DAEsysExtDE2}.
The obtained system we write in the form \eqref{DAEsysExtDEeta}. Due to the properties of $\tilde{f}(t,x_{s_1},x_{s_2},x_{p_1},x_{p_2})$, $\eta(t,x_{s_1},x_{s_2},x_{p_1})$ (and, accordingly, $q(t,x_{s_1},x_{p_1})$) and $\phi_{s_2}(t)$, the function $\widetilde{\Upsilon}(t,\omega)$ is continuous in $(t,\omega)$ and satisfy locally a Lipschitz condition with respect to $\omega$ on $[t_0,\infty)\times X_{s_1}\times X_1$. Consequently, there exists a unique solution $\omega=\omega(t)$ of \eqref{DAEsysExtDEeta} on some interval $[t_0,\beta)$ which satisfies the initial condition \eqref{DEeta_ini}.

The subsequent proof coincides with the proof of Theorem~\ref{Th_SingGlobSol} (see the part of the proof after~\eqref{DEeta_ini}).
 \end{proof}

 \section{Lagrange stability of singular semilinear DAEs}\label{SectLagrSt}

The equation \eqref{DAE} is called \emph{Lagrange stable for the initial point $(t_0,x_0)$} if the solution of the IVP \eqref{DAE}, \eqref{ini} is Lagrange stable (see the definition in Section \ref{Preliminaries}) for this initial point.

The \emph{equation \eqref{DAE}} is called \emph{Lagrange stable} if each solution of the IVP \eqref{DAE}, \eqref{ini} is Lagrange stable (i.e., the equation is Lagrange stable for each consistent initial point).

 \begin{theorem}\label{Th_SingUstLagr} %[об устойчивости по Лагранжу ДАУ]
Let $f\in C([t_+,\infty)\times \Rn,\Rm)$ and $\lambda A+B$ be a singular pencil of operators such that its regular block $\lambda A_r+B_r$ from \eqref{penc} has the index not higher than 1. Assume that condition \ref{SoglSing} of Theorem \ref{Th_SingGlobSol} holds and condition \ref{InvSing} of Theorem \ref{Th_SingGlobSol} or condition \ref{InvSing-Lipsch} of Theorem \ref{Th_SingGlobSol-Lipsch} holds. Let the following conditions be satisfied:
\begin{enumerate}
\addtocounter{enumi}{2}
\item\label{LagrSing} There exists a number $R>0$, a positive definite function  $V\in C^1([t_+,\infty)\times D_{s_1}\times D_{p_1},\R)$, where a set $D_{s_1}\times D_{p_1}\subset X_{s_1}\times X_1$ is such that $D_{s_1}\times D_{p_1}\supset \{\|(x_{s_1},x_{p_1})\|\ge R\}$, and a function $\chi\in C([t_+,\infty)\times (0,\infty),\R)$ such that:
      %the following holds:
  \begin{enumerate}[label={\upshape(\alph*)},ref={\upshape(\alph*)},topsep=1pt]
  \item\label{LagrSing1} ${V(t,x_{s_1},x_{p_1})\to\infty}$  uniformly in $t$ on $[t_+,\infty)$ as ${\|(x_{s_1},x_{p_1})\|\to\infty}$;

  \item\label{LagrSing2} for all $(t,x_{s_1}+x_{s_2}+x_{p_1}+x_{p_2})\in L_{t_+}$, for which $x_{s_2}\in D_{s_2}$ and ${\|(x_{s_1},x_{p_1})\|\ge R}$, the inequality \eqref{LagrDAEsing} is satisfied (i.e., condition \ref{ExtensSing2} of Theorem~\ref{Th_SingGlobSol} holds);

  \item\label{UstLagr} the differential inequality \eqref{L1v}, i.e., $dv/dt\le \chi(t,v)$  \,($t\in [t_+,\infty)$), does not have unbounded positive solutions for $t\in [t_+,\infty)$.
  \end{enumerate}
\end{enumerate}
Then for each initial point ${(t_0,x_0)\in L_{t_+}}$, where $S_2x_0\in D_{s_2}$, the initial value problem \eqref{DAE}, \eqref{ini} has a unique global solution $x(t)$ for which the choice of the function $\phi_{s_2}\in C([t_0,\infty),D_{s_2})$  with the initial value $\phi_{s_2}(t_0)=S_2 x_0$ uniquely defines the component  $S_2x(t)=\phi_{s_2}(t)$ when $\rank(\lambda A+B)<n$.

Let, in addition to the above conditions, the following conditions also hold:
 \begin{enumerate}
\addtocounter{enumi}{4}   
 \item\label{LagrAs} For all $(t,x_{s_1}+x_{s_2}+x_{p_1}+x_{p_2})\in L_{t_+}$, for which $x_{s_2}\in D_{s_2}$ and $\|x_{s_1}+x_{s_2}+x_{p_1}\|\le M<\infty$\, ($M$ is an arbitrary constant), the inequality
     $$
     \|x_{p_2}\|\le K_M<\infty
     $$
     or the inequality
     $$
     \|Q_2 f(t,x_{s_1}+x_{s_2}+x_{p_1}+x_{p_2})\|\le K_M<\infty,
     $$
     where $K_M=K(M)$ is some constant, is satisfied.

  \item\label{LagrBs} $\|F_2 f(t,x)\|<+\infty$ for all $(t,x)\in L_{t_+}$ such that $S_2x\in D_{s_2}$ and $\|x\|\le C<\infty$ \,($C$ is an arbitrary constant).
 \end{enumerate}  
Then, for the initial points $(t_0,x_0)\in L_{t_+}$ where $S_2x_0\in D_{s_2}$ and any function $\phi_{s_2}\in C([t_0,\infty),D_{s_2})$ satisfying the relations $\phi_{s_2}(t_0)=S_2 x_0$ and $\sup\limits_{t\in [t_0,\infty)}\|\phi_{s_2}(t)\|<+\infty$, the equation \eqref{DAE}, where ${S_2x=\phi_{s_2}(t)}$, is Lagrange stable; when $\rank(\lambda A+B)=n<m$, the component $S_2 x$ is absent.
 \end{theorem}

 \begin{remark}\label{RemUstLagrGlobSolv}
If condition~\ref{LagrSing} of Theorem~\ref{Th_SingUstLagr} holds, then condition~\ref{ExtensSing} of Theorem~\ref{Th_SingGlobSol} holds.  This is easily verified since conditions \ref{LagrSing1} and \ref{UstLagr} of Theorem~\ref{Th_SingUstLagr} imply conditions  \ref{ExtensSing1} and \ref{GlobSolv}  of Theorem~\ref{Th_SingGlobSol}, respectively, and conditions \ref{LagrSing2} of Theorem~\ref{Th_SingUstLagr} and \ref{ExtensSing2} of Theorem~\ref{Th_SingGlobSol} coincide. Note that condition~\ref{LagrSing} of Theorem~\ref{Th_SingGlobSol} must also hold for Theorem \ref{Th_SingGlobSol-Lipsch}.
 \end{remark}

 \begin{proof}
We will carry out the proof, assuming that condition \ref{InvSing} of Theorem \ref{Th_SingGlobSol} holds. If we replace it by condition \ref{InvSing-Lipsch} of Theorem \ref{Th_SingGlobSol-Lipsch}, then in the proof of the present theorem it will be necessary to replace ``Theorem \ref{Th_SingGlobSol}'' by  ``Theorem \ref{Th_SingGlobSol-Lipsch}''.

Considering Remark \ref{RemUstLagrGlobSolv}, we conclude that all conditions of Theorem \ref{Th_SingGlobSol} hold. Consequently, for an arbitrary initial point $(t_0,x_0)\in L_{t_+}$, where $S_2x_0\in D_{s_2}$, there exists a unique solution $x(t)$ of the IVP \eqref{DAE}, \eqref{ini} on $[t_0,\infty)$, such that $S_2x(t)=\phi_{s_2}(t)$ where $\phi_{s_2}\in C([t_0,\infty),D_{s_2})$ is some chosen function with the initial value $\phi_{s_2}(t_0)=S_2 x_0$. Thus, the existence of a global solution of the IVP \eqref{DAE}, \eqref{ini} is proved.

Let us prove the Lagrange stability. As shown in the proof of Theorem  \ref{Th_SingGlobSol}, the solution of the IVP  \eqref{DAE}, \eqref{ini} can be represented in the form
$$
x(t)=\omega_{s_1}(t)+\phi_{s_2}(t)+\omega_{p_1}(t)+ q(t,\omega_{s_1}(t),\omega_{p_1}(t)).
$$
It is assumed that the function $\phi_{s_2}(t)$ defining the component $S_2x(t)$ of the solution $x(t)$ was chosen so that $\sup\limits_{t\in [t_0,\infty)}\|\phi_{s_2}(t)\|<+\infty$.  This is fulfilled due to the requirements of the present theorem and, obviously, does not affect the proof of Theorem \ref{Th_SingGlobSol}. It follows from condition \ref{LagrSing} that the derivative of the function $V$ along the trajectories of the equation \eqref{DAEsysExtDEeta} satisfies the inequality \eqref{L1dVdtSing} for all $t\ge t_0$, $\|\omega\|\ge R$, and the differential inequality \eqref{L1v} does not have unbounded positive solutions for $t\in [t_+,\infty)$.  Then by \cite[Chapter~IV, Theorem XV]{LaSal-Lef} the equation \eqref{DAEsysExtDEeta} is Lagrange stable. Consequently, $\sup\limits_{t\in[t_0,\infty)}\|\omega(t)\|<+\infty$. Hence there exists a constant $M>0$ such that
\begin{equation}\label{LAsing}
\|\omega_{s_1}(t)+\phi_{s_2}(t)+\omega_{p_1}(t)\|\le M,\qquad  t\in [t_0,\infty).
\end{equation}

Denote
$$
u(t):= q(t,\omega_{s_1}(t),\omega_{p_1}(t))
$$
and recall that the function $q(t,\omega_{s_1}(t),\omega_{p_1}(t))= \eta(t,\omega_{s_1}(t),\phi_{s_2}(t),\omega_{p_1}(t))$ is a solution of the equation \eqref{DAEsysExtAE1} (as well as the equivalent equation  \eqref{DAEsysAE1equiv}) with respect to the variable $x_{p_2}$. Therefore,
\begin{equation}\label{ExtAE1equiv2}
u(t)=\Es B_2^{(-1)}Q_2 f(t,\omega_{s_1}(t)+\phi_{s_2}(t)+\omega_{p_1}(t)+u(t)).
\end{equation}
Then from \eqref{LAsing}, condition \ref{LagrAs} and the boundedness of the norm of the operator $\Es B_2^{(-1)}\in \mathrm{L}(\Rm,\Rn)$ it follows that there exists a constant $K_{M}=K(M)$ (depending on the constant $M$, in general) such that for all $t\in [t_0,\infty)$ the following estimate holds:
\begin{equation}\label{Bound}
\|u(t)\|\le K_{M}.
\end{equation}

It follows from the above that $\|x(t)\|\le M+K_{M}<\infty$ for all $t\in [t_0,\infty)$, i.e., the solution $x(t)$ is bounded on $[t_0,\infty)$ and, therefore, is Lagrange stable. Condition \ref{LagrBs} ensures the correctness of the equality \eqref{DAEsysExtAE2}, which is equivalent to the equality $F_2Bx(t)=F_2f(t,x(t))$.  Thus, the Lagrange stability of the DAE \eqref{DAE} for the initial points  $(t_0,x_0)\in L_{t_+}$, where $S_2x_0\in D_{s_2}$, and an arbitrary function $\phi_{s_2}\in C([t_0,\infty),D_{s_2})$ satisfying the relations $\phi_{s_2}(t_0)=S_2 x_0$ and $\sup\limits_{t\in [t_0,\infty)}\|\phi_{s_2}(t)\|<+\infty$ is proved.
 \end{proof}

 \section{Lagrange instability of singular semilinear DAEs (the blow-up of solutions in finite time)}\label{SectLagrUnst}

The equation \eqref{DAE} is called \emph{Lagrange unstable for the initial point $(t_0,x_0)$} if the solution of the IVP \eqref{DAE}, \eqref{ini} is Lagrange unstable (see the definition in Section \ref{Preliminaries}) for this initial point.

The \emph{equation \eqref{DAE}} is called \emph{Lagrange unstable} if each solution of the IVP \eqref{DAE}, \eqref{ini} is Lagrange unstable (i.e., the equation is Lagrange unstable for each consistent initial point).

 \begin{theorem}\label{Th_singNeLagr}
Let $f\in C([t_+,\infty)\times \Rn,\Rm)$ and $\lambda A+B$ be a singular pencil of operators such that its regular block $\lambda A_r+B_r$ from \eqref{penc} has the index not higher than 1. Assume that condition \ref{SoglSing} of Theorem \ref{Th_SingGlobSol} holds and condition \ref{InvSing} of Theorem \ref{Th_SingGlobSol} or condition \ref{InvSing-Lipsch} of Theorem \ref{Th_SingGlobSol-Lipsch} holds. Let the following conditions be satisfied:
\begin{enumerate}%[3)]
\addtocounter{enumi}{2} 
\item\label{ExistOmegaSing} There exists a region  $\Omega \subset  X_{s_1}\times X_1$ such that $0\not\in\Omega$ and the component $(S_1+P_1)x(t)$ of each existing solution $x(t)$ with the initial point  $(t_0,x_0)\in L_{t_+}$, where $(S_1x_0,P_1x_0)\in\Omega$ and $S_2x_0\in D_{s_2}$, remains all the time in $\Omega$.

\item\label{NeUstLagrSing} There exists a positive definite function $V\in C^1([t_+,\infty)\times \Omega,\R)$ and a function $\chi\in C([t_+,\infty)\times (0,\infty),\R)$ such that: 
    \begin{enumerate}[label={\upshape(\alph*)},ref={\upshape(\alph*)},topsep=1pt]
   \item for all $(t,x_{s_1}+x_{s_2}+x_{p_1}+x_{p_2})\in L_{t_+}$, for which  $x_{s_2}\in D_{s_2}$ and $(x_{s_1},x_{p_1})\in\Omega$, the inequality
   \begin{equation}\label{NeLagrDAESing}
    V'_{\eqref{DAEsysExtDE1},\eqref{DAEsysExtDE2}}(t,x_{s_1},x_{p_1})\ge \chi\big(t,V(t,x_{s_1},x_{p_1})\big),
   \end{equation}
   where $V'_{\eqref{DAEsysExtDE1},\eqref{DAEsysExtDE2}}(t,x_{s_1},x_{p_1})$ has the form \eqref{dVDAEsing}, is satisfied;

   \item\label{NeustLagr} the differential inequality  \eqref{L2v}, i.e., $dv/dt\ge \chi(t,v)$  \,($t\in [t_+,\infty)$), does not have global positive solutions.
   \end{enumerate}
\end{enumerate}
Then for each initial point $(t_0,x_0)\in L_{t_+}$, where $S_2x_0\in D_{s_2}$ and $(S_1x_0,P_1x_0)\in\Omega$, the initial value problem \eqref{DAE}, \eqref{ini} has a unique solution $x(t)$ for which the choice of the function  $\phi_{s_2}\in C([t_0,\infty),D_{s_2})$  with the initial value $\phi_{s_2}(t_0)=S_2 x_0$ uniquely defines the component $S_2x(t)=\phi_{s_2}(t)$  when $\rank(\lambda A+B)<n$ (when $\rank(\lambda A+B)=n<m$, the component $S_2 x$ is absent), and this solution is Lagrange unstable (has a finite escape time).
 \end{theorem}

 \begin{proof}
It is proved in the same way as in Theorem~\ref{Th_SingGlobSol} (or Theorem \ref{Th_SingGlobSol-Lipsch}) that there exists a unique solution $\omega(t)$ of the equation \eqref{DAEsysExtDEeta} on some interval $[t_0,\beta)$ which satisfies the initial condition  \eqref{DEeta_ini}.  Further, it follows from the proof of  Theorem~\ref{Th_SingGlobSol} that there exists a unique solution $x(t)=\omega_{s_1}(t)+ \phi_{s_2}(t)+ \omega_{p_1}(t)+ q(t,\omega_{s_1}(t),\omega_{p_1}(t))$ of the IVP \eqref{DAE}, \eqref{ini} on $[t_0,\beta)$.

According to condition \ref{ExistOmegaSing}, there exists a region $\Omega\subset  X_{s_1}\times X_1$ such that $0\not\in\Omega$ and the component $(S_1+P_1)x(t)$ of each existing solution $x(t)$ with the initial point $(t_0,x_0)\in L_{t_+}$, for which $(S_1x_0,P_1x_0)\in\Omega$ and $S_2x_0\in D_{s_2}$, remains all the time in $\Omega$. Further, it will be assumed that the initial point $(t_0,x_0)$ for the solution mentioned above has been chosen so that condition \ref{ExistOmegaSing} is satisfied.  Then the initial value $\omega_0=(x_{s_1,0}^\T,x_{p_1,0}^\T)^\T$ from the condition \eqref{DEeta_ini} belongs to $\Omega$ and, therefore, the solution $\omega(t)$ of the equation \eqref{DAEsysExtDEeta} remains all the time in  $\Omega$. In addition, $\omega=0 \not\in\Omega$.

By virtue of condition \ref{NeUstLagrSing}, for all $t\ge t_0$, $\omega(t)\in \Omega$ the inequality
 \begin{equation*}
V'_{\eqref{DAEsysExtDEeta}}(t,\omega)=\dfrac{\partial V}{\partial t}(t,\omega)+\dfrac{\partial V}{\partial \omega}(t,\omega)\cdot \widetilde{\Upsilon}(t,\omega)\ge \chi\big(t,V(t,\omega)\big)
 \end{equation*}
holds, and the inequality  \eqref{L2v}, $t\ge t_0$, does not have global positive solutions. Hence, using the theorem \cite[Chapter~IV, Theorem~XIV]{LaSal-Lef}, we obtain that the solution $\omega(t)$ has a finite escape time, i.e., it is defined on some finite interval $[t_0,T)$ and $\lim\limits_{t\to T-0}\|\omega(t)\|= +\infty$.  Consequently, the solution $x(t)=\omega_{s_1}(t)+\phi_{s_2}(t)+\omega_{p_1}(t)+ q(t,\omega_{s_1}(t),\omega_{p_1}(t))$ also has a finite escape time and, accordingly, is Lagrange unstable.

The uniqueness of the solution $x(t)$ can be proved as in the proof of Theorem \ref{Th_SingGlobSol}.
 \end{proof}

The statement of Theorem~\ref{Th_singNeLagr} means that the equation \eqref{DAE} is Lagrange unstable for the initial points $(t_0,x_0)\in L_{t_+}$, for which $S_2x_0\in D_{s_2}$ and $(S_1x_0,P_1x_0)\in\Omega$.

 \section{Dissipativity (ultimate boundedness) of singular semilinear DAEs}\label{SectDissipat}

Solutions of an equation are called \emph{ultimately bounded} if there exists a constant $C>0$ (not depending on the choice of initial values) and for each solution $x(t)$ with initial values $t_0$, $x_0$ there exists a number $\tau=\tau(t_0,x_0)\ge t_0$ such that $\|x(t)\|<C$ for all $t\in [t_0+\tau,\infty)$. If at the same time the number $\tau$ does not depend on the choice of $t_0$ (i.e., $\tau=\tau(x_0)$), then the solutions are called \emph{uniformly ultimately bounded}.

The DAE \eqref{DAE} is called \emph{ultimately bounded} (\emph{uniformly ultimately bounded}, respectively) or \emph{dissipative} (\emph{uniformly dissipative}, respectively) if for any consistent initial point $(t_0,x_0)$ there exists a global solution of the initial value problem \eqref{DAE}, \eqref{ini} and all solutions are ultimately bounded (uniformly ultimately bounded, respectively).

 \smallskip
An operator function $H\colon J\to \mathrm{L}(X)$, where $X$ is a finite-dimensional linear or Hilbert space and $J\subseteq \R$ is an interval, is called \emph{self-adjoint} if the operator $H(t)$ is self-adjoint (i.e., self-adjoint for every $t\in J$).

A self-adjoint operator $H(t)\in \mathrm{L}(X)$, where $t\in J$, is called \emph{positive} if $(H(t)x,x)>0$ for all  $t\in J$ and $x\ne 0$, and  \emph{positive definite} or \emph{uniformly positive} if there exists a constant $c>0$ such that $(H(t)x,x)\ge c\|x\|^2$ for all $t$ and $x$.
A self-adjoint operator function $H\colon J\to \mathrm{L}(X)$ is called \emph{positive} (\emph{positive definite} or \emph{uniformly positive}, respectively) if the operator $H(t)$  is positive (positive definite, respectively).

Notice that a positive self-adjoint time-invariant operator $H\in \mathrm{L}(X)$, where $X$ is a finite-dimensional linear space, is positive definite.

 \smallskip
In what follows, we will use the notation
$$
(z,w)_{H}:=(H(t)z,w)
$$
for a scalar product with the weight $H(t)$.
 \begin{theorem}[on the uniform dissipativity of the DAE]\label{Th_SingDissipat}
Let $f\in C([t_+,\infty)\times \Rn,\Rm)$ and $\lambda A+B$ be a singular pencil of operators such that its regular block $\lambda A_r+B_r$ from \eqref{penc} has the index not higher than 1. Assume that condition \ref{SoglSing} of Theorem \ref{Th_SingGlobSol} holds and condition \ref{InvSing} of Theorem \ref{Th_SingGlobSol} or condition \ref{InvSing-Lipsch} of Theorem \ref{Th_SingGlobSol-Lipsch} holds. Let the following conditions be satisfied:
\begin{enumerate}%[3)]
\addtocounter{enumi}{2}
\item\label{Dissipat1Sing}  There exist a number $R>0$, a positive definite function $V\in C^1([t_+,\infty)\times D_{s_1}\times D_{p_1},\R)$, where a set $D_{s_1}\times D_{p_1}\subset X_{s_1}\times X_1$ is such that $D_{s_1}\times D_{p_1}\supset \{\|(x_{s_1},x_{p_1})\|\ge R\}$, and functions $U_j\in C([0,\infty))$, $j=0,1,2$, such that
    $U_0(r)$  is non-decreasing and $U_0(r)\to+\infty$ as $r\to+\infty$, $U_1(r)$ is increasing, $U_2(r)>0$ for $r>0$, and for all $(t,x_{s_1}+x_{s_2}+x_{p_1}+x_{p_2})\in L_{t_+}$, for which $x_{s_2}\in D_{s_2}$ and ${\|(x_{s_1},x_{p_1})\|\ge R}$, the inequality
    $$
    U_0(\|(x_{s_1},x_{p_1})\|)\le V(t,x_{s_1},x_{p_1})\le U_1(\|(x_{s_1},x_{p_1})\|)
    $$
    and one of the following inequalities (where $V'_{\eqref{DAEsysExtDE1},\eqref{DAEsysExtDE2}}(t,x_{s_1},x_{p_1})$ has the form \eqref{dVDAEsing}) hold:
  \begin{enumerate}[label={\upshape(\alph*)},ref={\upshape(\alph*)},topsep=1pt]
  \item\label{Dissipat1Sing:a}  $V'_{\eqref{DAEsysExtDE1},\eqref{DAEsysExtDE2}}(t,x_{s_1},x_{p_1})\le -U_2\big(\|(x_{s_1},x_{p_1})\|\big)$;

  \item\label{Dissipat1Sing:b}  $V'_{\eqref{DAEsysExtDE1},\eqref{DAEsysExtDE2}}(t,x_{s_1},x_{p_1})\le -U_2\big(\,((x_{s_1},x_{p_1}),(x_{s_1},x_{p_1}))_{H}\,\big)$, where  $H\in C([t_+,\infty),\mathrm{L}(X_{s_1}\times X_1))$  is some positive definite self-adjoint operator function such that $H(t)\big|_{X_{s_1}}\colon X_{s_1}\to X_{s_1}\times \{0\}$ and $H(t)\big|_{X_1}\colon X_1\to \{0\}\times X_1$ for any fixed $t$, and $\sup\limits_{t\in [t_+,\infty)}\|H(t)\|<\infty$;

  \item\label{Dissipat1Sing:c} $V'_{\eqref{DAEsysExtDE1},\eqref{DAEsysExtDE2}}(t,x_{s_1},x_{p_1})\le -\alpha\, V(t,x_{s_1},x_{p_1})$, where $\alpha>0$ is some constant.
  \end{enumerate}

\item\label{Dissipat2Sing}  There exist a constant  $\beta>0$ and a number $T> t_+$ such that $\|Q_2f(t,x_{s_1}+x_{s_2}+x_{p_1}+x_{p_2})\|\le \beta\, \|(x_{s_1},x_{p_1})\|$ for all $(t,x_{s_1}+x_{s_2}+x_{p_1}+x_{p_2})\in  L_T$, where $x_{s_2}\in D_{s_2}$.

\item\label{DissipatBs}  %=\eqref{LagrBs}
     $\|F_2f(t,x)\|<+\infty$ for all $(t,x)\in L_{t_+}$ such that $S_2x\in D_{s_2}$ and $\|x\|\le C<\infty$ \,($C$ is an arbitrary constant).
\end{enumerate}
Then, for the initial points  $(t_0,x_0)\in L_{t_+}$ where $S_2x_0\in D_{s_2}$ and any function $\phi_{s_2}\in C([t_0,\infty),D_{s_2})$  satisfying the relations  $\phi_{s_2}(t_0)=S_2 x_0$ and $\sup\limits_{t\in [t_0,\infty)}\|\phi_{s_2}(t)\|<+\infty$ the equation \eqref{DAE}, where $S_2x=\phi_{s_2}(t)$,  is uniformly dissipative (uniformly ultimately bounded); when $\rank(\lambda A+B)=n<m$, the component $S_2 x$ is absent.
 \end{theorem}

 \begin{remark}\label{RemDissipatSing}
If condition~\ref{Dissipat1Sing} of Theorem~\ref{Th_SingDissipat} holds, then condition~\ref{ExtensSing} of Theorem~\ref{Th_SingGlobSol} holds.  This can be readily verified if we take into account that by virtue of condition~\ref{Dissipat1Sing} of Theorem~\ref{Th_SingDissipat}, the inequality $V'_{\eqref{DAEsysExtDE1},\eqref{DAEsysExtDE2}}(t,x_{s_1},x_{p_1})\le 0$ holds (for all $(t,x_{s_1}+x_{s_2}+x_{p_1}+x_{p_2})\in L_{t_+}$ such that $x_{s_2}\in D_{s_2}$ and ${\|(x_{s_1},x_{p_1})\|\ge R}$) and the inequality $dv/dt\le 0$ has no positive solutions with finite escape time.
 \end{remark}

 \begin{proof}   
It follows from the conditions of the present theorem and Remark \ref{RemDissipatSing} that the conditions of Theorem \ref{Th_SingGlobSol} (or Theorem \ref{Th_SingGlobSol-Lipsch}) are satisfied.  Hence, there exists a unique global solution $x(t)$ of the IVP \eqref{DAE}, \eqref{ini} for each consistent initial point $(t_0,x_0)$ with $S_2x_0\in D_{s_2}$ and some chosen function $\phi_{s_2}\in C([t_0,\infty),D_{s_2})$ with the initial value $S_2x_0$ which defines the component $S_2x(t)=\phi_{s_2}(t)$. As in the proof of Theorem \ref{Th_SingGlobSol} (as well as of Theorem \ref{Th_SingGlobSol-Lipsch}), the solution has the form $x(t)=\omega_{s_1}(t)+\phi_{s_2}(t)+\omega_{p_1}(t)+ q(t,\omega_{s_1}(t),\omega_{p_1}(t))$, where $q(t,\omega_{s_1}(t),\omega_{p_1}(t))= \eta(t,\omega_{s_1}(t),\phi_{s_2}(t),\omega_{p_1}(t))$. By virtue of the conditions of the present theorem, it is assumed that  $\sup\limits_{t\in [t_0,\infty)}\|\phi_{s_2}(t)\|=\gamma<+\infty$.

Using condition~\ref{Dissipat1Sing}, where the inequality \ref{Dissipat1Sing:a}  or \ref{Dissipat1Sing:c} holds, and the proof of the theorem \cite[Theorem 10.4]{Yoshizawa} (on the dissipativity of explicit ODEs) and its corollary, we obtain that solutions of \eqref{DAEsysExtDEeta} are uniformly dissipative, that is, there exists a number $M>0$ and, for each solution $\omega(t)=(\omega_{s_1}(t)^\T,x_{p_1}(t)^\T)^\T$ satisfying the initial condition \eqref{DEeta_ini}, there exists a number $\tau_1=\tau_1(x_0)\ge t_0$ such that $\|\omega(t)\|<M$  for every $t\ge t_0+\tau_1$. Further, it follows from the properties of the operator $H(t)$ present in the inequality \ref{Dissipat1Sing:b} that $C_0\|\omega\|^2\le(H(t)\omega,\omega)\le C_1\|\omega\|^2$, where $C_0, C_1>0$ are some constants, for all $t$, $\omega$. Thus, it is easy to verify that from condition~\ref{Dissipat1Sing}, where the inequality \ref{Dissipat1Sing:b} holds, it follows that the solutions of the equation \eqref{DAEsysExtDEeta} are uniformly dissipative.

Recall that the function $u(t)= q(t,\omega_{s_1}(t),\omega_{p_1}(t))$ satisfies the equality \eqref{ExtAE1equiv2}. Therefore, according to condition \ref{Dissipat2Sing}, there exists a constant $\beta_0>0$ and a number $\tau_2=\tau_2(x_0)>t_0$ such that
$\|u(t)\|\le \|\Es B_2^{(-1)}\|\, \|Q_2 f(t,\omega_{s_1}(t)+\phi_{s_2}(t)+\omega_{p_1}(t)+u(t))\|\le \beta_0\,\|\omega(t)\|<\beta_0\,M$ for all $t\ge \tau_2$. Hence, for each solution with the initial values $t_0$, $x_0$ there exists a number $\tau=\tau(x_0)\ge t_0$ such that $\|x(t)\|\le \|\omega_{s_1}(t)\|+\|\phi_{s_2}(t)\|+\|\omega_{p_1}(t)\|+ \|u(t)\|< (2+\beta_0)M+\gamma=k$ for all $t\in [t_0+\tau,\infty)$, where the constant $k>0$ does not depend on $t_0$, $x_0$. Consequently, the DAE \eqref{DAE} is uniformly dissipative, and condition \ref{LagrBs} ensures the correctness of the equality \eqref{DAEsysExtAE2}.
\end{proof}

 \section{Replacement some conditions of the theorems by weaker ones}\label{SectWeakerCond}

This section shows how we can weaken some requirements of Theorems \ref{Th_SingGlobSol}, \ref{Th_SingGlobSol-Lipsch} and, as a consequence, some requirements of another theorems as well.

Below we give the definitions (cf. \cite{RF1,Fil.UMJ}) that will be used in the sequel.

Let $Z$ and $W$ be $n$-dimensional linear spaces.

A system of $n$ pairwise disjoint projectors $\{\Theta_i\}_{i=1}^n$ ($\Theta_l\, \Theta_j =\delta_{lj}\, \Theta_l$; the projectors are one-dimensional), where $\Theta_i\in \mathrm{L}(Z)$, such that their sum is the identity operator $I_Z=\sum\limits_{i=1}^n\Theta_i$ in $Z$ is called an \emph{additive resolution (or decomposition) of the identity} in $Z$.
Notice that an additive resolution of the identity in $Z$ generates the decomposition of $Z$ into the direct sum of the subspaces $Z_i=\Theta_i\, Z$, $i=1,...,n$ (i.e., $Z=Z_1\dot +...\dot +Z_n$).

An operator function $\Phi\colon D\to \mathrm{L}(W,Z)$, where $D\subset W$,  is called \emph{basis invertible} on an interval $J\subset D$ (or on a convex hull $J=\mathop{\mathrm{Conv}}\{w_1,w_2\}$ of $w_1,\, w_2\in D$) if for some additive resolution (decomposition) of the identity $\{\Theta_i\}_{i=1}^n$ in $Z$ and for each set of elements $\{w^k\}_{k=1}^n\subset J$ the operator $\Lambda=\sum\limits_{i=1}^n\Theta_i\Phi(w^i)\in \mathrm{L}(W,Z)$ has an inverse $\Lambda^{-1}\in \mathrm{L}(Z,W)$. This definition in terms of matrices is given in  \cite[p.~176]{Fil.MPhAG}. %(and also in \cite{Fil.UMJ}).

Note that the property of basis invertibility does not depend on the choice of an additive resolution of the identity in $Z$. Obviously, it follows from the basis invertibility of the mapping $\Phi$ on an interval $J$ that for each $w^*\in J$ the operator $\Phi (w^*)\in \mathrm{L}(W,Z)$ is invertible. The converse statement does not hold true, except for the case when $W$, $Z$ are one-dimensional spaces (see \cite[Example~1]{Fil.UMJ}).

 \begin{theorem}\label{Th_SingGlobSolBInv}
Theorem~\ref{Th_SingGlobSol} remains valid if conditions \ref{SoglSing} and \ref{InvSing} are replaced by the following:
\begin{enumerate}
\item\label{SoglSing2}  For any fixed  $t\in [t_+,\infty)$, $x_{s_1}\in X_{s_1}$, $x_{s_2}\in D_{s_2}$, where $D_{s_2}\subset X_{s_2}$ is a some set, and $x_{p_1}\in X_1$, there exists $x_{p_2}\in X_2$ such that $(t,x_{s_1}+x_{s_2}+x_{p_1}+x_{p_2})\in L_{t_+}$.

\item\label{BasInvSing} There exists the partial derivative ${\frac{\partial}{\partial x} f\in C([t_+,\infty)\times \Rn, \mathrm{L}(\Rn,\Rm))}$.\; For any fixed $t_*$, ${x_*^i=x_{s_1}^*+x_{s_2}^*+x_{p_1}^*+x_{p_2}^i}$  such that $(t_*,x_*^i)\in L_{t_+}$ and $x_{s_2}^*\in D_{s_2}$, the operator function  $\Phi_{t_*,x_{s_1}^*,x_{s_2}^*,x_{p_1}^*}(x_{p_2})$ defined by
    %\begin{multline}\label{funcPhiSingBInv}
   \begin{equation}\label{funcPhiSingBInv}
    \Phi_{t_*,x_{s_1}^*,x_{s_2}^*,x_{p_1}^*}\colon\! X_2\to \mathrm{L}(X_2,Y_2),\;\,
    \Phi_{t_*,x_{s_1}^*,x_{s_2}^*,x_{p_1}^*}(x_{p_2})=\!\bigg[\frac{\partial Q_2f}{\partial x}(t,x_{s_1}^*+x_{s_2}^*+x_{p_1}^*+x_{p_2})- B\bigg] P_2,
   \end{equation}
    is basis invertible on $[x_{p_2}^1,x_{p_2}^2]$.
\end{enumerate}
 \end{theorem}

 \begin{remark}\label{RemarkInv-BasInv}
Note that $\Phi_{t_*,x_{s_1}^*,x_{s_2}^*,x_{p_1}^*}(x_{p_2}^*)=\Phi_{t_*,x_*}$, where  $x_*=x_{s_1}^*+x_{s_2}^*+x_{p_1}^*+x_{p_2}^*$ and $\Phi_{t_*,x_*}$ is the operator defined by \eqref{funcPhiSing}, for any fixed $x_{p_2}^*\in X_2$. In addition, if the space $X_2$ is one-dimensional, then condition \ref{BasInvSing} of Theorem \ref{Th_SingGlobSolBInv} is equivalent
to condition \ref{InvSing} of Theorem \ref{Th_SingGlobSol}.
 \end{remark}

 \begin{proof}
As in the proof of Theorem \ref{Th_SingGlobSol}, consider the system
\eqref{DAEsysExtDE1}--\eqref{DAEsysExtAE2} equivalent to the DAE \eqref{DAE} and the equation \eqref{DAEsysAE1equiv}, where $\Psi$ is the mapping \eqref{funcPsiSing}, which is equivalent to the equation \eqref{DAEsysExtAE1} (or \eqref{DAEsysExtAE1equiv}, where $\widehat{\Psi}$ is the mapping \eqref{HatPsiSing}).

The partial derivative of $\Psi$ with respect to $x_{p_2}$ at the point $(t_*,x^*_{s_1},x^*_{s_2},x^*_{p_1},x^*_{p_2})$ has the form \eqref{W_tx} and can be written as
 \begin{equation*}
W_{t_*,x_*}=\dfrac{\partial \Psi}{\partial x_{p_2}}(t_*,x^*_{s_1},x^*_{s_2},x^*_{p_1},x^*_{p_2}) =B_2^{-1}\Phi_{t_*,x_{s_1}^*,x_{s_2}^*,x_{p_1}^*}(x_{p_2}^*)\in\mathrm{L}(X_2),
 \end{equation*}
where $x_*=x^*_{s_1}+x^*_{s_2}+x^*_{p_1}+x^*_{p_2}$ and $\Phi_{t_*,x_{s_1}^*,x_{s_2}^*,x_{p_1}^*}\in C(X_2,\mathrm{L}(X_2,Y_2))$ is the operator function defined by \eqref{funcPhiSingBInv}. Define the operator function
 \begin{equation}\label{Wu_tx}
W_{t_*,x_{s_1}^*,x_{s_2}^*,x_{p_1}^*}\colon\! X_2\to\mathrm{L}(X_2),\;\,  W_{t_*,x_{s_1}^*,x_{s_2}^*,x_{p_1}^*}(x_{p_2}):=\!\dfrac{\partial \Psi}{\partial x_{p_2}}(t_*,x^*_{s_1},x^*_{s_2},x^*_{p_1},x_{p_2}) \!=\! B_2^{-1}\Phi_{t_*,x_{s_1}^*,x_{s_2}^*,x_{p_1}^*}(x_{p_2}),\!
 \end{equation}
where $t_*$, $x_{s_1}^*$, $x_{s_2}^*$ and $x_{p_1}^*$ are arbitrary fixed elements of $[t_+,\infty)$, $X_{s_1}$, $X_{s_2}$ and $X_1$, respectively. Recall that the basis invertibility of the operator function $\Phi_{t_*,x_{s_1}^*,x_{s_2}^*,x_{p_1}^*}\colon X_2\to \mathrm{L}(X_2,Y_2)$ on some interval $J$ imply the invertibility of the operator $\Phi_{t_*,x_{s_1}^*,x_{s_2}^*,x_{p_1}^*}(x_{p_2}^*)$ for each (fixed) $x_{p_2}^*\in J$. Thus, it follows from condition \ref{BasInvSing} of the present theorem that for any fixed element $(t_*,x_{s_1}^*+x_{s_2}^*+x_{p_1}^*+x_{p_2}^*)\in L_{t_+}$  such that $x_{s_2}^*\in D_{s_2}$ the operator $W_{t_*,x_*}=W_{t_*,x_{s_1}^*,x_{s_2}^*,x_{p_1}^*}(x_{p_2}^*)$ (where $x_*=x_{s_1}^*+x_{s_2}^*+x_{p_1}^*+x_{p_2}^*$, and $W_{t_*,x_*}$ was defined in \eqref{W_tx}) has the inverse $W_{t_*,x_*}^{-1}= \big(\Phi_{t_*,x_{s_1}^*,x_{s_2}^*,x_{p_1}^*}(x_{p_2}^*)\big)^{-1}B_2\in \mathrm{L}(X_2)$.

 \smallskip
Let us prove that condition \ref{SoglSing} of Theorem \ref{Th_SingGlobSol} holds. Due to condition \ref{SoglSing2} of the present theorem, for each (fixed) $t\in [t_+,\infty)$, $x_{s_1}\in X_{s_1}$, $x_{s_2}\in D_{s_2}$, $x_{p_1}\in X_1$ there exists $x_{p_2}\in X_2$ such that $(t,x_{s_1}+x_{s_2}+x_{p_1}+x_{p_2})\in L_{t_+}$, and it is necessary to prove the uniqueness of such a $x_{p_2}$ in order to show that condition \ref{SoglSing} of Theorem \ref{Th_SingGlobSol} is satisfied.

Take arbitrary fixed $t_*\in [t_+,\infty)$, $x_{s_1}^*\in X_{s_1}$, $x_{s_2}^*\in D_{s_2}$, $x_{p_1}^*\in X_1$ and $x_{p_2}^i\in X_2$, ${i=1,2}$, such that ${(t_*,x_{s_1}^*+x_{s_2}^*+x_{p_1}^*+x_{p_2}^i)\in L_{t_+}}$,  then  $(t_*,x_{s_1}^*,x_{s_2}^*,x_{p_1}^*,x_{p_2}^i)$ must satisfy  \eqref{DAEsysAE1equiv},  i.e., ${\Psi(t_*,x_{s_1}^*,x_{s_2}^*,x_{p_1}^*,x_{p_2}^i)=0}$, ${i=1,2}$.
Note that the projector $P_2$ restricted to $X_2$ is the identity operator in $X_2$.  It follows from the basis invertibility of the operator function $\Phi_{t_*,x_{s_1}^*,x_{s_2}^*,x_{p_1}^*}$ on $[x_{p_2}^1,x_{p_2}^2]$ that for some additive resolution of the identity $\{\Theta_i\}_{i=1}^d$ in $X_2$ (where $d=\dim X_2$;\, $\sum\limits_{i=1}^d\Theta_i=I_{X_2}=P_2\big|_{X_2}$) the operator
\begin{equation}\label{LambdaSing}
\Lambda=\sum\limits_{i=1}^d \Theta_i W_{t_*,x_{s_1}^*,x_{s_2}^*,x_{p_1}^*}(x_{p_2,i})=B_2^{-1}\sum\limits_{i=1}^d \Theta_i \Phi_{t_*,x_{s_1}^*,x_{s_2}^*,x_{p_1}^*}(x_{p_2,i})
\end{equation}
is invertible for each set of the elements $\{x_{p_2,k}\}_{k=1}^d\subset [x_{p_2}^1,x_{p_2}^2]$. Hence, the operator function $W_{t_*,x_{s_1}^*,x_{s_2}^*,x_{p_1}^*}$ is basis invertible on $[x_{p_2}^1,x_{p_2}^2]$. 
Using the additive resolution of the identity $\{\Theta_i\}_{i=1}^d$, we define the functions
$$
\Psi_i:=\Theta_i \Psi\colon [t_+,\infty)\times X_{s_1}\times X_{s_2}\times X_1\times X_2\to X_{2,i}=\Theta_i X_2,\qquad {i=1,...,d}.
$$
Note that $X_{2,i}$, ${i=1,...,d}$,  are one-dimensional spaces isomorphic to $\R$, and $X_2=X_{2,1}\dot +...\dot +X_{2,d}$. By the finite increment formula, there exist points $x_{p_2,i}\in [x_{p_2}^1,x_{p_2}^2]$, ${i=1,...,d}$, such that
 \begin{multline*}
\Psi_i(t_*,x_{s_1}^*,x_{s_2}^*,x_{p_1}^*,x_{p_2}^2)- \Psi_i(t_*,x_{s_1}^*,x_{s_2}^*,x_{p_1}^*,x_{p_2}^1) = \dfrac{\partial \Psi_i}{\partial x_{p_2}}(t_*,x_{s_1}^*,x_{s_2}^*,x_{p_1}^*,x_{p_2,i}) \big(x_{p_2}^2-x_{p_2}^1\big)=  \\
=\Theta_i \dfrac{\partial \Psi}{\partial x_{p_2}}(t_*,x_{s_1}^*,x_{s_2}^*,x_{p_1}^*,x_{p_2,i}) \big(x_{p_2}^2-x_{p_2}^1\big)= \Theta_i W_{t_*,x_{s_1}^*,x_{s_2}^*,x_{p_1}^*}(x_{p_2,i}) \big(x_{p_2}^2-x_{p_2}^1\big), i=1,...,d.
\end{multline*}
Since $\Psi(t_*,x_{s_1}^*,x_{s_2}^*,x_{p_1}^*,x_{p_2}^i)=0$, $i=1,2$, then, summing the obtained expressions over $i$, we obtain
 $$
\sum\limits_{i=1}^d \Theta_i W_{t_*,x_{s_1}^*,x_{s_2}^*,x_{p_1}^*}(x_{p_2,i})\big(x_{p_2}^2-x_{p_2}^1\big)= \Lambda\big(x_{p_2}^2-x_{p_2}^1\big)=0,
 $$
where $\Lambda$ is defined in \eqref{LambdaSing}. Since the operator $\Lambda^{-1}$ exists, then $x_{p_2}^2=x_{p_2}^1$. This holds for each point $(t_*,x_{s_1}^*+x_{s_2}^*+x_{p_1}^*+x_{p_2}^i)\in L_{t_+}$, $i=1,2$, where $x_{s_2}^*\in D_{s_2}$,  since these points were chosen arbitrarily.
Thus, the proof of condition \ref{SoglSing} of Theorem \ref{Th_SingGlobSol}  is complete.

As in the proof of Theorem \ref{Th_SingGlobSol}, take arbitrary fixed $t_*\in [t_+,\infty)$, $x_{s_1}^*\in X_{s_1}$, $x_{s_2}^*\in D_{s_2}$, $x_{p_1}^*\in X_1$. As proved above, there exists a unique $x_{p_2}^*\in X_2$ such that $(t_*,x_*)\in L_{t_+}$, where $x_*=x_{s_1}^*+x_{s_2}^*+x_{p_1}^*+x_{p_2}^*$,
and for this $(t_*,x_*)$ the operator $W_{t_*,x_*}=W_{t_*,x_{s_1}^*,x_{s_2}^*,x_{p_1}^*}(x_{p_2}^*)$ has the inverse $W_{t_*,x_*}^{-1}\in \mathrm{L}(X_2)$.

The further proof coincides with the proof of Theorem \ref{Th_SingGlobSol}.

Generally, we proved above that conditions \ref{SoglSing}, \ref{InvSing} of Theorem \ref{Th_SingGlobSol} are satisfied, and the rest of the conditions of Theorem \ref{Th_SingGlobSol} are the same as in the present theorem. Hence, the theorem is proved.
 \end{proof}

 \begin{corollary}
Theorems~\ref{Th_SingUstLagr}, \ref{Th_singNeLagr} and \ref{Th_SingDissipat} (which contain conditions \ref{SoglSing}, \ref{InvSing} of Theorem \ref{Th_SingGlobSol}) remain valid if instead of condition \ref{SoglSing} of Theorem \ref{Th_SingGlobSol} and instead of condition \ref{InvSing} of Theorem \ref{Th_SingGlobSol} or condition \ref{InvSing-Lipsch} of Theorem \ref{Th_SingGlobSol-Lipsch} one requires that conditions \ref{SoglSing2}, \ref{BasInvSing} of Theorem \ref{Th_SingGlobSolBInv} hold.
 \end{corollary}

Below we show how condition \ref{ExtensSing} of Theorem \ref{Th_SingGlobSol} can be weakened.

 First, we consider an ODE
\begin{equation}\label{ODE}
\frac{dx}{dt}=F(t,x),
\end{equation}
where $t\in [t_+,\infty)$, $t_+\ge 0$, $x\in W$ and $W$ is an $n$-dimensional  Euclidean space, and the function ${F\in C([t_+,\infty)\times W,W)}$ satisfies locally a Lipschitz condition ($F(t,x)$ is locally Lipschitz continuous) with respect to $x$ on $[t_+,\infty)\times W$, i.e., for each point $(t_*,x_*)\in [t_+,\infty)\times W$ there exists a neighborhood $U(t_*,x_*)$ and a number $L\ge 0$ such that $\|F(t,x_1)-F(t,x_2)\|\le L\|x_1-x_2\|$ for any $(t,x_1),(t,x_2)\in U(t_*,x_*)$. 
According to \cite{LaSal-Lef}, a solution $x(t)$ of the ODE \eqref{ODE}, which satisfies some initial condition $x(t_0)=x_0$, is called \emph{defined in the future}, if it can be extended for all $t\ge t_0$, i.e., to the whole interval $[t_0,\infty)$, and hence this solution is global by the definition given in this paper.  Thus, these definitions are equivalent.
Consider the ODE (the ODE \eqref{ODE} with a truncation)
\begin{equation}\label{ODE-cut}
\frac{dx}{dt}=F_{\scrscr T}(t,x),\quad \text{where}\quad
F_{\scrscr T}(t,x):= \begin{cases}
 F(t,x), & t_+ \le t \le T, \\
 F(T,x), & t > T,
\end{cases}\quad  \text{$T\ge t_+$ is a parameter}.
\end{equation}
The function $F_{\scrscr T}(t,x)$  is called the \emph{truncation} of the function $F(t,x)$ over $t$, and it has the same properties as $F(t,x)$, i.e.,  $F_{\scrscr T}(t,x)$ is continuous and locally satisfies a Lipschitz condition with respect to $x$ on $[t_+,\infty)\times W$.

Below is the lemma proved in \cite{Fil}, which generalizes Theorem~\cite[Chapter~IV, Theorem~XIII]{LaSal-Lef} and will be used in the sequel.

Recall that if $D$ is some set in a space $X$, then $D^c$ denotes the complement of the set $D$ (relative to $X$), i.e., $D^c=X\setminus D$.
 \begin{lemma}[cf. {\cite[Lemma 3.1]{Fil}}]\label{Fil-Lem1}
Let there exist a positive definite function $V\in C^1([t_+,\infty)\times D^c,\R)$, where $D^c$ is the complement of some bounded set $D\subset W$ containing $0$ (the origin). Let for each number $T>0$ there exist a set $D_{\scrscr T}\supset D$ and a function $\chi_{\scrscr T}\in C([t_+,\infty)\times (0,\infty),\R)$ such that the following holds:
\begin{enumerate}
 \item $V(t,x)\to\infty$ uniformly in $t$ on each finite interval $[a,b)\subset [t_+,\infty)$ as ${\|x\|\to\infty}$;

 \item for all $t\in[t_+,\infty)$, $x\in D_{\scrscr T}^c$ the inequality
    $$
    V'_{\eqref{ODE-cut}}(t,x)\le \chi_{\scrscr T}(t,V(t,x)),
    $$
    where $V'_{\eqref{ODE-cut}}(t,x)= \dfrac{\partial V}{\partial t}(t,x) +\dfrac{\partial V}{\partial x}(t,x)\cdot F_{\scrscr T}(t,x)$ is the derivative of $V$ along the trajectories of the equation~\eqref{ODE-cut}, is satisfied;

 \item the differential inequality $dv/dt\le \chi_{\scrscr T}(t,v)$ \,($t\in [t_+,\infty)$) does not have positive solutions with finite escape time.
\end{enumerate}
Then every solution of the ODE \eqref{ODE} is global (defined in the future).
 \end{lemma}

 \begin{proof}
The proof is carried out in the same way as the proof of the lemma~\cite[Lemma~3.1]{Fil}.
 \end{proof}

Let us return to the consideration of the DAE \eqref{DAE}. Recall that it is equivalent to the system \eqref{DAEsysExtDE1}--\eqref{DAEsysExtAE2}.
Introduce the truncation of the function $f(t,x)$ over $t$:
\begin{equation*}%\label{ODE-cut}
f_{\scrscr T}(t,x):= \begin{cases}
 f(t,x), & t_+ \le t \le T, \\   f(T,x), & t > T,
\end{cases}\quad  \text{$T\ge t_+$ is a parameter}.
\end{equation*}
Then
\begin{equation}\label{UpsilonTrunc}
\Upsilon_{\scrscr T}(t,x):=\begin{pmatrix}
\Es A_{gen}^{(-1)}\big(F_1 f_{\scrscr T}(t,x)-\Es B_{gen} x_{s_1}-\Es B_{und} x_{s_2}\big) \\
\Es A_1^{(-1)} \big(Q_1 f_{\scrscr T}(t,x)-\Es B_1 x_{p_1}\big)  \end{pmatrix}= \begin{cases}
 \Upsilon(t,x), & t_+ \le t \le T, \\  \Upsilon(T,x), & t > T
\end{cases}\;  \text{($T$ is a parameter)}
\end{equation}
is the truncation of the function $\Upsilon(t,x)$ \eqref{Upsilon} over $t$ and a vector consisting of the right-hand sides of the equations  \eqref{DAEsysExtDE1}, \eqref{DAEsysExtDE2} with a truncation, i.e., the equations
\begin{align}
\frac{d}{dt} x_{s_1} &=\Es A_{gen}^{(-1)}[F_1 f_{\scrscr T}(t,x)-\Es B_{gen} x_{s_1}-\Es B_{und} x_{s_2}],  \label{DAEsysDE1Trunc}  \\
\frac{d}{dt}x_{p_1} &=\Es A_1^{(-1)}[Q_1 f_{\scrscr T}(t,x)-\Es B_1 x_{p_1}]. \label{DAEsysDE2Trunc}
\end{align}
The derivative of the function $V\in C^1([t_+,\infty)\times D_{s_1}\times D_{p_1},\R)$  \,(where $D_{s_1}\times D_{p_1}\subset X_{s_1}\times X_1$) along the trajectories of the system \eqref{DAEsysDE1Trunc}, \eqref{DAEsysDE2Trunc} has the form
 \begin{multline}\label{dVDAEsingTrunc}
V'_{\eqref{DAEsysDE1Trunc},\eqref{DAEsysDE2Trunc}}(t,x_{s_1},x_{p_1})= \frac{\partial V}{\partial t}(t,x_{s_1},x_{p_1})+\frac{\partial V}{\partial (x_{s_1},x_{p_1})}(t,x_{s_1},x_{p_1})\cdot \Upsilon_{\scrscr T}(t,x) = \\
= \frac{\partial V}{\partial t}(t,x_{s_1},x_{p_1})+ \frac{\partial V}{\partial x_{s_1}}(t,x_{s_1},x_{p_1})\cdot  \left[\Es A_{gen}^{(-1)}\big(F_1 f_{\scrscr T}(t,x)-\Es B_{gen} x_{s_1}-\Es B_{und} x_{s_2}\big)\right] + \\
+\frac{\partial V}{\partial x_{p_1}}(t,x_{s_1},x_{p_1})\cdot   \left[\Es A_1^{(-1)} \big(Q_1 f_{\scrscr T}(t,x)-\Es B_1 x_{p_1}\big)\right].
 \end{multline}

 \begin{theorem}\label{Th_SingGlobSolTrunc}
Theorems~\ref{Th_SingGlobSol} and \ref{Th_SingGlobSol-Lipsch}  remain valid if condition \ref{ExtensSing} of Theorem~\ref{Th_SingGlobSol} (which must also hold for Theorem \ref{Th_SingGlobSol-Lipsch}) is replaced by the following:
\begin{enumerate} 
\addtocounter{enumi}{2}
\item\label{ExtensSingTrunc} There exists a positive definite function  $V\in C^1([t_+,\infty)\times D_{s_1}\times D_{p_1},\R)$, where a set $D_{s_1}\times D_{p_1}\subset X_{s_1}\times X_1$ is such that  $D_{s_1}\times D_{p_1}\supset \{\|(x_{s_1},x_{p_1})\|\ge R\}$, and  $R>0$ is some number, and for each number $T>0$ there exists a number $R_{\scrscr T}\ge R$ and a function  $\chi_{\scrscr T}\in C([t_+,\infty)\times (0,\infty),\R)$ such that the following holds:
  \begin{enumerate}[label={\upshape(\alph*)},ref={\upshape(\alph*)},topsep=1pt]
  \item\label{ExtensSing1Trunc}   ${V(t,x_{s_1},x_{p_1})\to\infty}$ uniformly in $t$ on each finite interval $[a,b)\subset[t_+,\infty)$ as  ${\|(x_{s_1},x_{p_1})\|\to\infty}$;

  \item\label{ExtensSing2Trunc} for all $(t,x_{s_1}+x_{s_2}+x_{p_1}+x_{p_2})\in L_{t_+}$,  for which $x_{s_2}\in D_{s_2}$ and ${\|(x_{s_1},x_{p_1})\|\ge R_{\scrscr T}}$, the inequality
    \begin{equation}\label{LagrDAEsingTrunc}
  V'_{\eqref{DAEsysDE1Trunc},\eqref{DAEsysDE2Trunc}}(t,x_{s_1},x_{p_1})\le \chi_{\scrscr T}\big(t,V(t,x_{s_1},x_{p_1})\big),
    \end{equation}
    where $V'_{\eqref{DAEsysDE1Trunc},\eqref{DAEsysDE2Trunc}}(t,x_{s_1},x_{p_1})$ has the form \eqref{dVDAEsingTrunc}, holds;

  \item\label{GlobSolvTrunc}  the differential inequality $dv/dt\le \chi_{\scrscr T}(t,v)$   \,($t\in [t_+,\infty)$) does not have positive solutions with finite escape time.
  \end{enumerate}
 \end{enumerate}
 \end{theorem}

 \begin{proof}
The proof coincides with the proof of Theorem \ref{Th_SingGlobSol} (or Theorem \ref{Th_SingGlobSol-Lipsch}), except for the part where the existence of a global solution of the ODE \eqref{DAEsysExtDEeta}, i.e.,
 $$
d\omega/dt=\widetilde{\Upsilon}(t,\omega),
$$
is proved. Let us prove this part using the conditions of the present theorem.

As shown in the proof of Theorem \ref{Th_SingGlobSol}  (as well as in the proof of Theorem \ref{Th_SingGlobSol-Lipsch}), there exists a unique solution $\omega=\omega(t)$ of the equation \eqref{DAEsysExtDEeta} on some interval  $[t_0,\beta)$, which satisfies the initial condition \eqref{DEeta_ini}. Let us prove that this solution is global, i.e., the maximal interval of the existence of the solution is  $[t_0,\infty)$.  Recall that $(t_0,x_0)\in L_{t_+}$, where $S_2x_0\in D_{s_2}$, is an arbitrarily chosen initial point and that $\phi_{s_2}\in C([t_0,\infty),D_{s_2})$ is an arbitrarily chosen function with the initial value $\phi_{s_2}(t_0)=S_2 x_0$.

Consider the ODE  \eqref{DAEsysExtDEeta} with a truncation, that is,
\begin{equation}\label{DAEsysDEetaTrunc}
\frac{d}{dt}\omega=\widetilde{\Upsilon}_{\scrscr T}(t,\omega),\quad \text{where}\quad \omega=\begin{pmatrix} x_{s_1} \\ x_{p_1} \end{pmatrix},\;\; \widetilde{\Upsilon}_{\scrscr T}(t,\omega):=\begin{cases}
\widetilde{\Upsilon}(t,\omega), & t_0 \le t \le T, \\  \widetilde{\Upsilon}(T,\omega), & t > T,
\end{cases}\;\;  \text{$T$ is a parameter}.
\end{equation}
Note that  $\widetilde{\Upsilon}_{\scrscr T}(t,\omega):=\begin{cases}
\widetilde{\Upsilon}(t,\omega)= \Upsilon(t,x_{s_1}+\phi_{s_2}(t)+x_{p_1}+q(t,x_{s_1},x_{p_1})), & t_0\le t\le T, \\  \widetilde{\Upsilon}(T,\omega)= \Upsilon(T,x_{s_1}+\phi_{s_2}(T)+x_{p_1}+q(T,x_{s_1},x_{p_1})), & t>T,
\end{cases}$.

We choose a number $R>0$ ($R<\infty$) and a function $V(t,x_{s_1},x_{p_1})$ such that condition \ref{ExtensSingTrunc} of the theorem holds, and introduce the function $V(t,\omega):=V(t,x_{s_1},x_{p_1})$. Due to condition \ref{ExtensSingTrunc}, for each $T>0$ there exists a number $R_{\scrscr T}\ge R$ and a function $\chi_{\scrscr T}\in C([t_+,\infty)\times (0,\infty),\R)$ such that the derivative of $V$ along the trajectories of the equation \eqref{DAEsysDEetaTrunc} satisfies the inequality
 \begin{equation}\label{L1dVdtSingTrunc}
V'_{\eqref{DAEsysDEetaTrunc}}(t,\omega)=\dfrac{\partial V}{\partial t}(t,\omega)+\dfrac{\partial V}{\partial \omega}(t,\omega)\cdot \widetilde{\Upsilon}_{\scrscr T}(t,\omega)\le \chi_{\scrscr T}\big(t,V(t,\omega)\big)
 \end{equation}
for all $t\ge t_0$ and $\|\omega\|\ge R_{\scrscr T}$. Since, by virtue of condition \ref{GlobSolvTrunc}, the differential inequality $dv/dt\le \chi_{\scrscr T}(t,v)$ ($t\in [t_0,\infty)$) does not have positive solutions with finite escape time, then by Lemma~\ref{Fil-Lem1} every solution of \eqref{DAEsysExtDEeta}, including the solution $\omega(t)=(\omega_{s_1}(t)^\T,\omega_{p_1}(t)^\T)^\T$, is global, i.e.,  exists on $[t_0,\infty)$. Thus, what was needed has been proved.
 \end{proof}

 \section{On the choice of the functions $\chi$ and $V$ when checking the conditions of proved theorems}\label{Sect_chi_V}

The proved theorems contain conditions in a general form, and the main difficulty in applying the theorems lies in choosing suitable functions $\chi$ and $V$.

First, consider the function $\chi\in C([t_+,\infty)\times (0,\infty),\R)$ which is present in Theorems \ref{Th_SingGlobSol}, \ref{Th_SingGlobSol-Lipsch}, \ref{Th_SingGlobSolBInv}, \ref{Th_SingUstLagr} and \ref{Th_singNeLagr} on the global solvability, Lagrange stability and Lagrange instability of the DAE, respectively. Let us choose it in the form \eqref{La-kU}, that is,
 \begin{equation*}
\chi(t,v)=k(t)\,U(v),
 \end{equation*}
where $k\in C([t_+,\infty),\R)$ and $U\in C(0,\infty)$, then the conditions of the theorems take the following form:
 \begin{enumerate}
\item 
 In Theorems \ref{Th_SingGlobSol}, \ref{Th_SingGlobSol-Lipsch} and \ref{Th_SingGlobSolBInv} on the global solvability all conditions remain unchanged, except for condition~\ref{ExtensSing} (this condition is contained in all of these  theorems) which takes the form:
  \begin{itemize}
 \item[\ref{ExtensSing}.]  There exists a number $R>0$, a positive definite function  $V\in C^1([t_+,\infty)\times D_{s_1}\times D_{p_1},\R)$, where a set $D_{s_1}\times D_{p_1}\subset X_{s_1}\times X_1$ is such that $D_{s_1}\times D_{p_1}\supset \{\|(x_{s_1},x_{p_1})\|\ge R\}$, and functions ${k\in C([t_+,\infty),\R)}$, $U\in C(0,\infty)$ such that:
    \begin{enumerate}[label={\upshape(\alph*)},ref={\upshape(\alph*)},topsep=1pt]
   \item condition \ref{ExtensSing1} of Theorem \ref{Th_SingGlobSol} holds, i.e.,  ${V(t,x_{s_1},x_{p_1})\to\infty}$ uniformly in $t$ on each finite interval $[a,b)\subset[t_+,\infty)$ as ${\|(x_{s_1},x_{p_1})\|\to\infty}$;
   \item  for all $(t,x_{s_1}+x_{s_2}+x_{p_1}+x_{p_2})\in L_{t_+}$, for which $x_{s_2}\in D_{s_2}$ and ${\|(x_{s_1},x_{p_1})\|\ge R}$, the following inequality holds:
        \begin{equation}\label{LagrDAEsing-kU}
       V'_{\eqref{DAEsysExtDE1},\eqref{DAEsysExtDE2}}(t,x_{s_1},x_{p_1})\le k(t)\, U\big(V(t,x_{s_1},x_{p_1})\big);
        \end{equation}
   \item  $\int\limits_{{\textstyle v}_0}^{\infty}\dfrac{dv}{U(v)} =\infty$\;  ($v_0>0$ is some constant). %% условия~\ref{GlobSolv}
    \end{enumerate}
  \end{itemize}

\item 
  In Theorem \ref{Th_SingUstLagr} on the Lagrange stability  all conditions remain unchanged, except for condition~\ref{LagrSing} which takes the form:
   \begin{itemize}
  \item[\ref{LagrSing}.]  There exists a number $R>0$, a positive definite function  $V\in C^1([t_+,\infty)\times D_{s_1}\times D_{p_1},\R)$, where a set $D_{s_1}\times D_{p_1}\subset X_{s_1}\times X_1$ is such that $D_{s_1}\times D_{p_1}\supset \{\|(x_{s_1},x_{p_1})\|\ge R\}$, and functions ${k\in C([t_+,\infty),\R)}$, $U\in C(0,\infty)$ such that:
    \begin{enumerate}[label={\upshape(\alph*)},ref={\upshape(\alph*)},topsep=1pt]
   \item condition \ref{LagrSing1} of Theorem \ref{Th_SingUstLagr} holds, i.e., ${V(t,x_{s_1},x_{p_1})\to\infty}$  uniformly in $t$ on $[t_+,\infty)$ as ${\|(x_{s_1},x_{p_1})\|\to\infty}$;

   \item for all $(t,x_{s_1}+x_{s_2}+x_{p_1}+x_{p_2})\in L_{t_+}$, for which $x_{s_2}\in D_{s_2}$ and ${\|(x_{s_1},x_{p_1})\|\ge R}$, the inequality \eqref{LagrDAEsing-kU} is satisfied (i.e., condition \ref{ExtensSing2} of Theorem~\ref{Th_SingGlobSol} holds).

   \item $\int\limits_{{\textstyle v}_0}^{\infty}\dfrac{dv}{U(v)} =\infty$,\; $\int\limits_{t_0}^{\infty}k(t)dt<\infty$\; ($t_0\ge t_+$, $v_0>0$ are some numbers).  
   \end{enumerate}
  \end{itemize}

\item 
   In Theorem \ref{Th_singNeLagr} on the Lagrange instability  all conditions remain unchanged, except for condition~\ref{NeUstLagrSing} which takes the form:
 \begin{itemize}
  \item[\ref{NeUstLagrSing}.]   There exists a positive definite function $V\in C^1([t_+,\infty)\times \Omega,\R)$ and functions ${k\in C([t_+,\infty),\R)}$, $U\in C(0,\infty)$ such that:
   \begin{enumerate}[label={\upshape(\alph*)},ref={\upshape(\alph*)},topsep=1pt]
  \item  for all $(t,x_{s_1}+x_{s_2}+x_{p_1}+x_{p_2})\in L_{t_+}$, for which  $x_{s_2}\in D_{s_2}$ and $(x_{s_1},x_{p_1})\in\Omega$, the following inequality holds:
  \begin{equation}\label{NeLagrDAESing-kU}
    V'_{\eqref{DAEsysExtDE1},\eqref{DAEsysExtDE2}}(t,x_{s_1},x_{p_1})\ge k(t)\, U\big(V(t,x_{s_1},x_{p_1})\big);
  \end{equation}

  \item $\int\limits_{{\textstyle v}_0}^{\infty}\dfrac{dv}{U(v)}<\infty$,\; $\int\limits_{t_0}^{\infty}k(t)dt=\infty$\;  ($t_0\ge t_+$, $v_0>0$ are some numbers).   
   \end{enumerate}
  \end{itemize}
 \end{enumerate}

The validity of the theorems with the above changes in the conditions follows directly from the remarks on differential inequalities given in Section~\ref{LaSDiffIneq}. Their proofs can be readily obtain from the proofs of the corresponding original theorems, and therefore they are not given here.  It is clear that the inequalities \eqref{LagrDAEsing} and \eqref{NeLagrDAESing} with the function $\chi$ in the form \eqref{La-kU} take the form \eqref{LagrDAEsing-kU} and \eqref{NeLagrDAESing-kU}.

Recall that $V'_{\eqref{DAEsysExtDE1},\eqref{DAEsysExtDE2}}(t,x_{s_1},x_{p_1})$  has the form \eqref{dVDAEsing}.

 \smallskip
Now, consider the positive definite scalar function $V(t,x_{s_1},x_{p_1})$ which is present in all theorems proved above. This function will be called a \emph{Lyapunov type function}. Let us choose it in the form
\begin{equation}\label{funcV}
V(t,x_{s_1},x_{p_1})=\big((x_{s_1},x_{p_1}),(x_{s_1},x_{p_1})\big)_{H}= \big(H(t)(x_{s_1},x_{p_1}),(x_{s_1},x_{p_1})\big),
\end{equation}
where $H\in C([t_+,\infty),\mathrm{L}(X_{s_1}\times X_1))$ is a positive definite self-adjoint operator function such that $H(t)\big|_{X_{s_1}}\colon X_{s_1}\to X_{s_1}\times \{0\}$ and $H(t)\big|_{X_1}\colon X_1\to \{0\}\times X_1$ for any fixed $t$.  Due to the properties of the operator function $H$,   the function \eqref{funcV} satisfies the conditions of Theorems \ref{Th_SingGlobSol}, \ref{Th_SingGlobSol-Lipsch}, \ref{Th_SingGlobSolBInv}, \ref{Th_SingUstLagr}, \ref{Th_singNeLagr} on the global solvability, Lagrange stability and instability,
and if in addition ${\sup\limits_{t\in [t_+,\infty)}\|H(t)\|<\infty}$, then the function \eqref{funcV} also satisfies the conditions of Theorem \ref{Th_SingDissipat} on the uniform dissipativity, however, of course, the conditions on the derivative $V'_{\eqref{DAEsysExtDE1},\eqref{DAEsysExtDE2}}(t,x_{s_1},x_{p_1})$ in these theorems need to be checked.

The conditions $H(t)\big|_{X_{s_1}}\colon X_{s_1}\to X_{s_1}\times \{0\}$ and $H(t)\big|_{X_1}\colon X_1\to \{0\}\times X_1$ ($t\ge t_+$ is fixed) mean that the pair of subspaces $\{X_{s_1},X_{s_1}\times \{0\}\}$ and the pair of subspaces $\{X_1,\{0\}\times X_1\}$ are invariant under the operator $H(t)\in \mathrm{L}(X_{s_1}\times X_1)$ (for each $t$) and it has the block structure
 \begin{equation}\label{BlockStrH}
H(t)=\begin{pmatrix}
   H_{s_1}(t) & 0   \\
   0   & H_1(t) \end{pmatrix}\colon X_{s_1}\times X_1\to X_{s_1}\times X_1,
 \end{equation}
where $H_{s_1}\in C([t_+,\infty),\mathrm{L}(X_{s_1}))$ and $H_1\in C([t_+,\infty),\mathrm{L}(X_1))$ are positive definite self-adjoint operator functions.  Note that if we identify  $X_{s_1}\times \{0\}$ with  $X_{s_1}$ and $\{0\}\times X_1$ with  $X_1$, i.e., identify $X_{s_1}\times X_1=X_{s_1}\times \{0\}\dot+\{0\}\times X_1$ with  $X_{s_1}\dot+ X_1$ as in Section \ref{BlockStruct}, then  $H(t)$ ($t$ fixed) can be considered as the operator $H(t)\colon X_{s_1}\dot+ X_1\to X_{s_1}\dot+ X_1$.

   \smallskip
If ${H(t)\equiv H\in \mathrm{L}(X_{s_1}\times X_1)}$ is a time-invariant operator, then for all theorems it suffices to require that it be self-adjoint and positive and that the pairs of subspaces $\{X_{s_1},X_{s_1}\times \{0\}\}$ and $\{X_1,\{0\}\times X_1\}$ be invariant under $H$. Then the function \eqref{funcV} takes the form $V(t,x_{s_1},x_{p_1})\equiv V(x_{s_1},x_{p_1})=\big(H(x_{s_1},x_{p_1}),(x_{s_1},x_{p_1})\big)$ and satisfies the conditions of all theorems, except for the conditions on $V'_{\eqref{DAEsysExtDE1},\eqref{DAEsysExtDE2}}(t,x_{s_1},x_{p_1})$ which need to be checked.

For the function $V$ of the form \eqref{funcV} the derivative $V'_{\eqref{DAEsysExtDE1},\eqref{DAEsysExtDE2}}(t,x_{s_1},x_{p_1})$ \eqref{dVDAEsing} takes the form
 \begin{multline}\label{VderivDAE}
V'_{\eqref{DAEsysExtDE1},\eqref{DAEsysExtDE2}}(t,x_{s_1},x_{p_1})=
 \Big(\frac{d}{dt}H(t)(x_{s_1},x_{p_1}),(x_{s_1},x_{p_1})\Big)+
   2\Big( H(t)(x_{s_1},x_{p_1}),\Upsilon(t,x)\Big) =  \\
 =\Big(\frac{d}{dt}H(t)(x_{s_1},x_{p_1}),(x_{s_1},x_{p_1})\Big)+
   2\Big(H_{s_1}(t)x_{s_1},\left[\Es A_{gen}^{(-1)}\big(F_1 f(t,x)-\Es B_{gen} x_{s_1}-\Es B_{und} x_{s_2}\big)\right] \Big) +  \\
  +2\Big(H_1(t)x_{p_1}, \left[\Es A_1^{(-1)} \big(Q_1 f(t,x)-\Es B_1 x_{p_1}\big)\right]\Big),
 \end{multline}
where $H_{s_1}(t)$, $H_1(t)$ are operators defined in \eqref{BlockStrH}, and $\Upsilon(t,x)$ has the form \eqref{Upsilon}.

   \section{Isothermal models of gas networks in the form of DAEs}\label{Sect_GasNetDAE}

 \subsection{A model of a gas flow for a single pipe (in the isothermal case)}\label{Sect_GasDAEsingle}

Consider a mathematical model of a gas flow for a single pipe, assuming that the temperature is constant. The model consists of the \emph{isothermal Euler equations} which we write in the form (see, e.g., \cite[(ISO1),~p.12]{GNM}, \cite[p.~27]{GNM})
 \begin{align}%\label{ISO1}\tag{ISO1}
& \partial_t\rho +\partial_x(\rho v) = 0, \label{ISO1.1} \\
& \partial_t(\rho v) +\partial_x(p+\rho v^2) = -\dfrac{\lambda_{fr}}{2D}\rho v|v|-g\rho\, \dfrac{dh}{dx}, \label{ISO1.2}
 \end{align}
where $x\in [0,L]$, $L<\infty$ is the length of the pipe, and $t\in \mathcal{I}\subset[0,\infty)$, $\mathcal{I}$ is the time interval, and \emph{the equation of state} for a gas for the constant temperature (the isothermal case):
 \begin{equation}\label{ISGE}%\tag{ISGE}
p=R_s T_0 \rho z(p).
 \end{equation}
Here $\rho=\rho(t,x)$ denotes the density, $v=v(t,x)$ is the velocity of the gas, $p=p(t,x)$ is the pressure, $g$ is the gravitational acceleration, $D$ is the pipe diameter, $\lambda_{fr}$ is the pipe friction coefficient, $T=T_0$ is the temperature (the surface temperature of the pipe wall),  $R_s$ is the specific gas constant, $z=z(p)$ is the compressibility factor, and $h=h(x)$ is the height profile of the pipe over ground  (see, e.g.,\cite[p.4]{GNM} and references therein).

A special model to describe the compressibility factor, which is used by the American Gas Association (AGA) and is a good approximation for pressures up to 70 bar, has the form $z(p)=1+\alpha p$,  where $\alpha$ is a certain constant \cite[p.5]{GNM}. Then $p=R_s T_0 \rho/(1-\alpha R_s T_0 \rho)$ that is one of the commonly used equation of state in the isothermal case.

In what follows, we denote the slope $\dfrac{dh}{dx}(x)$ of the pipe by $\dfrac{dh}{dx}(x)=s_{lope}(x)$ (cf. \cite{GU2018, MartinMM05}).

When modelling the dynamics of a gas flow, the assumption $(\rho v^2)_x=0$ (i.e., we assume that this term is negligibly small) can be used (see, e.g., \cite{KSSTW22, BandaHerty, MartinMM05}) in order to simplify the model, then we obtain the gas dynamics equations in the form \cite[(ISO2),~p.12]{GNM} (the similar system is used in \cite{KSSTW22}):
\begin{equation}\label{ISO2}%\tag{ISO2}
 \begin{split}
 &\partial_t\rho+ \partial_x(\rho v) = 0, \\
 &\partial_t(\rho v)+ \partial_x p =-\dfrac{\lambda_{fr}}{2D}\rho v|v|- g\rho\,s_{lope}
 \end{split}
\end{equation}
with the same equation of state \eqref{ISGE}. The equations  \eqref{ISO2} are often referred to as a semilinear model of the gas flow dynamics \cite{KSSTW22,GNM}.

The similar model for a gas flow in a pipeline was presented in  \cite{BandaHerty}, where it was used the isothermal Euler equations with the equation of state in the form $p=c^2 \rho$, where $c=Z R T_0/M_g$ was assumed to be a certain constant, $Z$ is the natural gas compressibility factor, $R$ is the universal gas constant, $T$ is the absolute gas temperature and $M_g$ is the gas molecular weight.

In \cite[p.26]{GNM} and \cite[p.2,3]{KSSTW22},  $q$ denotes a mass flow and it is defined as $q=S \rho v$, where $S$ is the cross-sectional area of a pipe.
We will denote by $q:=\rho v$ a mass flow by the cross-sectional area equal to 1, in order not to introduce additional notation, and assume that the total mass flow is $\tilde{q}=q S$. Also, we will assume that the directions of gas flows in pipes are known.
Then the system of the isothermal Euler equations \eqref{ISO2} and the gas state equation \eqref{ISGE} takes the form
 \begin{align}
\partial_t\rho+\partial_x q &= 0, \label{ISO2.1} \\
\partial_t q +\partial_x p +\rho\,g s_{lope} &=-\dfrac{\lambda_{fr}}{2 D} q^2\rho^{-1}, \label{ISO2.2}  \\
p &=R_s T_0\rho z(p).  \label{SGE}
 \end{align}
Further, we assume that $s_{lope}(x)\equiv\sin \theta$, where the parameter $\theta$ denotes the angle of the pipe slope (cf. \cite{GNM, KSSTW22}), and discretize the equations \eqref{ISO2.1}, \eqref{ISO2.2} in the phase variable (in space), using a scheme similar to the topology-adapted discretization scheme from \cite{BGHHST,Huc18}, and obtain the spatially discretized equations
 \begin{align}
\dfrac{d\rho_r}{dt}+ \dfrac{q_r-q_l}{L} &= 0, \label{ISO2.1discr} \\
\dfrac{d q_l}{dt}+\dfrac{p_r-p_l}{L} +\rho_r\,g \sin\theta &=-\dfrac{\lambda_{fr}}{2 D} \frac{q_l^2}{\rho_r}, \label{ISO2.2discr}  \\
p_r &=R_s T_0\rho_r z(p_r).  \label{SGEdiscr}
 \end{align}
where $q_r(t):=q(t,L)$, $p_r(t):=p(t,L)$, $\rho_r(t):=\rho(t,L)$ and $q_l(t):=q(t,0)$, $p_l(t):=p(t,0)$. If we represent the pipe as a graph consisting of an edge and two vertices (nodes), define the vertices as the left and right nodes and fix the edge orientation from the left node to the right node, then $q_r(t)$, $p_r(t)$ and $\rho_r(t)$ are defined at the right end of pipe and $q_l(t)$, $p_l(t)$ are defined at the left end of pipe. In general, previously, the pipe is divided into parts of a short length through the introduction of artificial nodes and the specified spatial discretization are performed on each part (subpipe).

Suppose that the functions $q_r$ and $p_l$ are given, that is, we consider the boundary conditions of the form
 \begin{equation}\label{bond_ISO2}
q(t,L)=q_r(t),\quad p(t,0)=p_l(t),\qquad  t\in\mathcal{I}.
 \end{equation}
The functions  $p_r$, $\rho_r$ and $q_l$ need to be found.

We introduce the variable vector $x=(\rho_r,q_l,p_r)^\T$ (we denote it by $x$ for convenience and comparison with further results, since the original variable $x$ is already absent from the equations) and denote
\begin{equation}\label{ModPipeDiscret}
A=\begin{pmatrix} 1&0&0 \\ 0&1&0 \\ 0&0&0 \end{pmatrix},\,
B=\begin{pmatrix} 0& -L^{-1}& 0 \\ g\sin\theta& 0& L^{-1} \\ 0&0&1 \end{pmatrix},\,
f(t,x)=\begin{pmatrix} -L^{-1}q_r(t) \\ L^{-1}p_l(t) -0.5\lambda_{fr} D^{-1}q_l^2\rho_r^{-1} \\ R_s T_0\, \rho_r z(p_r) \end{pmatrix}.
\end{equation}
Then the system \eqref{ISO2.1discr}--\eqref{SGEdiscr} can be written in the vector form
 \begin{equation}\label{DAE_ISO2discr}
\dfrac{d}{dt}[Ax]+Bx=f(t,x),\qquad t\in\mathcal{I},
 \end{equation}
where $A,\, B\in\R^{3\times 3}$ and $f\in C(\mathcal{I}\times \R^3,\R^3)$.
The initial condition for \eqref{DAE_ISO2discr} can be given as
 \begin{equation}\label{ini_ISO2}
x(t_0)=x_0,\qquad x_0=(\rho_r^0,q_l^0,p_r^0)^\T.
 \end{equation}
where $\rho_r^0$ and $p_r^0$ have to satisfy the equation \eqref{SGEdiscr} for $t=t_0$, i.e.,  $p_r^0=R_s T_0\rho_r^0 z(p_r^0)$.

In general, the DAE \eqref{DAE_ISO2discr} is regular (since the pencil $\lambda A+B$ is regular), but if any of the input parameters (i.e., $q_r(t)$ or $p_l(t)$) is not specified, then the system \eqref{ISO2.1discr}--\eqref{SGEdiscr} is underdetermined and the corresponding DAE is singular (nonregular). Also, if it is required to realize the evolution of some variable (i.e., $p_r$, or $\rho_r$, or $q_l$) such that it becomes equal to the prescribed function, then this system is overdetermined and the corresponding DAE is singular.

 \subsection{A model of a gas network (in the isothermal case)}\label{Sect_GasDAEnetwork}

Now, consider a mathematical model of a gas network, where a gas flow in each pipe described by a system of the type  \eqref{ISO2.1}, \eqref{ISO2.2}, \eqref{SGE}. In addition to pipes, the gas network also includes valves, regulators and compressors.

Following \cite{GNM}, \cite{KSSTW22}, we describe a gas network as oriented connected graph $G=(\mV,\mE)$, where $\mV$ denotes a set of nodes (vertices), $\mE$ denotes a set of edges, and each edge joins two distinct nodes (i.e., there are no self-loops). We fix the orientation of edge $e\in\mE$, denoting its endpoints by $v_l$ and $v_r$ and assuming that the edge is oriented from the left node $v_l$ to the right node $v_r$. 
Note that the orientation of the edge may not coincide with the direction of a gas flow.

We collect all nodes with a fixed pressure in $\mV_{pset}$ and refer to them as \emph{pressure nodes} \cite{KSSTW22,BGHHST}. Fixed pressure means the existence of a time-dependent function chosen in advance, which yields the respective pressure value at each point in time.
All other nodes we collect in $\mV_{qset}$. Accordingly, $\mV= \mV_{pset}\cup \mV_{qset}$.

We denote the sets of edges corresponding to the pipes, valves and regulating elements (regulators and compressors) by $\mE_{pip}$, $\mE_{val}$ and $\mE_{reg}$, respectively. Thus, $\mE= \mE_{pip}\cup\mE_{val}\cup \mE_{reg}$.

First, introduce the vector $p$ of the pressures of nodes $u\in\mV_{pset}$, and the vectors $q_{pip,r}$, $q_{pip,l}$, $q_{val}$ and $q_{reg}$ of flows at the right ends of pipes, at the left ends of pipes, through valves and through regulating elements, respectively. % \cite[p.3-7]{KSSTW22}.

As mentioned above, at the pressure nodes $u\in\mV_{pset}$, the pressure function $p^{set}(t)=(\ldots,p^{set}_u(t),\ldots)^\T_{u\in \mV_{pset}}$  is given. %\cite[p.4,5]{KSSTW22}
At the nodes $u\in\mV_{qset}=\mV\setminus \mV_{pset}$ (which include junction, demand and source nodes), the function $q^{set}(t)=(\ldots,q^{set}_u(t),\ldots)^\T_{u\in \mV_{qset}}$, which specifies the relationships between the flows $q_{pip,r}$, $q_{pip,l}$, $q_{val}$ and $q_{reg}$ in a Kirchhoff-type flow balance equation (see \eqref{KSSTW-5} below), is given.

The mathematical model of a gas network consisting of pipes, valves, regulators and compressors after applying spatial discretization (more precisely, a topologically adaptive discretization of the isothermal Euler equations for pipes and pipelines \cite{Huc18}, \cite{BGHHST}) has the form \cite[(9), p.7]{KSSTW22}:
 \begin{align}
A_{pip,r}^\T \dfrac{d}{dt}\phi(p)+D_q(q_{pip,r}-q_{pip,l})&=0, \label{KSSTW-1} \\
\dfrac{d}{dt}q_{pip,l}+D_p(A_{pip,r}^\T+A_{pip,l}^\T)p+f_{pip}(p,q_{pip,l},t)&=0, \label{KSSTW-2}
\\
D_{val}\dfrac{d}{dt}q_{val}+f_{val}(p,q_{val},t)&=0, \label{KSSTW-3}
\\
D_{reg}\dfrac{d}{dt}q_{reg}-f_{reg}(p,q_{reg},t)&=0, \label{KSSTW-4}
\\
A_{pip,l}q_{pip,l}+A_{pip,r}q_{pip,r}+A_{val}q_{val}+A_{reg}q_{reg}&=q^{set}(t), \label{KSSTW-5}
\\
f_{pb}(p)&=0, \label{KSSTW-6}
\\
f_{qb}(q_{pip,l},q_{pip,r},q_{val},q_{reg})&=0, \label{KSSTW-7}
 \end{align}
where $A_{pip,l}:=\big(a_{ij}^{pip,l}\big)_{\begin{subarray}{l} i=1,...,|\mV_{qset}|,\\ j=1,...,|\mE_{pip}|\end{subarray}}$,   $A_{pip,r}:=\big(a_{ij}^{pip,r}\big)_{\begin{subarray}{l} i=1,...,|\mV_{qset}|,\\ j=1,...,|\mE_{pip}|\end{subarray}}$, $A_{val}:=\big(a_{ij}^{val}\big)_{\begin{subarray}{l} i=1,...,|\mV_{qset}|,\\ j=1,...,|\mE_{val}|\end{subarray}}$ and $A_{reg}:=\big(a_{ij}^{reg}\big)_{\begin{subarray}{l} i=1,...,|\mV_{qset}|,\\ j=1,...,|\mE_{reg}|\end{subarray}}$ are constant incidence matrices with the entries
 \begin{align*}
& a_{ij}^{pip,l}= \begin{cases} -1 & \text{if node $i$ is the left node of pipe $j$}, \\  0 & \text{else},  \end{cases} \quad
 a_{ij}^{pip,r}= \begin{cases} 1 & \text{if node $i$ is the right node of pipe $j$}, \\ 0 & \text{else}, \end{cases}  \\
& a_{ij}^{val}= \begin{cases} -1 & \text{if node $i$ is the left node of valve $j$}, \\ 1 & \text{if node $i$ is the right node of valve $j$}, \\ 0 & \text{else},  \end{cases}      \\
& a_{ij}^{reg}= \begin{cases} -1 & \text{if node $i$ is the left node of  regulating element $j$}, \\ 1 & \text{if node $i$ is the right node of regulating element $j$}, \\ 0 & \text{else},  \end{cases}  \\
 \end{align*}
$D_q:=\diag\{...,\frac{\kappa_e}{L_e},...\}_{e\in\mE_{pip}}$,
$D_p:=\diag\{...,\frac{S_e}{L_e},...\}_{e\in\mE_{pip}}$, $D_{val}:=\diag\{...,\mu_e,...\}_{e\in\mE_{val}}$ and $D_{reg}:=\diag\{...,\mu_e,...\}_{e\in\mE_{reg}}$ are constant diagonal matrices, where $\mu_e\ge 0$, $\kappa_e=R_s T_0/S_e$  \,(as above, $T_0=const$ is the temperature and $R_s$ is the specific gas constant), $S_e$ and $L_e$ are the cross-sectional area and the length of pipe $e$, respectively. Here $p$, $q_{pip,r}$, $q_{pip,l}$, $q_{val}$ and $q_{reg}$ are unknown and the remaining functions and parameters are given.  $f_{pip}(p,q_{pip,l},t)$, $f_{val}(p,q_{val},t)$ and $f_{reg}(p,q_{reg},t)$ are functions specified in \cite[(4),(5),(8), p.5,6,7]{KSSTW22}; $f_{pb}(p)$ and $f_{qb}(q_{pip,l},q_{pip,r},q_{val},q_{reg})$ are given continuous functions (see \cite{KSSTW22} for details).

Note that the elements of $\phi(p)=\begin{pmatrix}\vdots \\\varphi(p_u) \\\vdots\end{pmatrix}_{u\in\mV_{qset}}$ from \eqref{KSSTW-1} are expressed as $\varphi(p)=\dfrac{p}{z(p)}$ \cite[p.2,5]{KSSTW22}, $p=p_u$, $u\in \mV_{qset}$, where the function $\varphi(p)$ can be also derived from the equation of state for a real gas (in the isothermal case)  $p=R_s T_0 \rho z(p)$ \eqref{ISGE}, i.e., $\varphi(p)=R_s T_0 \rho$.  Thus, we introduce an additional variable $\varrho=\begin{pmatrix}\vdots \\\rho_u \\\vdots\end{pmatrix}_{u\in \mV_{qset}}$, and instead of \eqref{KSSTW-1}, we use the system
\begin{align*}
 A_{pip,r}^\T \dfrac{d}{dt}\varrho+D_q(q_{pip,r}-q_{pip,l})&=0, %\label{GN_DAE-1-1}
\\ \varrho &=\phi(p), 
\end{align*}
which is equivalent to \eqref{KSSTW-1}, taking into account the coefficient $\kappa_e$, and also rewrite the function $f_{pip}(p,q_{pip,l},t)$ (this function also include $\phi(p)$ \cite[p.~4,~5]{KSSTW22}), without changing its notation, as $f_{pip}(\varrho,q_{pip,l},t)$.

Finally, we obtain the following differential-algebraic system describing the gas network:
\begin{align}
A_{pip,r}^\T \dfrac{d}{dt}\varrho+D_q(q_{pip,r}-q_{pip,l})&=0, \label{GN_DAE-1-1} \\
\dfrac{d}{dt}q_{pip,l}+D_p(A_{pip,r}^\T+A_{pip,l}^\T)p&= -f_{pip}(\varrho,q_{pip,l},t), \label{GN_DAE-2} \\
D_{val}\dfrac{d}{dt}q_{val}&=-f_{val}(p,q_{val},t), \label{GN_DAE-3} \\
D_{reg}\dfrac{d}{dt}q_{reg}&=f_{reg}(p,q_{reg},t), \label{GN_DAE-4} \\
A_{pip,l}q_{pip,l}+A_{val}q_{val}+A_{reg}q_{reg}+A_{pip,r}q_{pip,r}&=q^{set}(t), \label{GN_DAE-5} \\
\varrho&=\phi(p), \label{GN_DAE-1-2} \\
0&=f_{pb}(p), \label{GN_DAE-6} \\
0&=f_{qb}(q_{pip,l},q_{pip,r},q_{val},q_{reg}). \label{GN_DAE-7}
 \end{align}

Denote
\begin{align}
%x=\begin{pmatrix}x_1\\ x_2\\ x_3\\ x_4\\ x_5\\ x_6 \end{pmatrix}
& x=\begin{pmatrix}\varrho\\ q_{pip,l}\\ q_{val}\\ q_{reg}\\ q_{pip,r}\\ p \end{pmatrix}\!,\;
f(t,x)=\begin{pmatrix}
 0        \\
 -f_{pip}(\varrho,q_{pip,l},t),  \\
 -f_{val}(p,q_{val},t)   \\
 f_{reg}(p,q_{reg},t),   \\
 q^{set}(t)\\
 \phi(p)   \\
 f_{pb}(p)   \\
 f_{qb}(q_{pip,l},q_{pip,r},q_{val},q_{reg})  \end{pmatrix}\,,\;
A=\begin{pmatrix}
 A_{pip,r}^\T& 0& 0& 0& 0& 0  \\
 0& I& 0& 0& 0& 0 \\
 0& 0& D_{val}&0 &0 & 0 \\
 0& 0& 0& 0& D_{reg}& 0 \\
 0& 0& 0& 0& 0& 0 \\
 0& 0& 0& 0& 0& 0 \\
 0& 0& 0& 0& 0& 0 \\
 0& 0& 0& 0& 0& 0 \end{pmatrix}\!,  \nonumber \\
& B=\begin{pmatrix}
 0& -D_q& 0& 0& D_q& 0   \\
 0& 0& 0& 0& 0& D_p(A_{pip,r}^\T+A_{pip,l}^\T)  \\
 0& 0& 0&0 &0 & 0    \\
 0& 0& 0& 0& 0& 0    \\
 0& A_{pip,l}& A_{val}& A_{reg}& A_{pip,r}& 0   \\
 I& 0& 0& 0& 0& 0    \\
 0& 0& 0& 0& 0& 0    \\
 0& 0& 0& 0& 0& 0  \end{pmatrix}\!.   \label{GN_DAE}
\end{align}
Then the system \eqref{GN_DAE-1-1}--\eqref{GN_DAE-7} takes the form of the singular (nonregular) DAE (see the definition in Section~\ref{Preliminaries})
 \begin{equation}\label{sGN-DAE} %\label{sDAE}
\frac{d}{dt}[Ax]+Bx(t)=f(t,x),
 \end{equation}
where $A,\, B\in \R^{m\times n}$ and $f\in C(\mathcal{I}\times \R^n,\R^m)$.

The initial condition for the DAE \eqref{sGN-DAE} has the form  %\eqref{sDAE}
 \begin{equation}\label{sGN-ini}
x(0)=x_0,
 \end{equation}
where $x_0=(\varrho^0,q_{pip,l}^0,q_{val}^0,q_{reg}^0,q_{pip,r}^0,p^0)^\T$ is chosen so that the values $t_0$, $x_0$ satisfy all algebraic equations in the system \eqref{GN_DAE-1-1}--\eqref{GN_DAE-7} (or satisfy the consistency condition defined in Remark \ref{RemConsistIni}).

In \cite{KSSTW22}, the vector form of the DAE corresponding to the system \eqref{KSSTW-1}--\eqref{KSSTW-7} is slightly different from the above, but, in general, it is also a nonregular  DAE in the sense that the number of unknowns is not equal to the number of equations. However, in \cite{KSSTW22}, it is mentioned  that with a proper choice of the directions of pipe and some additional conditions to the positions of regulators and valves (as described in, e.g., \cite{Huc18}), the resulting DAE system will have index 1 that means it will be a regular as well.

The model for a gas network in the form of a nonregular DAE of the type \eqref{sGN-DAE} are also obtained in \cite{A-PerJ}.

 \section{Analysis of a singular (nonregular) semilinear DAE with the characteristic pencil of the rank $\rank(\lambda A+B) <n,m$}\label{ExGener}

In this section, we consider a simple example which demonstrates the application of the obtained results.

Consider the singular semilinear DAE (a DAE of the form \eqref{DAE})
 \begin{equation}\label{DAEgenSing}
\frac{d}{dt}[Ax]+Bx=f(t,x),
 \end{equation}
where $t\in [t_+,\infty)$ ($t_+\ge 0$), $x=(x_1,x_2,x_3)^\T\in \R^3$, a function $f(t,x)=(f_1(t,x),f_2(t,x),f_3(t,x))^\T\in C([t_+,\infty)\times \R^3,\R^3)$ has the continuous partial derivative $\dfrac{\partial f}{\partial x}$ on $[t_+,\infty)\times \R^3$, and $A,\, B\in \mathrm{L}(\Rn,\Rm)$ \,(here the terminology from Section \ref{BlockStruct} is used), $n=m=3$,  are the operators to which the matrices
\begin{equation}\label{pencilGener}
A=\begin{pmatrix} 1 & 0 & -1 \\
                  0 & 0 & 0 \\
                  0 & 0 & 0 \end{pmatrix},\quad
B=\begin{pmatrix} 1 & -1 & -1 \\
                  1 & 1  & -1 \\
                  0 & 2 & 0 \end{pmatrix}
\end{equation}
correspond with respect to the standard bases in $\Rn$, $n=3$, and $\Rm$, $m=3$. As usual, a basis in $\R^k$ is standard if the $i$th coordinate of the basis vector $e_j$ ($j=1,...,k$) is equal to $\delta_{ij}$.

The pencil $\lambda A+B$ of the operators \eqref{pencilGener} is singular and its rank equals ${\rank(\lambda A+B)=2}$.

Generally, in this section we consider the matrices corresponding to the operators (from $\R^3$ into $\R^3$) with respect to the standard bases in $\R^3$ (as well as we consider the coordinates of vectors with respect to the standard basis in $\R^3$), and if the bases are different, then this will be explicitly indicated.

The singular pencil \eqref{pencilGener} was studied in \cite[Section~4.4]{Fil.KNU2019}. In \cite{Fil.KNU2019}, it is shown that the subspaces from the decomposition \eqref{Xssrr} where $n=3$, i.e., $\R^3=X_s\dot +X_r =X_{s_1}\dot +  X_{s_2}\dot +  X_1 \dot +X_2$, and from the decomposition \eqref{Yssrr} where $m=3$, i.e., $\R^3=Y_s\dot +Y_r =Y_{s_1}\dot + Y_{s_2}\dot + Y_1\dot +Y_2$, can be represented as
\begin{gather*}
X_s = X_{s_1}\dot +X_{s_2}=Lin\{s_i\}_{i=1}^2,\; {X_{s_1}=Lin\{s_1\}},\; X_{s_2}=Lin\{s_2\},\; X_r = Lin\{p\},\; X_1=\{0\},\; X_2 = X_r,  \\ 
Y_s = Y_{s_1}\dot +  Y_{s_2}= Lin\{l_i\}_{i=1}^2,\; {Y_{s_1} = Lin\{l_1\}},\; Y_{s_2} = Lin\{l_2\},\; Y_r = Lin\{q\},\; Y_1=\{0\},\; Y_2=Y_r, 
\end{gather*}
where
 \begin{equation}\label{BasisVectors}
s_1 =\begin{pmatrix} 1 \\ 0 \\ 0 \end{pmatrix},\,
s_2 =\begin{pmatrix} 1 \\ 0 \\ 1 \end{pmatrix},\,
p =\begin{pmatrix} 0 \\ 1 \\ 0 \end{pmatrix},\;
l_1 =\begin{pmatrix} 1 \\ 0 \\ 0 \end{pmatrix},\,
l_2 =\begin{pmatrix} 0 \\ 1 \\ 0 \end{pmatrix},\;
q =\begin{pmatrix} -1/2 \\ 1/2 \\ 1 \end{pmatrix},
 \end{equation}
and that the projection matrices corresponding to the projectors $S\colon \R^3\to X_s$, $S=S_1+S_2$, $S_i\colon \R^3\to X_{s_i}$, $F\colon \R^3\to Y_s$, $F=F_1+F_2$,  $F_i\colon\R^3\to Y_{s_i}$,  $P\colon \R^3\to X_r$, $P=P_1+P_2$, $P_i\colon \R^3\to X_i$,  $Q\colon \R^3\to Y_r$, $Q=Q_1+Q_2$, $Q_i\colon\R^3\to Y_i$, $i=1,2$, which are defined in \eqref{ProjRS}, \eqref{ProjS} and \eqref{ProjR},  have the form
  % %(with respect to the standard bases in $\R^3$)
 \begin{equation}\label{ProjGenSing}
 \begin{split}
S_1=\begin{pmatrix}
     1  & 0 & -1  \\
     0  & 0 & 0 \\
     0 & 0 & 0  \end{pmatrix}, \,
S_2= \begin{pmatrix}
        0 & 0 & 1  \\
        0 & 0 & 0  \\
        0 & 0 & 1   \end{pmatrix},\,
S=\begin{pmatrix}
     1 & 0 & 0  \\
     0 & 0 & 0 \\
     0 & 0 & 1 \end{pmatrix},\,
P= \begin{pmatrix}
        0 & 0 & 0 \\
        0 & 1 & 0 \\
        0 & 0 & 0  \end{pmatrix},\, P_1=0,\, P_2=P, \\
F_1 =\begin{pmatrix}
     1   & 0 & 1/2 \\
     0   & 0 & 0 \\
     0 & 0 & 0  \end{pmatrix},\,
F_2 =\begin{pmatrix}
     0 & 0 & 0  \\
     0 & 1 & -1/2 \\
     0 & 0 & 0  \end{pmatrix},\,
F = \begin{pmatrix}
     1   & 0 & 1/2 \\
     0   & 1 & -1/2 \\
     0 & 0 & 0  \end{pmatrix},\,
Q = \begin{pmatrix}
     0 & 0 & -1/2 \\
     0 & 0 & 1/2  \\
     0 & 0 & 1    \end{pmatrix},\\
Q_1=0,\, Q_2=Q.
  \end{split}
 \end{equation}

Note that if in $\R^3$, instead of the standard basis, we take the basis which is the union of the bases of the summands from the decomposition \eqref{Xssrr} of the space $\Rn$ where $n=3$, i.e., we take the vectors $s_1$, $s_2$, $p$ defined in \eqref{BasisVectors}, then the matrices corresponding to the projectors $S$, $P$, $S_i$, $P_i$, $i=1,2$, with respect to the basis $s_1$, $s_2$, $p$ in $\R^3$ will have a simple form (and will be self-adjoint):
\begin{equation*}
S_1=\begin{pmatrix}
     1  & 0 & 0  \\
     0  & 0 & 0 \\
     0 & 0 & 0  \end{pmatrix}, \,
S_2= \begin{pmatrix}
        0 & 0 & 0  \\
        0 & 1 & 0  \\
        0 & 0 & 0   \end{pmatrix},\,
S=\begin{pmatrix}
     1 & 0 & 0  \\
     0 & 1 & 0 \\
     0 & 0 & 0 \end{pmatrix},\,
P= \begin{pmatrix}
        0 & 0 & 0 \\
        0 & 0 & 0 \\
        0 & 0 & 1  \end{pmatrix},\,
P_1=0,\, P_2=P,
\end{equation*}
Similarly, if in $\R^3$, instead of the standard basis, we take the basis which is the union of the bases of the summands from the decomposition  \eqref{Yssrr} of the space $\Rm$ where $m=3$, i.e., we take $l_1$, $l_2$, $q$ defined in \eqref{BasisVectors}, then the matrices corresponding to the projectors $F$, $Q$, $F_i$, $Q_i$, $i=1,2$, with respect to the basis $l_1$, $l_2$, $q$ in $\R^3$ will have a simple form (and will be self-adjoint):
\begin{equation*}
F_1 =\begin{pmatrix}
     1  & 0 & 0 \\
     0  & 0 & 0 \\
     0 & 0 & 0  \end{pmatrix},\,
F_2 =\begin{pmatrix}
     0 & 0 & 0  \\
     0 & 1 & 0 \\
     0 & 0 & 0  \end{pmatrix},\,
F = \begin{pmatrix}
     1   & 0 & 0 \\
     0   & 1 & 0 \\
     0 & 0 & 0  \end{pmatrix},\,
Q = \begin{pmatrix}
     0 & 0 & 0 \\
     0 & 0 & 0  \\
     0 & 0 & 1    \end{pmatrix},\,
Q_1=0,\, Q_2=Q.
\end{equation*}
However, in this case, the operators $A,\, B$ as the operators from $\Rn$ into $\Rm$, where $n=m=3$, and in general the DAE \eqref{DAEgenSing}, \eqref{pencilGener}, must be considered with respect to the new bases $s_1$, $s_2$, $p$ in $\Rn$ and $l_1$, $l_2$, $q$ in $\Rm$ ($n=m=3$). In what follows, we continue to use the standard bases in $\R^3$ (in $\Rn$, $\Rm$, $n=m=3$).

As shown in \cite{Fil.KNU2019}, the matrices (with respect to the standard bases in $\R^3$)
\begin{gather*}
\EuScript A_r=0,\; \EuScript A_{gen}=A,\; \EuScript B_{und}=0, \\
\EuScript B_r=\begin{pmatrix}
        0 & -1 & 0 \\
        0 & 1 & 0  \\
        0 & 2 & 0  \end{pmatrix},\;
\EuScript B_{gen}=\begin{pmatrix}
        1 & 0 & -1 \\
        0 & 0 & 0  \\
        0 & 0 & 0  \end{pmatrix},\;
\EuScript B_{ov}= \begin{pmatrix}
        0 & 0 & 0 \\
        1 & 0 & -1  \\
        0 & 0 & 0  \end{pmatrix},\;
\EuScript A_{gen}^{(-1)}= \begin{pmatrix}
               1 & 0 & 1/2 \\
               0 & 0 & 0  \\
               0 & 0 & 0  \end{pmatrix}
\end{gather*}
correspond to the operators $\EuScript A_r$, $\EuScript B_r$, $\EuScript A_{gen}$, $\EuScript B_{gen}$, $\EuScript B_{und}$, $\EuScript B_{ov}$  introduced in \eqref{AsrBsrExtend}, \eqref{ABssExtend} and to the semi-inverse operator $\EuScript A_{gen}^{(-1)}$ defined in Remark~\ref{Rem_InvAgen}. Accordingly,  $A_{gen}=\EuScript A_{gen}\big|_{X_{s_1}}\in \mathrm{L}(X_{s_1},Y_{s_1})$, $B_{gen}=\EuScript B_{gen}\big|_{X_{s_1}} \in \mathrm{L}(X_{s_1},Y_{s_1})$, $B_{und}=\EuScript B_{und}\big|_{X_{s_2}}\in \mathrm{L}(X_{s_2},Y_{s_1})$, $B_{ov}=\EuScript B_{ov}\big|_{X_{s_1}}\in \mathrm{L}(X_{s_1},Y_{s_2})$, $A_r=\EuScript A_r \big|_{X_r}\in \mathrm{L}(X_r,Y_r)$, $B_r=\EuScript B_r\big|_{X_r}\in \mathrm{L}(X_r,Y_r)$ and $A_{gen}^{-1}=\EuScript A_{gen}^{(-1)}\big|_{Y_{s_1}}\in \mathrm{L}(Y_{s_1},X_{s_1})$. Since $A_r=0$ and $B_r$ has the inverse $B_r^{-1}\in \mathrm{L}(Y_r,X_r)$, then $\lambda A_r +B_r$ is a regular pencil of index 1.

The DAE \eqref{DAEgenSing}, \eqref{pencilGener}, is the vector form of the system
 \begin{align}
\frac{d}{dt}(x_1-x_3)+x_1-x_2-x_3 &= f_1(t,x),  \label{genSing1}\\
x_1+x_2-x_3 &= f_2(t,x),  \label{genSing2}\\
2x_2 &=f_3(t,x). \label{genSing3}
 \end{align}

Note that a point $(t,x)$ belongs to the manifold $L_{t_+}$ (introduced in Remark~\ref{RemConsistIni}) if and only if it satisfies the equations \eqref{DAEsysProj3}, \eqref{DAEsysProj4} or the equations equivalent to them, e.g.,  \eqref{DAEsysExtAE1}, \eqref{DAEsysExtAE2}.  It is readily verified that the equations \eqref{DAEsysProj3} and  \eqref{DAEsysProj4} (as well as  \eqref{DAEsysExtAE1}, \eqref{DAEsysExtAE2}) are equivalent to the equations
 \begin{align}
x_2 &=\frac{1}{2}f_3(t,x), \label{genSing3equiv}\\
x_1-x_3 &= f_2(t,x)-\frac{1}{2}f_3(t,x)  \label{genSing2equiv}
 \end{align}
respectively, which are the ``algebraic part'' of the DAE \eqref{DAEgenSing}, \eqref{pencilGener}, and are equivalent to the algebraic equations \eqref{genSing2}, \eqref{genSing3}. 
Also, notice that the ODE \eqref{DAEsysExtDE2} (or \eqref{DAEsysProj2}) is not present in the system \eqref{genSing1}--\eqref{genSing3} since the projector $Q_1=0$, and the ODE \eqref{DAEsysExtDE1} (or \eqref{DAEsysProj1}) is equivalent to
\begin{equation}\label{genSing1equiv}
\frac{d}{dt}(x_1-x_3)=-(x_1-x_3)+ f_1(t,x)+\frac{1}{2}f_3(t,x),
\end{equation}
that is, the equation \eqref{genSing1} into which \eqref{genSing3} (or \eqref{genSing3equiv}) is substituted.

The components (projections) of a vector $x=(x_1,x_2,x_3)^\T\in \R^3$ represented as \eqref{xsr} have the form
\begin{equation*}
x_{s_1} = S_1 x=(x_1-x_3,0,0)^\T,\; x_{s_2}=S_2x=(x_3,0,x_3)^\T,\;
x_{p_1}=P_1 x=0,\; x_{p_2}=P_2 x=(0,x_2,0)^\T,
\end{equation*}
$x_s= x_{s_1} + x_{s_2}$, $x_r=x_{p_2}$, where $S_i$, $P_i$, $i=1,2$, were presented in \eqref{ProjGenSing}. Make the change of variables
\begin{equation}\label{NewVariable}
w=x_1-x_3,\quad \xi=x_3,\quad u=x_2,
\end{equation}
then
$$
x_{s_1}=w\, (1,0,0)^\T,\; x_{s_2}=\xi\, (1,0,1)^\T,\; x_{p_2}=u\, (0,1,0)^\T.
$$
Taking into account the new notations \eqref{NewVariable}, we consider the function
 \begin{equation}\label{tilde_f}
\widetilde{f}(t,w,\xi,u):=f(t,w+\xi,u,\xi)=f(t,x_1,x_2,x_3)=f(t,x)\in C([t_+,\infty)\times \R^3,\R^3),
 \end{equation}
which, obviously, has the continuous partial derivative $\dfrac{\partial f}{\partial (w,\xi,u)}(t,w,\xi,u)$ for all $(t,w,\xi,u)\in [t_+,\infty)\times\R\times\R\times\R$.
In the new notations the system of the equations \eqref{genSing1equiv}, \eqref{genSing3equiv}, \eqref{genSing2equiv} takes the form
 \begin{align}
\frac{d}{dt}w &=-w+ \widetilde{f}_1(t,w,\xi,u)+ \frac{1}{2}\widetilde{f}_3(t,w,\xi,u), \label{genSing1New}\\
u &=\frac{1}{2}\widetilde{f}_3(t,w,\xi,u),  \label{genSing3New}\\
w &= \widetilde{f}_2(t,w,\xi,u)-\frac{1}{2}\widetilde{f}_3(t,w,\xi,u). \label{genSing2New}
 \end{align}

Now we find the conditions under which there exists a global solution of the DAE \eqref{DAEgenSing}, \eqref{pencilGener}, and, accordingly, the system  \eqref{genSing1}--\eqref{genSing3}. To do this, we use Theorems \ref{Th_SingGlobSol} and \ref{Th_SingGlobSolBInv}, and in addition the remarks regarding the functions $\chi$ and $V$ from Section \ref{Sect_chi_V}.

It follows from the mentioned above that condition \ref{SoglSing} of Theorem \ref{Th_SingGlobSol} is satisfied if:
\begin{enumerate}[label={\upshape(\roman*)},ref={\upshape(\roman*)}]%%[(i)]
 \setlength\itemsep{0mm}
\item\label{genSoglSing} There exists a set $\widetilde{D}_{s_2}\subset\R$ such that for any fixed  $t\in [t_+,\infty)$, $w\in\R$ and $\xi\in\widetilde{D}_{s_2}$ there exists a unique $u\in\R$ such that the equalities \eqref{genSing3New}, \eqref{genSing2New} are satisfied.
\end{enumerate}

The matrix corresponding to the operator $\widehat{\Phi}_{t_*,x_*}$ defined (for fixed $t_*$, $x_*$) by \eqref{funcPhiSingExt} has the form
$$
\widehat{\Phi}_{t_*,x_*}=\begin{pmatrix}
   0 & -\dfrac{1}{2}\dfrac{\partial f_3}{\partial x_2}(t_*,x_*)+1 & 0 \\
   0 & \dfrac{1}{2}\dfrac{\partial f_3}{\partial x_2}(t_*,x_*)-1 & 0  \\
   0 & \dfrac{\partial f_3}{\partial x_2}(t_*,x_*)-2 & 0  \end{pmatrix},
$$
and since the equality $\widehat{\Phi}_{t_*,x_*}x_{p_2}=0$, $x_{p_2}\in X_2$, yields $x_{p_2}=0$ if the relation $\frac{\partial f_3}{\partial x_2}(t_*,x_*)-2\ne 0$ holds, then the operator $\Phi_{t_*,x_*}=\widehat{\Phi}_{t_*,x_*}\big|_{X_2}$ \eqref{funcPhiSing} has the inverse $\Phi_{t_*,x_*}^{-1}\in \mathrm{L}(Y_2,X_2)$ if this relation is satisfied. Note that $\frac{\partial f_3}{\partial x_2}(t,x_1,x_2,x_3)=\frac{\partial f_3}{\partial x_2}(t,w+\xi,u,\xi)=\frac{\partial \widetilde{f}_3}{\partial u}(t,w,\xi,u)$.

Thus, condition \ref{InvSing} of Theorem \ref{Th_SingGlobSol} is satisfied if:
\begin{enumerate}[label={\upshape(\roman*)},  ref={\upshape(\roman*)}] %%[(i)]
 \addtocounter{enumi}{1}
 \setlength\itemsep{0mm}
\item\label{genInvSing} For any fixed  $t_*\in [t_+,\infty)$, $w_*\in\R$, $\xi_*\in\widetilde{D}_{s_2}$, $u_*\in\R$ satisfying the equalities \eqref{genSing3New}, \eqref{genSing2New}, the following relation holds:
    $$
    \frac{\partial \widetilde{f}_3}{\partial u}(t_*,w_*,\xi_*,u_*)\ne 2.
    $$
\end{enumerate}

Also, since the space $X_2$ is one-dimensional, then condition \ref{BasInvSing} of Theorem \ref{Th_SingGlobSolBInv} is satisfied if condition \ref{genInvSing} holds (see Remark~\ref{RemarkInv-BasInv} for explanation). Consequently, we can use condition \ref{SoglSing2} of Theorem \ref{Th_SingGlobSolBInv} instead of more restrictive condition \ref{SoglSing} of Theorem \ref{Th_SingGlobSol} and, accordingly, replace condition \ref{genSoglSing} by the following:
\begin{enumerate}[label={\upshape(\roman*)$'$}, ref={\upshape(\roman*)$'$\;}] %%[(i)$'$]
 \setlength\itemsep{0mm}
\item\label{genSoglSing2} There exists a set $\widetilde{D}_{s_2}\subset\R$ such that for any fixed  $t\in [t_+,\infty)$, $w\in\R$ and $\xi\in\widetilde{D}_{s_2}$ there exists $u\in\R$ such that the equalities \eqref{genSing3New}, \eqref{genSing2New} hold.
\end{enumerate}

Recall that $X_1=\{0\}$, the equation \eqref{DAEsysExtDE2} is not present in the system \eqref{genSing1New}--\eqref{genSing3New}, and the equation \eqref{DAEsysExtDE1} is equivalent to \eqref{genSing1New}. Thus, condition \ref{ExtensSing} of Theorem \ref{Th_SingGlobSol} (as well as Theorem \ref{Th_SingGlobSolBInv}) is fulfilled if:
\begin{enumerate}[label={\upshape(\roman*)},  ref={\upshape(\roman*)}]%%[(i)]
 \addtocounter{enumi}{2}
 \setlength\itemsep{0mm}
\item\label{genExtensSing} There exists a number $R>0$ ($R$ can be sufficiently large), a positive definite function  $\widetilde{V}\in C^1([t_+,\infty)\times \widetilde{D}_{s_1},\R)$, where  $\widetilde{D}_{s_1}=\{|w|\ge R\}$, and a function $\chi\in C([t_+,\infty)\times (0,\infty),\R)$ such that:
    \begin{enumerate}[label={\upshape(\alph*)}, ref={\upshape(\alph*)}]
     \setlength\itemsep{0mm}
  \item ${\widetilde{V}(t,w)\to\infty}$ uniformly in $t$ on each finite interval $[a,b)\subset[t_+,\infty)$ as ${|w|\to\infty}$;

  \item for all $t\in [t_+,\infty)$, $w\in\R$, $\xi\in\widetilde{D}_{s_2}$, $u\in\R$ satisfying  \eqref{genSing3New}, \eqref{genSing2New} and ${|w|\ge R}$, the following inequality holds:
      \begin{multline*}
  \widetilde{V}'_{\eqref{genSing1New}}(t,w)= \frac{\partial \widetilde{V}}{\partial t}(t,w)+ \frac{\partial \widetilde{V}}{\partial w}(t,w)\cdot \Big[-w+\widetilde{f}_1(t,w,\xi,u)+ \frac{1}{2}\widetilde{f}_3(t,w,\xi,u)\Big] \le    \\
  \le \chi\big(t,\widetilde{V}(t,w)\big).
  \end{multline*}

  \item the differential inequality $dv/dt\le \chi(t,v)$  ($t\in [t_+,\infty)$) does not have positive solutions with finite escape time.
  \end{enumerate}
\end{enumerate}
Condition \ref{genExtensSing} is given in the most general form, and if we take the function $\widetilde{V}$ of the type \eqref{funcV} and the function $\chi$ of the form  \eqref{La-kU}, then we obtain a particular case of this condition, which is convenient for practical application. Namely, let $\widetilde{V}(t,w)=Hw^2$, where $H=const>0$, $w\in\R$, and $\chi(t,v)=k(t)\,U(v)$, where $k\in C([t_+,\infty),\R)$ and $U\in C(0,\infty)$. Then $\widetilde{V}'_{\eqref{genSing1New}}(t,w)=-2H\, w^2 +2H\, w \big[\widetilde{f}_1(t,w,\xi,u)+ \frac{1}{2}\widetilde{f}_3(t,w,\xi,u)\big]$, and, taking into account the remarks from Section \ref{Sect_chi_V}, condition \ref{genExtensSing} is converted into the following one:
\begin{enumerate}[label={\upshape(\roman*)$'$\;},ref={\upshape(\roman*)$'$\;}] %%[(i)$'$]
 \addtocounter{enumi}{2}
 \setlength\itemsep{0mm}
\item\label{genExtensSing2} There exists a number $R>0$ ($R$ can be sufficiently large) and functions $k\in C([t_+,\infty),\R)$, $U\in C(0,\infty)$ such that $\int\limits_{{\textstyle v}_0}^{\infty}\dfrac{dv}{U(v)} =\infty$ \,($v_0>0$ is some constant) and the inequality
   $$
    -2H\, w^2 +2H\, w \Big[\widetilde{f}_1(t,w,\xi,u)+ \frac{1}{2}\widetilde{f}_3(t,w,\xi,u)\Big]\le k(t)\,U(Hw^2),
   $$
  where $H>0$ is some constant, holds for all $t\in [t_+,\infty)$, $w\in\R$, $\xi\in\widetilde{D}_{s_2}$, ${u\in\R}$ satisfying  \eqref{genSing3New}, \eqref{genSing2New} and ${|w|\ge R}$.
\end{enumerate}

Finally, the following conclusions can be drawn.

\emph{Let conditions \ref{genSoglSing2} and \ref{genInvSing}, where the function $\widetilde{f}(t,w,\xi,u)$ is defined by \eqref{tilde_f}, be fulfilled and let condition \ref{genExtensSing} or \ref{genExtensSing2} hold, then by Theorem \ref{Th_SingGlobSolBInv} (as well as by Theorem \ref{Th_SingGlobSol} if condition  \ref{genSoglSing2} is replaced by \ref{genSoglSing})\, for each initial point $(t_0,x_0)\in [t_+,\infty)\times \R^3$, where $x_0 =(x_{0,1},x_{0,2},x_{0,3})^{\T}$, for which the equalities  \eqref{genSing3equiv}, \eqref{genSing2equiv} hold and $x_{0,3}\in \widetilde{D}_{s_2}$, the initial value problem \eqref{DAEgenSing}, \eqref{pencilGener}, ${x(t_0)=x_0}$ has a unique global solution $x(t)$ with the component $x_{s_2}(t)=S_2x(t)=\varphi_{s_2}(t)\, (1,0,1)^\T$, where ${\varphi_{s_2}\in C([t_0,\infty),\widetilde{D}_{s_2})}$ is some function with the initial value  $\varphi_{s_2}(t_0)=x_{0,3}$.}

\end{document}